\documentclass{amsproc}

\usepackage{physics}
\usepackage{xcolor}
\usepackage{hyperref}
\hypersetup{colorlinks,linkcolor=blue,citecolor=cyan}
\usepackage{tikz, tikz-cd, tikz-3dplot}
\usetikzlibrary{decorations.markings, arrows.meta}
\tikzset{
    partial ellipse/.style args={#1:#2:#3}{
        insert path={+ (#1:#3) arc (#1:#2:#3)}
    }
}
\usepackage{float}
\allowdisplaybreaks

\newtheorem{thm}{Theorem}[section]
\newtheorem{mainthm}{Theorem}[section]

\newtheorem{cor}[thm]{Corollary}
\newtheorem{lem}[thm]{Lemma}
\newtheorem{prop}[thm]{Proposition}

\theoremstyle{definition}
\newtheorem{defn}[thm]{Definition}
\newtheorem{eg}[thm]{Example}

\newtheorem{qn}[thm]{Question}
\newtheorem{rmk}[thm]{Remark}

\numberwithin{equation}{section}

\DeclareMathOperator{\Sk}{\mathrm{Sk}} 
\DeclareMathOperator{\SkAlg}{\mathrm{SkAlg}} 
\newcommand{\qbin}[2]{\begin{bmatrix}{#1}\\ {#2}\end{bmatrix}}

\newcommand\dualY[0]{\rotatebox[origin=c]{180}{$Y$}}

\begin{document}

\title{Skeins, $q$-series, and modularity}

\author{Sunghyuk Park}
\address{Department of Mathematics \& Center of Mathematical Sciences and Applications (CMSA), Harvard University, Cambridge MA 02138}
\email{sunghyukpark@math.harvard.edu}


\date{}

\begin{abstract}
We study BPS $q$-series associated to 3-manifolds decorated by a line defect along an embedded link. 
We prove that these $q$-series depend only on the class of the link in the skein module, thereby defining a homomorphism from the skein module to the space of $q$-series. 
The image of this homomorphism is conjectured to be holomorphically quantum modular, and we verify this conjecture in an explicit example. 
\end{abstract}

\maketitle

\tableofcontents


\section{Introduction}
A notable recent advance in quantum topology is the emergence of new connections between 3-manifolds and (quantum) modular forms. 
The \emph{BPS $q$-series} (a.k.a. $\widehat{Z}$-invariants or homological blocks), introduced by Gukov--Pei--Putrov--Vafa \cite{GukovPeiPutrovVafa} and further developed by Gukov--Manolescu \cite{GukovManolescu}, play a defining role in this development, motivated by the program of categorifying quantum 3-manifold invariants. 

Quantum modularity, in its holomorphic form \cite{GaroufalidisZagier}, is a generalization of modularity exhibited by certain holomorphic functions on the upper and lower half-planes. 
Roughly speaking, we say that a matrix of holomorphic functions on the upper and lower half-planes is a matrix-valued \emph{holomorphic quantum modular form}, if its failure to be modular under $\gamma = \left(\begin{smallmatrix} a & b \\ c & d\end{smallmatrix}\right) \in \mathrm{SL}_2(\mathbb{Z})$ (i.e., the associated cocycle) can be holomorphically extended from $\mathbb{C}\setminus \mathbb{R}$ to the cut plane $\mathbb{C}_\gamma := \mathbb{C} \setminus \frac{1}{c}\mathbb{R}_{\leq -d}$ if $c\neq 0$. 

Guided by existing examples \cite{CCKPS_3d-modularity-revisited} and structural properties of the BPS $q$-series, one expects that for a fixed 3-manifold, the collection of BPS $q$-series obtained by inserting all possible line defects\footnote{For our purposes, these are the so-called ``Wilson line defects,'' which are embedded links each of whose component is colored by some object of $\mathrm{Rep}_q(G)$, where $G$ is the ``gauge group.''
In this article, we will be mostly focused on the case $G = \mathrm{SL}_2$.} should, after choosing an appropriate basis of defects, extend into such a matrix-valued holomorphic quantum modular form. 
At present, however, the BPS $q$-series for $3$-manifolds decorated by line defects along embedded links have been defined mathematically only in rather special settings, including negative definite plumbed 3-manifolds decorated by a very special class of links \cite{CCKPS_3d-modularity-revisited}. 

In this article, we develop this picture for arbitrary embedded links in a much broader class of $3$-manifolds, constructing the corresponding BPS $q$-series and studying their quantum modularity. 
The main results realizing these two aspects are summarized below.

\subsection*{Summary of main results}
Let $K \subset S^3$ be a Seifert-framed knot and $L \subset S^3 \setminus K$ be a framed link in the complement of $K$.\footnote{Throughout this article, $K$ and $L$ play different roles; $L$ is the line defect alluded to in the previous paragraph, while $K$ is used to construct $3$-manifolds, either by taking the complement $S^3 \setminus K$ or by some integral Dehn surgery.} 
We color $K$ by $V_n$, the $n$-dimensional irreducible representation of $U_q(\mathfrak{sl}_2)$, and each component $L_i$ of $L$ by $V_{m_i}$, for some fixed $m_i \geq 1$. 
Let $J_{(K,n),(L,\vec{m})}(q)$ denote the (unnormalized) colored Jones polynomial of the framed, colored link $K \sqcup L$. 
Then we have the following analog of the Melvin--Morton--Rozansky (MMR) expansion for $K$ in the presence of the line defect $L \subset S^3 \setminus K$:
\begin{mainthm}\label{mainthm:MMR}
In the large-color asymptotic where $\hbar \rightarrow 0$, $n\rightarrow \infty$ while $u := n\hbar$ is fixed, 
\[
\frac{1}{[n]} J_{(K,n),(L,\vec{m})}(q) \bigg\vert_{q=e^{\hbar}} \quad\overset{\substack{\hbar \rightarrow 0 \\ n \rightarrow \infty \\ n\hbar = u \text{ fixed}}}{\sim}\quad
\frac{\prod_{i}\chi_{m_i}(x^{\mathrm{lk}(L_i, K)})}{\Delta_K(x)} + \sum_{d\geq 1} \frac{P_d(x)}{\Delta_K(x)^{2d+1}} \hbar^d
,
\]
where $x := e^u$, 
$[n] := \frac{q^{\frac{n}{2}}-q^{-\frac{n}{2}}}{q^{\frac12} - q^{-\frac12}}$, 
$\chi_m(z) := \frac{z^{\frac{m}{2}}-z^{-\frac{m}{2}}}{z^{\frac12}-z^{-\frac12}}$, 
$\Delta_K(x)$ denotes the Alexander polynomial of $K$, and $P_d(x) \in \mathbb{Q}[x^{\pm \frac12}]$ are some polynomials determined by $K$ and $L$.\footnote{
The right-hand side should be understood as a power series in
$\mathbb{Q}[[u]][[\hbar]]$
or in
$x^{\bullet}\mathbb{Q}[x^{\pm 1}][[1-x]][[\hbar]]$, 
where $\bullet \in \frac{1}{2}\mathbb{Z}/\mathbb{Z}$ is $0$ (resp., $\frac{1}{2}$) if $\mathrm{lk}(L,K)$ is even (resp., odd). 
} 
\end{mainthm}
\begin{rmk}
Concretely, the ``large-color asymptotic'' above means the following. 
Consider the $\hbar$-expansion of the colored Jones polynomials
\[
\frac{1}{[n]} J_{(K,n),(L,\vec{m})}(q) \bigg\vert_{q=e^{\hbar}} = \sum_{d\geq 0} a_d(n)\,\hbar^d.
\]
Then, part of the claim is that, for each $d$, the coefficient $a_d(n)$ is actually a polynomial in $n$ of degree $\leq d$ (at least for large enough $n$), 
\[
a_d(n) = \sum_{0\leq j\leq d} a_{d,j} n^j,\quad a_{d,j} \in \mathbb{Q},
\]
so that we can rewrite the above $\hbar$-series into a double series in $\hbar$ and $u := n\hbar$
\[
\sum_{0\leq j\leq d} a_{d,j} n^j \hbar^d = 
\sum_{d' \geq 0} \qty( \sum_{j\geq 0}a_{d'+j,j} u^j ) \hbar^{d'}.
\]
The rest of the claim is that the coefficient of $\hbar^{d'}$ given by the $u$-series
\[
\sum_{j\geq 0}a_{d'+j,j} u^j \in \mathbb{Q}[[u]]
\]
is actually the power series expansion of 
\[
\frac{\prod_{i}\chi_{m_i}(x^{\mathrm{lk}(L_i, K)})}{\Delta_K(x)} \bigg\vert_{x=e^u}
\]
when $d'=0$, and more generally for any $d'$, some rational function in $x$ (uniquely determined by its $u$-expansion) whose denominator is a power of the Alexander polynomial. 
When $L$ is the empty link, this is nothing but the standard MMR expansion \cite{MelvinMorton, BarNatanGaroufalidis, Rozansky}. 
\end{rmk}

Since the colored Jones polynomials $J_{(K,n),(L,\vec{m})}(q)$ depend only on the class of the colored link $(L,\vec{m})$ in the skein module $\Sk^{\mathrm{SL}_2}_q(S^3 \setminus K)$, the same is true for the MMR expansion on the right-hand side of Theorem \ref{mainthm:MMR}. 
Hence, as a corollary, we obtain a linear map
\begin{equation}\label{eqn:MMR-skein}
\mathrm{MMR} : \Sk^{\mathrm{SL}_2}_{q}(S^3 \setminus K) \bigg\vert_{q=e^{\hbar}} \rightarrow \mathbb{Q}[x^{\pm \frac12},\frac{1}{\Delta_K(x)}][[\hbar]].
\end{equation}
When $K$ is braid-homogeneous (i.e., can be presented as a closure of a homogeneous braid), or more generally ``nice'' in the sense of \cite{Park-thesis}, we show that the MMR expansion can be resummed into a power series in $q$ and $x$ with \emph{integer} coefficients in a canonical way: 
\begin{mainthm}\label{mainthm:skein-q-series}
Let $K$ be a nice knot. 
Then, for any colored link $(L,\vec{m}) \subset S^3 \setminus K$, 
the associated MMR expansion given by Theorem~\ref{mainthm:MMR} (i.e., its image under the map $\mathrm{MMR}$ in \eqref{eqn:MMR-skein}), when expanded into a Laurent power series in $x^\frac12$, has coefficients all lying in the image of the natural injective map
\[
\mathbb{Z}[q^{\pm \frac14}] \bigg\vert_{q=e^{\hbar}} \hookrightarrow \mathbb{Q}[[\hbar]].
\]
In other words, the image $\mathrm{MMR}([L,\vec{m}])$ can be resummed uniquely into a two-variable series (the ``BPS $q$-series'')
$\widehat{Z}_{S^3 \setminus K, (L,\vec{m})}(x,q)$ with integer coefficients, which is an isotopy invariant of the link $L$ in $S^3 \setminus K$. 
Moreover, it factors through the skein module to give a $\mathbb{Z}[q^{\pm \frac14}]$-linear map
\[
\widehat{Z} : \Sk^{\mathrm{SL}_2}_q(S^3 \setminus K) \rightarrow \mathbb{Z}[q^{\pm \frac14}]((x^{\frac12})). 
\]
\end{mainthm}
\begin{rmk}[Grading]
The map $\widehat{Z}$ in Theorem \ref{mainthm:skein-q-series} respects the natural grading on the skein module. 
Recall that the $G$-skein module $\Sk^G_q(Y)$ is graded by $H_1(Y; Z(G)^\vee)$\footnote{Here, $Z(G)^\vee$ denotes the Pontryagin dual of the center of $G$.}, so $\Sk^{\mathrm{SL}_2}_q(S^3 \setminus K)$ is graded by $H_1(S^3 \setminus K; \mathbb{Z}/2) \cong \mathbb{Z}/2$; the $0$-(resp., $1$-)graded part is called ``even'' (resp., ``odd''). 
The map $\widehat{Z}$ maps the even part to 
$x^{\frac{1}{2}}\mathbb{Z}[q^{\pm \frac14}]((x))$
and the odd part to 
$\mathbb{Z}[q^{\pm \frac14}]((x))$. 
More generally, for any connected and simply connected complex semisimple Lie group $G$ of rank $r$, there should be a map\footnote{
Here, $N$ is the exponent of $P/Q$ where $Q$ and $P$ are the root and weight lattices, $\rho$ is the Weyl vector, $\omega_i$ is the $i$-th fundamental weight, and $(,)$ is the invariant bilinear form on $\mathfrak{h}_{\mathbb{R}}^*$ normalized so that $(\alpha,\alpha) = 2$ for short roots $\alpha$. 
}
\[
\widehat{Z}^G : \Sk_q^{G}(S^3 \setminus K) \rightarrow \mathbb{Z}[q^{\pm \frac{1}{2N}}] ((x_1^{\frac{1}{N}},\cdots,x_r^{\frac{1}{N}})),
\]
which maps the $\beta \in H_1(S^3 \setminus K; Z(G)^\vee) \cong Z(G)^\vee \cong P/Q$-graded part of the skein module to 
\[
\prod_{1\leq i\leq r}x_i^{(\rho+\beta, \omega_i)}\mathbb{Z}[q^{\pm \frac{1}{2N}}]((x_1, \cdots, x_r)). 
\]
\end{rmk}

Furthermore, we show that the linear map $\widehat{Z}$ constructed in Theorem \ref{mainthm:skein-q-series} is compatible with the Laplace transform surgery formula:\footnote{See Equation \eqref{eqn:Laplace} for the definition of the Laplace transform corresponding to $p$-surgery.}
\begin{mainthm}\label{mainthm:surgery}
Let $p \in \mathbb{Z}$ be a surgery slope for which the Laplace transform on $\widehat{Z}_{S^3 \setminus K}(x,q)$ converges.\footnote{
This assumption means precisely the following:
Consider the power series expansion $\widehat{Z}_{S^3\setminus K}(x,q) = \sum_{n\geq 0} f_{K,n}(q)\,x^{\frac{1}{2}+n}$. 
Then, the Laplace transform converges if $-\frac{1}{p} n^2 + \min \deg_q\,f_{K,n}(q)$ is eventually strictly increasing in $n$. 
}
Then, for any choice of grading $\alpha \in H_1(S_p^3(K);\mathbb{Z}/2)$\footnote{This is a non-trivial choice only when $p$ is even, in which case $\alpha \in \mathbb{Z}/2$.}
and an $\alpha$-twisted spin$^c$ structure $\mathfrak{s} \in \mathrm{Spin}^{c}(S^3_p(K), \alpha)$ (see Section \ref{subsec:twisted-spin-c}),
the map $\widehat{Z}$ in Theorem \ref{mainthm:skein-q-series}, composed with the Laplace transform, factors through $\Sk_q^{\mathrm{SL}_2}(S^3_p(K))_\alpha$ to give a $\mathbb{Z}[q^{\pm \frac14}]$-linear map 
\[
\widehat{Z}_{\mathfrak{s}} : \Sk_q^{\mathrm{SL}_2}(S^3_p(K))_\alpha \rightarrow q^{\Delta_\mathfrak{s}}\mathbb{Z}[q^{\pm \frac14}][[q]],
\]
where $\Delta_\mathfrak{s} \in \mathbb{Q}$. 
\end{mainthm}

Since the skein modules of closed 3-manifolds are finite-dimensional over $\mathbb{Q}(q^{\frac14})$ \cite{GunninghamJordanSafronov}, after tensoring with $\mathbb{Q}(q^{\frac14})$, the image of the map defined in Theorem \ref{mainthm:surgery} is a finite-dimensional vector space of $q$-series. 
We conjecture that this image can be extended to a matrix-valued holomorphic quantum modular form. 
We have verified this in the following example: 
\begin{mainthm}\label{mainthm:quantum-modularity}
For $Y = \pm\Sigma(2,3,7)$, the image of the map defined in Theorem \ref{mainthm:surgery} extends to a matrix-valued holomorphic quantum modular form. 
\end{mainthm}
The matrix-valued cocycle appearing in this holomorphic quantum modularity provides a natural candidate for the Langlands duality map for skein modules \cite{Jordan}; see Section \ref{subsec:Langlands}.

\subsection*{Organization of the paper}
This paper is organized as follows. 
In Section \ref{sec:MMR}, we show Theorem \ref{mainthm:MMR} on the MMR expansion with a line defect. 
Building upon that, in Section \ref{sec:skein-q-series}, we prove Theorem \ref{mainthm:skein-q-series}, constructing a linear map $\widehat{Z}$ from skein modules of complements of nice knots to $(x,q)$-series. 
Then, in Section \ref{sec:surgery}, we show that our map $\widehat{Z}$ is compatible with the Laplace transform surgery formula (Theorem \ref{mainthm:surgery}). 
Finally, in Section \ref{sec:modularity}, we show Theorem \ref{mainthm:quantum-modularity} and 
present some conjectures and speculations on quantum modularity and Langlands duality of skein modules.

\subsection*{Notation}
Here we collect the notation and convention we use throughout the paper. 
\begin{itemize}
\item $[n]$, without any subscript, denotes the ``balanced'' version of ``quantum $n$'' defined by
\[
[n] := \frac{q^{\frac{n}{2}}-q^{-\frac{n}{2}}}{q^{\frac12}-q^{-\frac12}} \in \mathbb{Z}[q^{\pm \frac{1}{2}}]. 
\]
We will also sometimes use ``unbalanced'' versions of quantum integers and quantum binomials, and they will be written with a subscript:
\[
[n]_q := \frac{1-q^n}{1-q},\quad \qbin{n}{k}_q := \frac{[n]_q[n-1]_q \cdots [n-k+1]_q}{[k]_q [k-1]_q \cdots [1]_q} \in \mathbb{Z}[q^{\pm 1}].
\]
\item For each $n\geq 1$, $V_n$ denotes the $n$-dimensional irreducible representation of $U_q(\mathfrak{sl}_2)$. 
Explicitly, it is spanned by basis vectors $\{v_i\}_{0\leq i\leq n-1}$, on which $E,F,K^{\pm 1}$ act by
\[
E v_i = [i] v_{i-1},\quad
F v_i = [n-1-i] v_{i+1},\quad
K^{\pm 1} v_i = q^{\pm \frac{n-1-2i}{2}} v_i.
\]
\item $V_\infty(x)$ denotes the Verma module with highest weight $\lambda_x := \log_q x - 1$. 
Explicitly, it is spanned by basis vectors $\{v_i\}_{i\geq 0}$, on which $E,F,K^{\pm 1}$ act by
\[
E v_i = [i] v_{i-1},\quad
F v_i = [\lambda_x-i] v_{i+1},\quad
K^{\pm 1} v_i = q^{\pm \frac{\lambda_x -2i}{2}} v_i.
\]
\item $J_{K,n}(q)$ denotes the usual framing-dependent, unnormalized colored Jones polynomial where increasing a framing of an $n$-colored component by 1 gives an extra factor of $q^{\frac{n^2-1}{4}}$. 
In particular, when everything is Seifert-framed, the 2-colored (i.e. ordinary) Jones polynomial satisfies the usual skein relations:
\begin{align*}
q J_{K_+}(q) - q^{-1} J_{K_-}(q) &= (q^{\frac12} - q^{-\frac12}) J_{K_=}(q), \\
J_{O}(q) &= [2].
\end{align*}
In our convention, the Hopf link $H$ (in either orientation) with color $(m,n)$ has the colored Jones polynomial
\[
J_{H,m,n}(q) = [mn]. 
\]
While this colored Jones polynomial is insensitive to overall change of orientation, in general, it depends on the orientation of individual components. 
\end{itemize}

\subsection*{Acknowledgements}
I would like to thank Miranda Cheng, Dan Freed, Sergei Gukov, David Jordan, Ahsan Khan, and Campbell Wheeler for inspiring conversations. 
I also thank the organizers of 2024 Richmond Geometry Meeting -- Marco Aldi, Allison Moore, and Nicola Tarasca -- for their hospitality, and for the opportunity to prepare these notes. 

This work was supported in part by Simons Foundation through Simons Collaboration on Global Categorical Symmetries.

\section{Melvin--Morton--Rozansky with skeins}\label{sec:MMR}

\subsection{Brief review of the finite state sum model}
In this subsection, we briefly review how the colored Jones polynomials $\frac{1}{[n]} J_{(K,n),(L,\vec{m})}(q)$ can be realized as a finite state sum, mainly to set up some notation and terminology. 

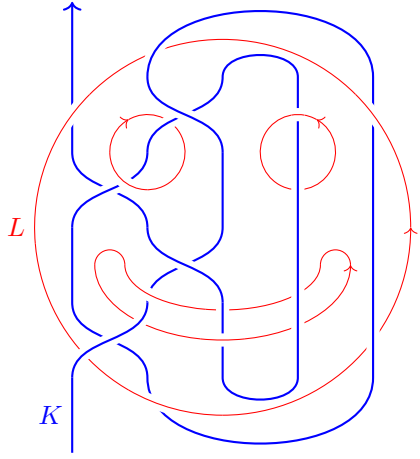
\begin{figure}
\centering
\[
\vcenter{\hbox{
\begin{tikzpicture}
\draw[red, ->] (2, 2) [partial ellipse = 0 : 360 : 2.5];
\draw[red, ->] (1, 3) [partial ellipse = 120 : -240 : 0.5];
\draw[red, ->] (3, 3) [partial ellipse = 60 : 420 : 0.5];
\draw[red, ->] (2, 1.5) [partial ellipse = -180 : 0 : 1.7 and 1.0];
\draw[red] (2, 1.5) [partial ellipse = -180 : 0 : 1.3 and 0.6];
\draw[red] (0.5, 1.5) [partial ellipse = 0 : 180 : 0.2];
\draw[red] (3.5, 1.5) [partial ellipse = 0 : 180 : 0.2];

\draw[white, line width=5] (0, -1) -- (0, 0) to[out=90, in=-90] (1, 1) to[out=90, in=-90] (2, 2) -- (2, 3) to[out=90, in=-90] (1, 4) to[out=90, in=90] (4, 4) -- (4, 0) to[out=-90, in=-90] (1, 0) to[out=90, in=-90] (0, 1) -- (0, 2) to[out=90, in=-90] (1, 3) to[out=90, in=-90] (2, 4) to[out=90, in=90] (3, 4) -- (3, 0) to[out=-90, in=-90] (2, 0) -- (2, 1) to[out=90, in=-90] (1, 2) to[out=90, in=-90] (0, 3) -- (0, 4) -- (0, 5); 
\draw[blue, thick, ->] (0, -1) -- (0, 0) to[out=90, in=-90] (1, 1) to[out=90, in=-90] (2, 2) -- (2, 3) to[out=90, in=-90] (1, 4) to[out=90, in=90] (4, 4) -- (4, 0) to[out=-90, in=-90] (1, 0) to[out=90, in=-90] (0, 1) -- (0, 2) to[out=90, in=-90] (1, 3) to[out=90, in=-90] (2, 4) to[out=90, in=90] (3, 4) -- (3, 0) to[out=-90, in=-90] (2, 0) -- (2, 1) to[out=90, in=-90] (1, 2) to[out=90, in=-90] (0, 3) -- (0, 4) -- (0, 5); 
\draw[white, line width=5] (0, 0) to[out=90, in=-90] (1, 1);
\draw[blue, thick] (0, 0) to[out=90, in=-90] (1, 1);
\draw[white, line width=5] (0, 2) to[out=90, in=-90] (1, 3);
\draw[blue, thick] (0, 2) to[out=90, in=-90] (1, 3);
\draw[white, line width=5] (2, 1) to[out=90, in=-90] (1, 2);
\draw[blue, thick] (2, 1) to[out=90, in=-90] (1, 2);
\draw[white, line width=5] (2, 3) to[out=90, in=-90] (1, 4);
\draw[blue, thick] (2, 3) to[out=90, in=-90] (1, 4);

\draw[white, line width=5] (2, 2) [partial ellipse = 0 : 45 : 2.5];
\draw[red] (2, 2) [partial ellipse = 0 : 45 : 2.5];
\draw[white, line width=5] (2, 2) [partial ellipse = 135 : 185 : 2.5];
\draw[red] (2, 2) [partial ellipse = 135 : 185 : 2.5];
\draw[white, line width=5] (2, 2) [partial ellipse = -120 : -90 : 2.5];
\draw[red] (2, 2) [partial ellipse = -120 : -90 : 2.5];
\draw[white, line width=5] (1, 3) [partial ellipse = -180 : -90 : 0.5];
\draw[red] (1, 3) [partial ellipse = -180 : -90 : 0.5];
\draw[white, line width=5] (3, 3) [partial ellipse = 80 : 100 : 0.5];
\draw[red] (3, 3) [partial ellipse = 80 : 100 : 0.5];
\draw[white, line width=5] (2, 1.5) [partial ellipse = -100 : -80 : 1.7 and 1.0];
\draw[red] (2, 1.5) [partial ellipse = -100 : -80 : 1.7 and 1.0];
\draw[white, line width=5] (2, 1.5) [partial ellipse = -180 : -100 : 1.3 and 0.6];
\draw[red] (2, 1.5) [partial ellipse = -180 : -100 : 1.3 and 0.6];

\node[red, left] at (-0.5, 2.0){$L$};
\node[blue, left] at (0, -0.5){$K$};
\end{tikzpicture}
}}
\]
\caption{$K \sqcup L$ as a colored $(1, 1)$-tangle}
\label{fig:(1,1)-tangle}
\end{figure}

Pick a base point on $K$ and place it at the point at infinity in $S^3$ to draw the link $K \sqcup L$ as a $(1,1)$-tangle diagram $D$; see Figure \ref{fig:(1,1)-tangle} for an example. 
The image of this colored tangle under the Reshetikhin-Turaev functor is
\[
\frac{1}{[n]} J_{(K,n),(L,\vec{m})}(q) \cdot \mathbf{1}_{V_n}.
\]
Hence, $\frac{1}{[n]} J_{(K,n),(L,\vec{m})}(q)$ is equal to the standard state sum on $D$, with the state on the external arcs fixed, say, to the highest weight vector $v_0 \in V_n$. 
To describe this explicitly, let us set up some notations and terminologies: 
\begin{itemize}
\item An \emph{arc} is an edge of the tangle diagram, viewed as a planar graph. 
An arc is called \emph{external} if it corresponds to one of the open ends of the tangle, and \emph{internal} if it is not. 
Let $\mathcal{A}$ denote the set of all internal arcs. 
We will also write $\mathcal{A} = \mathcal{A}_L \sqcup \mathcal{A}_K$, where $\mathcal{A}_L$ denotes the set of arcs of $L$, and $\mathcal{A}_K$ the set of internal arcs of $K$. 
\item A \emph{state} $s$ is a map $s : \mathcal{A} \rightarrow \mathbb{N}$. 
For each arc $a \in \mathcal{A}$, we call $s(a) \in \mathbb{N}$ the \emph{spin} on $a$. 
A state $s$ is called \emph{admissible} if, whenever $a \in \mathcal{A}$ is an arc colored by $V_{k}$, the corresponding spin is $0\leq s(a) \leq k-1$, so that it determines a vector $v_{s(a)} \in V_k$. 
Let $\mathcal{S}$ denote the set of all admissible states. 
\item For each admissible state $s \in \mathcal{S}$, let $\langle D_{(K,V_n), \{(L_i, V_{m_i})\}}; s\rangle$ denote the \emph{weight} of the state $s$, given by the product of the corresponding finite-dimensional $R$-matrices, cups and caps; these are summarized in the fourth row of Table \ref{tab:R-matrices} and the last column of Table \ref{tab:cupsandcaps}.  
\end{itemize}
In summary, the colored Jones polynomial can be computed as the state sum:
\[
\frac{1}{[n]} J_{(K,n),(L,\vec{m})}(q) 
= 
\sum_{s} \langle D_{(K,V_n), \{(L_i, V_{m_i})\}}; s\rangle. 
\]
This is a finite sum, since each component of the link $K \sqcup L$ is colored by some finite-dimensional representation. 
In the next subsection, we will consider an extension of this state sum model, where $K$ is colored by a Verma module $V_\infty(x)$ instead. 

\begin{table}
\centering
\begin{tabular}{|c|c||c|c|}
\hline
$
\vcenter{\hbox{
\begin{tikzpicture}[scale=0.7]
\draw[blue, ->, thick] (1, 0) -- (0, 1);
\draw[white, line width=5] (0, 0) -- (1, 1);
\draw[blue, ->, thick] (0, 0) -- (1, 1);
\node[blue, left] at (0, 0){$i$};
\node[blue, right] at (1, 0){$j$};
\node[blue, left] at (0, 1){$i'$};
\node[blue, right] at (1, 1){$j'$};
\node[blue, below] at (-0.2, -0.3){\tiny $V_\infty(x)$};
\node[blue, below] at (1.2, -0.3){\tiny $V_\infty(x)$};
\end{tikzpicture}
}}
$
 & $q^{\frac14} R(x,x,q)_{i,j}^{i',j'}$ 
&
$
\vcenter{\hbox{
\begin{tikzpicture}[scale=0.7]
\draw[blue, ->, thick] (0, 0) -- (1, 1);
\draw[white, line width=5] (1, 0) -- (0, 1);
\draw[blue, ->, thick] (1, 0) -- (0, 1);
\node[blue, left] at (0, 0){$i$};
\node[blue, right] at (1, 0){$j$};
\node[blue, left] at (0, 1){$i'$};
\node[blue, right] at (1, 1){$j'$};
\node[blue, below] at (-0.2, -0.3){\tiny $V_\infty(x)$};
\node[blue, below] at (1.2, -0.3){\tiny $V_\infty(x)$};
\end{tikzpicture}
}}
$
 & $q^{-\frac14} R^{-1}(x,x,q)_{i,j}^{i',j'}$ 
\\
\hline
$
\vcenter{\hbox{
\begin{tikzpicture}[scale=0.7]
\draw[red, ->] (1, 0) -- (0, 1);
\draw[white, line width=5] (0, 0) -- (1, 1);
\draw[blue, ->, thick] (0, 0) -- (1, 1);
\node[blue, left] at (0, 0){$i$};
\node[red, right] at (1, 0){$j$};
\node[red, left] at (0, 1){$i'$};
\node[blue, right] at (1, 1){$j'$};
\node[blue, below] at (-0.2, -0.3){\tiny $V_\infty(x)$};
\node[red, below] at (1.2, -0.3){\tiny $V_m$};
\end{tikzpicture}
}}
$
 & $x^{\frac{m}{4}} R(x,q^m,q)_{i,j}^{i',j'}$ 
&
$
\vcenter{\hbox{
\begin{tikzpicture}[scale=0.7]
\draw[blue, ->, thick] (0, 0) -- (1, 1);
\draw[white, line width=5] (1, 0) -- (0, 1);
\draw[red, ->] (1, 0) -- (0, 1);
\node[blue, left] at (0, 0){$i$};
\node[red, right] at (1, 0){$j$};
\node[red, left] at (0, 1){$i'$};
\node[blue, right] at (1, 1){$j'$};
\node[blue, below] at (-0.2, -0.3){\tiny $V_\infty(x)$};
\node[red, below] at (1.2, -0.3){\tiny $V_m$};
\end{tikzpicture}
}}
$
 & $x^{-\frac{m}{4}} R^{-1}(x,q^m,q)_{i,j}^{i',j'}$ 
 \\
\hline
$
\vcenter{\hbox{
\begin{tikzpicture}[scale=0.7]
\draw[blue, ->, thick] (1, 0) -- (0, 1);
\draw[white, line width=5] (0, 0) -- (1, 1);
\draw[red, ->] (0, 0) -- (1, 1);
\node[red, left] at (0, 0){$i$};
\node[blue, right] at (1, 0){$j$};
\node[blue, left] at (0, 1){$i'$};
\node[red, right] at (1, 1){$j'$};
\node[red, below] at (-0.2, -0.3){\tiny $V_m$};
\node[blue, below] at (1.2, -0.3){\tiny $V_\infty(x)$};
\end{tikzpicture}
}}
$
 & $x^{\frac{m}{4}} R(q^m,x,q)_{i,j}^{i',j'}$ 
&
$
\vcenter{\hbox{
\begin{tikzpicture}[scale=0.7]
\draw[red, ->] (0, 0) -- (1, 1);
\draw[white, line width=5] (1, 0) -- (0, 1);
\draw[blue, ->, thick] (1, 0) -- (0, 1);
\node[red, left] at (0, 0){$i$};
\node[blue, right] at (1, 0){$j$};
\node[blue, left] at (0, 1){$i'$};
\node[red, right] at (1, 1){$j'$};
\node[red, below] at (-0.2, -0.3){\tiny $V_m$};
\node[blue, below] at (1.2, -0.3){\tiny $V_\infty(x)$};
\end{tikzpicture}
}}
$
 & $x^{-\frac{m}{4}} R^{-1}(q^m,x,q)_{i,j}^{i',j'}$
 \\
 \hline
$
\vcenter{\hbox{
\begin{tikzpicture}[scale=0.7]
\draw[red, ->] (1, 0) -- (0, 1);
\draw[white, line width=5] (0, 0) -- (1, 1);
\draw[red, ->] (0, 0) -- (1, 1);
\node[red, left] at (0, 0){$i$};
\node[red, right] at (1, 0){$j$};
\node[red, left] at (0, 1){$i'$};
\node[red, right] at (1, 1){$j'$};
\node[red, below] at (-0.2, -0.3){\tiny $V_{m_1}$};
\node[red, below] at (1.2, -0.3){\tiny $V_{m_2}$};
\end{tikzpicture}
}}
$
 & $q^{\frac{m_1 m_2}{4}} R(q^{m_1},q^{m_2},q)_{i,j}^{i',j'}$ 
&
$
\vcenter{\hbox{
\begin{tikzpicture}[scale=0.7]
\draw[red, ->] (0, 0) -- (1, 1);
\draw[white, line width=5] (1, 0) -- (0, 1);
\draw[red, ->] (1, 0) -- (0, 1);
\node[red, left] at (0, 0){$i$};
\node[red, right] at (1, 0){$j$};
\node[red, left] at (0, 1){$i'$};
\node[red, right] at (1, 1){$j'$};
\node[red, below] at (-0.2, -0.3){\tiny $V_{m_1}$};
\node[red, below] at (1.2, -0.3){\tiny $V_{m_2}$};
\end{tikzpicture}
}}
$
 & $q^{-\frac{m_1 m_2}{4}} R^{-1}(q^{m_1},q^{m_2},q)_{i,j}^{i',j'}$\\
 \hline
\multicolumn{4}{|c|}{
{$
\begin{aligned}
R(x_1, x_2, q)_{i,j}^{i', j'} &:= \delta_{i+j, i'+j'} x_1^{-\frac{i-j'}{4} - \frac{j}{2} -\frac{1}{4}} x_2^{\frac{i-j'}{4} - \frac{j'}{2} - \frac{1}{4}} q^{(j+\frac{1}{2})(j'+\frac{1}{2})} \qbin{i}{j'}_q 
(q^{j+1}x_2^{-1};q)_{i-j'}, \\
R^{-1}(x_1,x_2,q)_{i,j}^{i',j'} &:= R(x_2^{-1},x_1^{-1},q^{-1})_{j,i}^{j',i'}.
\end{aligned}
$}
} \\
\hline
\end{tabular}
\caption{$R$-matrices; note that the exponential prefactor $e^{\frac{\log x_1 \log x_2}{4\log q}}$ is omitted from $R(x_1,x_2,q)$.}
\label{tab:R-matrices}
\end{table}

\begin{table}
\centering
\begin{tabular}{|c|c||c|c|}
\hline
$
\vcenter{\hbox{
\begin{tikzpicture}[scale=0.7]
\draw[blue, ->, thick] (0, 0) to[out=90, in=90] (1, 0);
\node[blue, left] at (0, 0){$i$};
\node[blue, right] at (1, 0){$\;$};
\node[blue, below] at (0.5, -0.1){\tiny $V_{\infty}(x)$};
\end{tikzpicture}
}}
$
& $x^{\frac14} q^{-\frac14 -\frac{i}{2}}$
& 
$
\vcenter{\hbox{
\begin{tikzpicture}[scale=0.7]
\draw[red, ->] (0, 0) to[out=90, in=90] (1, 0);
\node[red, left] at (0, 0){$i$};
\node[red, right] at (1, 0){$\;$};
\node[red, below] at (0.5, -0.1){\tiny $V_{m}$};
\end{tikzpicture}
}}
$
& $q^{\frac{m-1}{4} -\frac{i}{2}}$\\
\hline
$
\vcenter{\hbox{
\begin{tikzpicture}[scale=0.7]
\draw[blue, ->, thick] (1, 0) to[out=-90, in=-90] (0, 0);
\node[blue, left] at (0, 0){$\;$};
\node[blue, right] at (1, 0){$i$};
\node[blue, below] at (0.5, -0.3){\tiny $V_{\infty}(x)$};
\end{tikzpicture}
}}
$ & $x^{\frac14} q^{-\frac14 -\frac{i}{2}}$
&
$
\vcenter{\hbox{
\begin{tikzpicture}[scale=0.7]
\draw[red, ->] (1, 0) to[out=-90, in=-90] (0, 0);
\node[red, left] at (0, 0){$\;$};
\node[red, right] at (1, 0){$i$};
\node[red, below] at (0.5, -0.3){\tiny $V_{m}$};
\end{tikzpicture}
}}
$ & $q^{\frac{m-1}{4} -\frac{i}{2}}$
\\
\hline
$
\vcenter{\hbox{
\begin{tikzpicture}[scale=0.7]
\draw[blue, <-, thick] (0, 0) to[out=90, in=90] (1, 0);
\node[blue, left] at (0, 0){$\;$};
\node[blue, right] at (1, 0){$i$};
\node[blue, below] at (0.5, -0.1){\tiny $V_{\infty}(x)$};
\end{tikzpicture}
}}
$ & $x^{-\frac14} q^{\frac14 +\frac{i}{2}}$
&
$
\vcenter{\hbox{
\begin{tikzpicture}[scale=0.7]
\draw[red, <-] (0, 0) to[out=90, in=90] (1, 0);
\node[red, left] at (0, 0){$\;$};
\node[red, right] at (1, 0){$i$};
\node[red, below] at (0.5, -0.1){\tiny $V_{m}$};
\end{tikzpicture}
}}
$ & $q^{-\frac{m-1}{4} +\frac{i}{2}}$
\\
\hline
$
\vcenter{\hbox{
\begin{tikzpicture}[scale=0.7]
\draw[blue, <-, thick] (1, 0) to[out=-90, in=-90] (0, 0);
\node[blue, left] at (0, 0){$i$};
\node[blue, right] at (1, 0){$\;$};
\node[blue, below] at (0.5, -0.3){\tiny $V_{\infty}(x)$};
\end{tikzpicture}
}}
$ & $x^{-\frac14} q^{\frac14 +\frac{i}{2}}$
&
$
\vcenter{\hbox{
\begin{tikzpicture}[scale=0.7]
\draw[red, <-] (1, 0) to[out=-90, in=-90] (0, 0);
\node[red, left] at (0, 0){$i$};
\node[red, right] at (1, 0){$\;$};
\node[red, below] at (0.5, -0.3){\tiny $V_{m}$};
\end{tikzpicture}
}}
$ & $q^{-\frac{m-1}{4} +\frac{i}{2}}$
\\
\hline
\end{tabular}
\caption{Cups and caps}
\label{tab:cupsandcaps}
\end{table}

\subsection{Proof of Theorem \ref{mainthm:MMR}}
Since the finite-dimensional $R$-matrices agree with truncations of the infinite-dimensional $R$-matrix for the Verma module, the above finite state sum can be viewed as a truncation of the infinite state sum where $K$ is colored by the Verma module $V_\infty(x=q^n)$ with highest weight $\lambda_x = \log_q x -1  = n-1$: 
\[
\sum_{s} \langle D_{(K,V_n), \{(L_i, V_{m_i})\}}; s\rangle 
= 
\sum_{\substack{s \\ \max s\vert_{K} < n}} \langle D_{(K,V_\infty(x)), \{(L_i, V_{m_i})\}}; s\rangle \bigg\vert_{x=q^n}. 
\]
The weights for the crossings, cups and caps are summarized in Tables \ref{tab:R-matrices} and \ref{tab:cupsandcaps}. 

In the following, we use the notation $(1-x, \hbar)$ for the ideal generated by $1-x$ and $\hbar$ in the two-variable power series ring $\mathbb{Q}[x^{\pm 1}][[1-x, \hbar]]$, and use the big O notation $O((1-x, \hbar)^k)$ for any element in the $k$-th power $(1-x, \hbar)^k$. 

\begin{lem}\label{lem:truncation-bound}
In the infinite state sum where $K$ is colored by the Verma module $V_\infty(x)$, 
any state containing a spin $\geq n$ on $K$ is of $O((1-x, \hbar)^{n})$. 
That is, 
\[
\sum_{s} \langle D_{(K,V_\infty(x)), \{(L_i, V_{m_i})\}}; s\rangle = 
\sum_{\substack{s \\ \max s\vert_K < n}} \langle D_{(K,V_\infty(x)), \{(L_i, V_{m_i})\}}; s\rangle 
+ O((1-x, \hbar)^{n}),
\]
as formal power series in $x^{\bullet}\mathbb{Q}[x^{\pm 1}][[1-x, \hbar]]$, for some $\bullet \in \frac12\mathbb{Z}/\mathbb{Z}$. 
\end{lem}
\begin{proof}[Proof of Lemma \ref{lem:truncation-bound}]
Recall that the spins assigned to the external arcs of $K$ are fixed to be $0$. 
Hence, if some arc of $K$ is assigned a spin of $n$ or higher, then there must be at least $n$ total jumps in the spins along $K$. 
Since
\[
(x;q)_k \big \vert_{q=e^{\hbar}} = O((1-x, \hbar)^k),
\]
we have
\[
x^{\frac{m}{4}}R(q^m,x,q)_{i,j}^{i',j'} = O((1-x, \hbar)^{i-j'})
\quad\text{and}\quad
x^{-\frac{m}{4}}R^{-1}(x,q^m,q)_{i,j}^{i',j'} = O((1-x, \hbar)^{j-i'}),
\]
this means that the weight associated to such a state is of $O((1-x, \hbar)^n)$. 

Moreover, once we fix the power of $(1-x, \hbar)$, it bounds the total number of jumps in the spins, and hence only finitely many states contribute. 
It immediately follows that this infinite state sum gives rise to a power series 
\[
\sum_{s} \langle D_{(K,V_\infty(x)), \{(L_i, V_{m_i})\}}; s\rangle
\bigg\vert_{q=e^\hbar}
\in x^{\bullet}\mathbb{Q}[x^{\pm 1}][[1-x]][[\hbar]], 
\]
for some $\bullet \in \frac12 \mathbb{Z}/\mathbb{Z}$. 
\end{proof}
Thanks to this lemma, the part of the infinite state sum that is truncated out to get the finite state sum, regarded as a formal power series in $\mathbb{Q}[x^{\pm \frac12}][[1-x, \hbar]]$, vanishes\footnote{That is, escapes to infinite power.} in the $n\rightarrow \infty$ limit. 
That is, in the asymptotic limit where $n \rightarrow \infty$ while $e^{n\hbar} = x$ is fixed, it suffices to consider the (untruncated) infinite state sum where $K$ is colored by $V_\infty(x)$:
\[
\frac{1}{[n]} J_{(K,n),(L,\vec{m})}(q) \bigg\vert_{q=e^{\hbar}} 
\;\overset{\substack{\hbar \rightarrow 0 \\ n \rightarrow \infty \\ n\hbar = u \text{ fixed}}}{\sim}\;
\sum_{s} \langle D_{(K,V_\infty(x)), \{(L_i, V_{m_i})\}}; s\rangle
\bigg\vert_{q=e^\hbar}. 
\]

Now, it remains to show that this power series is of the desired form, in terms of rational functions in $x^{\frac{1}{2}}$, with denominators given by powers of the Alexander polynomial $\Delta_K(x)$. 

For this purpose, we first express the state sum as sum over states on $L$ and those on $K$:
\[
\sum_{s} \langle D_{(K,V_\infty(x)), \{(L_i, V_{m_i})\}}; s\rangle \bigg\vert_{q=e^{\hbar}}
= 
\sum_{s\vert_L} 
\qty(
\sum_{s\vert_K}
\langle D_{(K,V_\infty(x)), \{(L_i, V_{m_i})\}}; s\rangle \bigg\vert_{q=e^{\hbar}}
)
. 
\]
Note, the first sum over $s\vert_L : \mathcal{A}_L \rightarrow \mathbb{N}$ is a finite sum, since there are only finitely many possible admissible states on $L$. 
Thus, Theorem \ref{mainthm:MMR} will follow from the following two main lemmas:
\begin{lem}\label{lem:perturbed-random-walk}
For each fixed $s\vert_L$, the state sum over $K$ has an expansion of the form
\[
\sum_{s\vert_K}
\langle D_{(K,V_\infty(x)), \{(L_i, V_{m_i})\}}; s\rangle \bigg\vert_{q=e^{\hbar}} 
= 
\sum_{d\geq 0} \frac{P_d^{s\vert_L}(x)}{\Delta_K(x)^{2d+1}} \hbar^d,
\]
for some Laurent polynomials $P_d^{s\vert_L}(x) \in \mathbb{Q}[x^{\pm \frac12}]$. 
\end{lem}

\begin{lem}\label{lem:leading-term}
The leading polynomial $P_0^{s\vert_L}(x)$ vanishes unless the spins on arcs of $L$ are constant along each strand, in which case $P_0^{s\vert_L}(x)$ is the monomial
\[
\prod_{i} x^{\mathrm{lk}(L_i,K) (\frac{m_i - 1}{2} - s(L_i))},
\]
where $s(L_i)$ denotes the spin on any arc of $L_i$. 
In particular, 
\[
\sum_{s\vert_L} P_0^{s\vert_L}(x) = \prod_{i} \chi_{m_i}(x^{\mathrm{lk}(L_i,K)}).
\]
\end{lem}

We will prove the second lemma first, since it is easier. 
\begin{proof}[Proof of Lemma \ref{lem:leading-term}]
Observe that\footnote{Here, $O(\hbar^1)$ denotes terms of $\hbar$-degree $1$ or higher.}
\begin{align*}
x^{\frac{m}{4}} R(x,q^m,q)_{i,j}^{i',j'} &= 
\begin{cases}
x^{\frac{m-1-2j}{4}} + O(\hbar^1) &\text{if }i-j' = 0 \\
O(\hbar^1) &\text{if }i-j' > 0 \\
\end{cases}, \\
x^{-\frac{m}{4}} R^{-1}(q^m,x,q)_{i,j}^{i',j'} &= 
\begin{cases}
x^{-\frac{m-1-2i}{4}} + O(\hbar^1) &\text{if }j-i' = 0 \\
O(\hbar^1) &\text{if }j-i' > 0 \\
\end{cases}, \\
q^{\frac{m_1 m_2}{4}}R(q^{m_1}, q^{m_2},q)_{i,j}^{i',j'} &=
\begin{cases}
1 + O(\hbar^1) &\text{if }i-j'=0 \\
O(\hbar^1) &\text{if }i-j' > 0
\end{cases}, \\
q^{-\frac{m_1 m_2}{4}} R^{-1}(q^{m_1},q^{m_2},q)_{i,j}^{i',j'} &= 
\begin{cases}
1 + O(\hbar^1) &\text{if }j-i' = 0 \\
O(\hbar^1) &\text{if }j-i' > 0 \\
\end{cases}.
\end{align*}
It follows that the only states that contribute to the $\hbar^0$-term are those for which the spins on arcs of $L$ are constant, say $s(L_i) \in \{0, 1, \cdots, m_i-1\}$, along each component $L_i$. 
If $s\vert_L$ is such a state on $L$, and
$c_{+}(L_i,K)$ and $c_{-}(L_i,K)$ denote the number of positive and negative crossings between $L_i$ and $K$, 
then the $\hbar^0$-term of the state sum over $K$ is given by the product of two factors, the first of which is
\[
\prod_{i} (x^{\frac{m_i-1-2s(L_i)}{4}})^{c_{+}(L_i,K)-c_{-}(L_i,K)} = \prod_{i} x^{\mathrm{lk}(L_i,K) (\frac{m_i - 1}{2} - s(L_i))},
\]
which is the weight associated to the crossings involving any strand of $L$, 
and the second factor is the state sum purely on $K$, whose $\hbar^0$-term is given by the leading term $\frac{1}{\Delta_K(x)}$ of the standard MMR expansion. 
That is, 
\[
\sum_{s\vert_K}
\langle D_{(K,V_\infty(x)), \{(L_i, V_{m_i})\}}; s\rangle \bigg\vert_{q=e^{\hbar}} 
= 
\frac{\prod_{i} x^{\mathrm{lk}(L_i,K) (\frac{m_i - 1}{2} - s(L_i))}}{\Delta_K(x)} + O(\hbar^1),
\]
and therefore
\[
\sum_{s\vert_L}
\sum_{s\vert_K}
\langle D_{(K,V_\infty(x)), \{(L_i, V_{m_i})\}}; s\rangle \bigg\vert_{q=e^{\hbar}} 
= 
\frac{\prod_{i} \chi_{m_i}(x^{\mathrm{lk}(L_i,K)})}{\Delta_K(x)} + O(\hbar^1).
\]
\end{proof}

Now it remains to prove the first lemma (Lemma \ref{lem:perturbed-random-walk}), which we do in the next subsection.

\subsection{Proof of Lemma \ref{lem:perturbed-random-walk}}
The main idea is to mimic Rozansky's proof \cite{Rozansky} of the MMR conjecture. 
For this purpose, we will first consider the ``free state sum,'' a simpler state sum where we replace all the $R$-matrices involving a strand of $K$ by the \emph{parametrized $\mathsf{R}$-matrices}, which are summarized in Table \ref{tab:parametrized-R-matrices}. 
We use different formal parameters $\alpha_c,\beta_c,\gamma_c$ for each crossing $c$, so there would be a total of $3N$ such parameters, if $N$ is the number of crossings involving a strand of $K$.\footnote{
For crossings $c$ involving both $K$ and $L$, the parameters $\alpha_c$ and $\gamma_c$ will not play any role, so we could have just used the remaining $3N_{KK} + N_{KL}$ parameters, where $N_{KK}$ (resp. $N_{KL}$) is the number of crossings between strands of $K$ (resp. between $K$ and $L$), but we are writing it in a uniform way for simplicity of notation.
} 
We will denote the set of all those parameters by symbols $\{\alpha\}, \{\beta\}, \{\gamma\}$. 
To this end, we define the \emph{free state sum} to be
\[
\sum_{s\vert_K}
\langle D_{(K,V_\infty(x)), \{(L_i, V_{m_i})\}}; \{\alpha\}, \{\beta\}, \{\gamma\}; s\rangle^{\mathrm{free}},
\]
where 
$\langle D_{(K,V_\infty(x)), \{(L_i, V_{m_i})\}}; \{\alpha\}, \{\beta\}, \{\gamma\}; s\rangle^{\mathrm{free}}$ 
denotes the weight of the state $s$ in this free state sum, given by the product of $R$-matrices (for crossings between strands of $L$), parametrized $\mathsf{R}$-matrices (for crossings involving a strand of $K$), cups and caps; 
this weight will be a monomial in the $3N$ parameters $\{\alpha\}, \{\beta\}, \{\gamma\}$, with coefficient a monomial in $\mathbb{Z}[x^{\pm\frac12}, q^{\pm \frac14}]$. 

\begin{table}
\centering
\begin{tabular}{|c|c||c|c|}
\hline
$
\vcenter{\hbox{
\begin{tikzpicture}[scale=0.7]
\draw[blue, ->, thick] (1, 0) -- (0, 1);
\draw[white, line width=5] (0, 0) -- (1, 1);
\draw[blue, ->, thick] (0, 0) -- (1, 1);
\node[blue, left] at (0, 0){$i$};
\node[blue, right] at (1, 0){$j$};
\node[blue, left] at (0, 1){$i'$};
\node[blue, right] at (1, 1){$j'$};
\node[blue, below] at (-0.2, -0.3){\tiny $V_\infty(x)$};
\node[blue, below] at (1.2, -0.3){\tiny $V_\infty(x)$};
\end{tikzpicture}
}}
$
 & $q^{\frac14} x^{-\frac12} \mathsf{R}(\alpha,\beta,\gamma)_{i,j}^{i',j'}$ 
&
$
\vcenter{\hbox{
\begin{tikzpicture}[scale=0.7]
\draw[blue, ->, thick] (0, 0) -- (1, 1);
\draw[white, line width=5] (1, 0) -- (0, 1);
\draw[blue, ->, thick] (1, 0) -- (0, 1);
\node[blue, left] at (0, 0){$i$};
\node[blue, right] at (1, 0){$j$};
\node[blue, left] at (0, 1){$i'$};
\node[blue, right] at (1, 1){$j'$};
\node[blue, below] at (-0.2, -0.3){\tiny $V_\infty(x)$};
\node[blue, below] at (1.2, -0.3){\tiny $V_\infty(x)$};
\end{tikzpicture}
}}
$
 & $q^{-\frac14} x^{\frac12} \mathsf{R}^{-1}(\alpha,\beta,\gamma)_{i,j}^{i',j'}$ 
\\
\hline
$
\vcenter{\hbox{
\begin{tikzpicture}[scale=0.7]
\draw[red, ->] (1, 0) -- (0, 1);
\draw[white, line width=5] (0, 0) -- (1, 1);
\draw[blue, ->, thick] (0, 0) -- (1, 1);
\node[blue, left] at (0, 0){$i$};
\node[red, right] at (1, 0){$j$};
\node[red, left] at (0, 1){$i'$};
\node[blue, right] at (1, 1){$j'$};
\node[blue, below] at (-0.2, -0.3){\tiny $V_\infty(x)$};
\node[red, below] at (1.2, -0.3){\tiny $V_m$};
\end{tikzpicture}
}}
$
 & $x^{\frac{m-1}{4}}q^{-\frac{m}{4}} \mathsf{R}(\alpha,\beta,\gamma)_{i,j}^{i',j'}$ 
&
$
\vcenter{\hbox{
\begin{tikzpicture}[scale=0.7]
\draw[blue, ->, thick] (0, 0) -- (1, 1);
\draw[white, line width=5] (1, 0) -- (0, 1);
\draw[red, ->] (1, 0) -- (0, 1);
\node[blue, left] at (0, 0){$i$};
\node[red, right] at (1, 0){$j$};
\node[red, left] at (0, 1){$i'$};
\node[blue, right] at (1, 1){$j'$};
\node[blue, below] at (-0.2, -0.3){\tiny $V_\infty(x)$};
\node[red, below] at (1.2, -0.3){\tiny $V_m$};
\end{tikzpicture}
}}
$
 & $x^{-\frac{m-1}{4}} q^{\frac{m}{4}} \mathsf{R}^{-1}(\alpha,\beta,\gamma)_{i,j}^{i',j'}$ 
 \\
\hline
$
\vcenter{\hbox{
\begin{tikzpicture}[scale=0.7]
\draw[blue, ->, thick] (1, 0) -- (0, 1);
\draw[white, line width=5] (0, 0) -- (1, 1);
\draw[red, ->] (0, 0) -- (1, 1);
\node[red, left] at (0, 0){$i$};
\node[blue, right] at (1, 0){$j$};
\node[blue, left] at (0, 1){$i'$};
\node[red, right] at (1, 1){$j'$};
\node[red, below] at (-0.2, -0.3){\tiny $V_m$};
\node[blue, below] at (1.2, -0.3){\tiny $V_\infty(x)$};
\end{tikzpicture}
}}
$
 & $x^{\frac{m-1}{4}} q^{-\frac{m}{4}} \mathsf{R}(\alpha,\beta,\gamma)_{i,j}^{i',j'}$ 
&
$
\vcenter{\hbox{
\begin{tikzpicture}[scale=0.7]
\draw[red, ->] (0, 0) -- (1, 1);
\draw[white, line width=5] (1, 0) -- (0, 1);
\draw[blue, ->, thick] (1, 0) -- (0, 1);
\node[red, left] at (0, 0){$i$};
\node[blue, right] at (1, 0){$j$};
\node[blue, left] at (0, 1){$i'$};
\node[red, right] at (1, 1){$j'$};
\node[red, below] at (-0.2, -0.3){\tiny $V_m$};
\node[blue, below] at (1.2, -0.3){\tiny $V_\infty(x)$};
\end{tikzpicture}
}}
$
 & $x^{-\frac{m-1}{4}} q^{\frac{m}{4}} \mathsf{R}^{-1}(\alpha,\beta,\gamma)_{i,j}^{i',j'}$
 \\
 \hline
\multicolumn{4}{|c|}{
{$
\begin{aligned}
\mathsf{R}(\alpha,\beta,\gamma)_{i,j}^{i', j'} &:= \delta_{i+j, i'+j'} \binom{i}{j'} \alpha^j \beta^{j'} \gamma^{i-j'}, \\
\mathsf{R}^{-1}(\alpha,\beta,\gamma)_{i,j}^{i',j'} &:= \mathsf{R}(\alpha,\beta,\gamma)_{j, i}^{j', i'}.
\end{aligned}
$}
} \\
\hline
\end{tabular}
\caption{Summary of parametrized $\mathsf{R}$-matrices; $\mathsf{R}^{-1}$ is just a notation and is \emph{not} the inverse of $\mathsf{R}$.}
\label{tab:parametrized-R-matrices}
\end{table}

The virtue of the free state sum is that it can be interpreted as a random walk model, analogous to that of \cite{LinWang}, which allows us to compute the state sum explicitly, in terms of the transition matrix $\mathcal{B}$ of the random walk -- we will get back to this point later. 
At the same time, the usual $R$-matrices can be seen as a ``perturbation'' of these parametrized $\mathsf{R}$-matrices in the sense of following lemmas, whose proofs are given in Appendix \ref{subsec:diff-op-lemmas}:

\begin{lem}\label{lem:diff-op}
There exist differential operators $\mathsf{D}_d \in \mathbb{Q}[\alpha^{\frac12},\beta^{\frac12},\gamma^{\pm 1}, \partial_{\alpha}, \partial_{\beta}, \partial_{\gamma}]$ of degree $\deg_{\partial} \mathsf{D}_d \leq 2d$, with $\mathsf{D}_0 = 1$, such that
\begin{align*}
R(x_1, x_2, e^{\hbar})_{i,j}^{i', j'} &= 
x_1^{-\frac14}x_2^{-\frac14}
\sum_{d\geq 0} \hbar^d \mathsf{D}_d \cdot 
\mathsf{R}(\alpha, \beta, \gamma)_{i,j}^{i',j'} \bigg\vert_{\substack{\alpha = x_1^{-\frac{1}{2}} \\ \beta = x_2^{-\frac{1}{2}} \\ \gamma = x_1^{-\frac{1}{4}} x_2^{\frac{1}{4}} (1-x_2^{-1})}}, 
\\
R^{-1}(x_1, x_2, e^{\hbar})_{i,j}^{i', j'} 
&= x_1^{\frac14} x_2^{\frac14} 
\sum_{d\geq 0} (-\hbar)^d \mathsf{D}_d \cdot \mathsf{R}^{-1}(\alpha,\beta,\gamma)_{i,j}^{i',j'} \bigg\vert_{\substack{
\alpha = x_2^{\frac12} \\
\beta = x_1^{\frac12} \\
\gamma = x_1^{-\frac14} x_2^{\frac14} (1-x_1)
}} .
\end{align*}
\end{lem}
\begin{lem}\label{lem:diff-op-relative}
Once $\underline{i}', \underline{j} \in \{0, 1, \cdots, m-1\}$ are fixed, there exist differential operators $\mathsf{D}'_d \in \mathbb{Q}[\beta,\partial_\beta]$ of degree $\deg_\partial \mathsf{D}'_d \leq d$ such that
\[
R(x, q^m, e^{\hbar})_{i,\underline{j}}^{\underline{i}', j'} = 
x^{-\frac14}q^{-\frac{m}{4}} \sum_{d\geq 0} \hbar^d \mathsf{D}'_d \cdot 
\mathsf{R}(\alpha, \beta, \gamma)_{i,\underline{j}}^{\underline{i}',j'} \bigg\vert_{\substack{\alpha = x^{-\frac{1}{2}} \\ \beta = q^{-\frac{m}{2}} \\ \gamma = x^{-\frac{1}{4}} q^{\frac{m}{4}} (1-q^{-m})}}.
\]
Likewise, once $\underline{i}, \underline{j}' \in \{0, 1, \cdots, m-1\}$ are fixed, there exist differential operators $\mathsf{D}''_d \in \mathbb{Q}[\beta,\partial_\beta]$ of degree $\deg_\partial \mathsf{D}''_d \leq d$ such that
\[
R^{-1}(q^m,x,e^{\hbar})_{\underline{i},j}^{i', \underline{j}'} = x^{\frac14} q^{\frac{m}{4}} \sum_{d\geq 0} \hbar^d \mathsf{D}''_d \cdot \mathsf{R}^{-1}(\alpha,\beta,\gamma)_{\underline{i},j}^{i', \underline{j}'} \bigg\vert_{\substack{
\alpha = x^{\frac12}\\
\beta = q^{\frac{m}{2}}\\
\gamma = x^{\frac14}q^{-\frac{m}{4}}(1-q^m)
}}.
\]
\end{lem}
Whenever the under-strand of a crossing is a strand of $K$, specialization of Lemma \ref{lem:diff-op} at $x_2 = x$ and $x_1 = x$ (resp. $x_1 = q^m$) if the over-strand is a strand of $K$ (resp. strand of $L$ colored by $V_m$)
gives the desired expressions for the corresponding $R$-matrices. 
When the under-strand is a strand of $L$ (say, colored by $V_m$), however, we can't directly specialize Lemma \ref{lem:diff-op} at $x_2 = q^m$, since the differential operator $\mathsf{D}_d$ may contain negative powers of $\gamma$, which is of order $O(\hbar^{1})$ in this specialization. 
For this reason, when the under-strand is a strand of $L$, we use Lemma \ref{lem:diff-op-relative} instead, making use of the fact that we are fixing the spins on $L$, to choose differential operators (dependent on the spins on $L$) that do not involve any negative powers of $\gamma$. 

Combining the differential operators in Lemmas \ref{lem:diff-op} and \ref{lem:diff-op-relative}, we have:
\begin{cor}
There exist differential operators $\mathsf{D}^{\mathrm{total}}_d$ of degree $\deg_{\partial} \mathsf{D}^{\mathrm{total}}_d \leq 2d$ in $3N$ parameters $\{\alpha\}, \{\beta\}, \{\gamma\}$, involving only non-negative powers of $\gamma_c$ whenever the under-strand of the crossing $c$ is $L$, such that
\begin{align*}
&\sum_{s\vert_K}
\langle D_{(K,V_\infty(x)), \{(L_i, V_{m_i})\}}; s\rangle \bigg\vert_{q=e^{\hbar}}\\ 
&= 
\sum_{d\geq 0} \hbar^d
\mathsf{D}^{\mathrm{total}}_d
\cdot
\sum_{s\vert_K}
\langle D_{(K,V_\infty(x)), \{(L_i, V_{m_i})\}}; \{\alpha\}, \{\beta\}, \{\gamma\}; s\rangle^{\mathrm{free}}
\bigg\vert_{\{\alpha\},\{\beta\},\{\gamma\}}
,
\end{align*}
where the parameters are specialized as in Lemmas \ref{lem:diff-op} and \ref{lem:diff-op-relative}. 
\end{cor}

Thanks to the bound on the degree of the differential operators $\mathsf{D}^{\mathrm{total}}_d$, Lemma \ref{lem:perturbed-random-walk} will follow from the following last lemma:
\begin{lem}\label{lem:free-state-sum}
The free state sum is a rational function in the parameters and $q^{\pm \frac14}$ and has the following form:
\[
\sum_{s\vert_K}
\langle D_{(K,V_\infty(x)), \{(L_i, V_{m_i})\}}; \{\alpha\}, \{\beta\}, \{\gamma\}; s\rangle^{\mathrm{free}} = 
\frac{\mathsf{P}^{s\vert_L,\textrm{free}}}{\det(I-\mathcal{B})^{1+\delta}},
\]
where $\mathcal{B}$ denotes the transition matrix of the associated random walk on the arcs of $K$ (see Remark \ref{rmk:random-walk}), 
$\mathsf{P}^{s\vert_L,\textrm{free}}$ is some polynomial in the parameters $\{\alpha\}, \{\beta\}, \{\gamma\}$ and $q^{\pm \frac14}$ whose $\gamma$-degree is at least $\delta$, the total interaction between $K$ and $L$. 
In particular, after specialization of the parameters, it becomes 
\[
\sum_{k \geq \delta} \frac{P_k^{s\vert_L,\textrm{free}}(x)}{\Delta_K(x)^{k + 1}} \hbar^{k}
\]
for some Laurent polynomials $P_k^{s\vert_L,\textrm{free}}(x) \in \mathbb{Q}[x^{\pm \frac12}]$. 
\end{lem}

\begin{rmk}\label{rmk:random-walk}
By ``the associated random walk on the arcs of $K$,'' we mean the Markov chain whose nodes are the internal arcs of $K$ (minus the local maxima and minima) and whose ``transition probabilities'' are given by the parameter $\alpha_c, \beta_c$, or $\gamma_c$ at each crossing $c$, depending on the type of jump\footnote{$\alpha_c$ for the under-strand, $\beta_c$ for the over-strand, and $\gamma_c$ for the jump from over to under. }, and $q^{\pm \frac12}$ when going through a local maximum or minimum.\footnote{$q^{-\frac12}$ when going clockwise and $q^{\frac12}$ when going anti-clockwise.} 
See Example \ref{eg:random-walk} for an explicit example. 
\end{rmk}

\begin{proof}[Proof of Lemma \ref{lem:free-state-sum}]
Similarly to \cite{LinWang}, the main idea is to make use of the fact that the free state sum can be interpreted in terms of random walks of free bosons. 
Once a state is fixed, we think of each arc $a$ of $K$ with spin $s(a)$ as being occupied by $s(a)$ bosons.\footnote{We use the term ``boson'' just to indicate that the same arc can be occupied by multiple bosons. 
Another model, which counts simple multi-cycle-paths with sign (see Section \ref{subsec:proof-of-Thm-B}) will be called a random walk of ``fermions.''} 
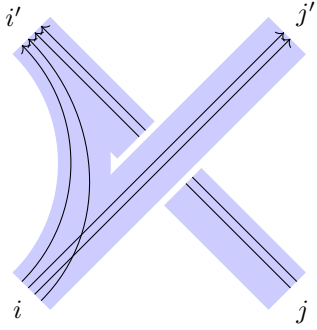
\begin{figure}
\centering
\[
\vcenter{\hbox{
\begin{tikzpicture}[scale=0.85]
\draw[blue!20, line width=20] (2, -2) -- (-2, 2);
\draw[blue!20, line width=20] (-2, -2) to[out=45, in=-45] (-2, 2);
\draw[->] (2.05, -1.95) -- (-1.95, 2.05);
\draw[->] (2.15, -1.85) -- (-1.85, 2.15);
\draw[white, line width=25] (-0.3, -0.3) -- (2, 2);
\draw[blue!20, line width=20] (-2, -2) -- (2, 2);
\draw[->] (-1.95, -2.05) -- (2.05, 1.95);
\draw[->] (-2.05, -1.95) -- (1.95, 2.05);
\draw[->] (-2.15, -1.85) to[out=45, in=-45] (-2.15, 1.85);
\draw[->] (-1.85, -2.15) to[out=45, in=-45] (-2.05, 1.95);
\node[below left] at (-2, -2){$i$};
\node[below right] at (2, -2){$j$};
\node[above left] at (-2, 2){$i'$};
\node[above right] at (2, 2){$j'$};
\end{tikzpicture}
}}
\]
\caption{Highway-and-cars diagram}
\label{fig:highway-and-cars diagram}
\end{figure}
For each, say positive, crossing with the spins on the four neighboring arcs $i,j,i',j'$ as usual, exactly $i-j'$ out of the $i$ bosons from the bottom left arc must jump to the top left arc, but the state itself doesn't specify which bosons jump; we call the information specifying that the \emph{jumping datum}. Note, there are $\binom{i}{j'}$ such choices, so we can drop the binomial factors in the parametrized $\mathsf{R}$-matrices and instead sum over states together with a decoration by the jumping data. 

A convenient way to visualize a jumping datum is to use what we will call the ``highway-and-cars diagram,''\footnote{While not completely unrelated, not to be confused with the cars in \cite[Sec. 3.2]{BarNatanvanderVeen}.} as shown in Figure \ref{fig:highway-and-cars diagram}, which is a decoration of the highway diagram (as in \cite{Park_inverted}) by the trajectories of the cars (i.e. bosons) on the highway. 
That is, each arc $a$ with spin $s(a)$ is represented by a highway with $s(a)$ lanes, each of which is occupied with a car, and the trajectories of the cars can cross only when they are approaching an interchange to move from an over-strand to an under-strand. 
Which cars stay on the over-strand and which cars go down to an under-strand determine the jumping datum. 

Once we have a highway-and-cars diagram, we have trajectories of the cars. 
Some of them will be cycles (i.e. closed loops), while others are paths with distinct endpoints; 
the sources (resp. sinks) of those paths must be where some spins jump from $L$ to $K$ (resp. from $K$ to $L$). 
The number of paths among those trajectories is exactly $\delta$, the total interaction between $K$ and $L$ ($=$ total jump of spins from $L$ to $K$ $=$ total jump of spins from $K$ to $L$). 
We call an unordered tuple of such cycles and paths a \emph{multi-cycle-path}, and call it \emph{primitive} if all the cycles appearing in the tuple are primitive, in the sense that it is not a power of any shorter cycle. 

Any multi-cycle-path coming from a highway-and-cars diagram must be primitive, since the trajectories of the cars can only cross when their routes diverge from each other. 
On the other hand, given any primitive multi-cycle-path, we can construct the corresponding highway-and-cars diagram: for each arc $a$, the set of all occurrences\footnote{The same arc may be used multiple times in a given cycle or a path, so by ``occurrence'', we are counting all of them.} of $a$ in a given primitive multi-cycle-path is totally ordered, by tracing back to the most recent crossing where they merged. 
It follows that there is a one-to-one correspondence between 
\begin{enumerate}
\item the states on $K$ decorated by jumping data and 
\item primitive multi-cycle-paths with the prescribed sources and sinks determined by the state on $L$ which we have fixed in the background. 
\end{enumerate}

Note that, up to a common overall factor of order $O(\hbar^0)$, the weight is multiplicative with respect to superposition, so that the weight of a primitive multi-cycle-path is the product of the weights of the individual primitive cycles and paths; 
this is why we call the state sum defined by parametrized $\mathsf{R}$-matrices ``free.''
We can then factorize the primitive multi-cycle part out, which gives a factor of $\frac{1}{\det(I-\mathcal{B})}$. 
The sum over paths gives a finite sum of products of $\delta$ terms like 
\[
(I + \mathcal{B} + \mathcal{B}^2 + \cdots)_{a,a'} = (I-\mathcal{B})^{-1}_{a,a'} = \frac{\mathrm{adj}(I-\mathcal{B})_{a,a'}}{\det(I-\mathcal{B})},
\]
where $\mathrm{adj}$ denotes the adjugate, and $a,a'$ are arcs at the start and end points of the path. 
It follows that the free state sum is a rational function with denominator $\det(I-\mathcal{B})^{1+\delta}$, as desired. 
\end{proof}

\begin{eg}\label{eg:random-walk}
Let us illustrate this through a concrete example. 
Consider the tangle diagram of $K \sqcup L$ depicted in the left-hand side of Figure \ref{fig:Markov-chain}, with a fixed state $s\vert_L$ on $L$. 
The corresponding Markov chain, along with the sources and sinks, is drawn on the right-hand side. 
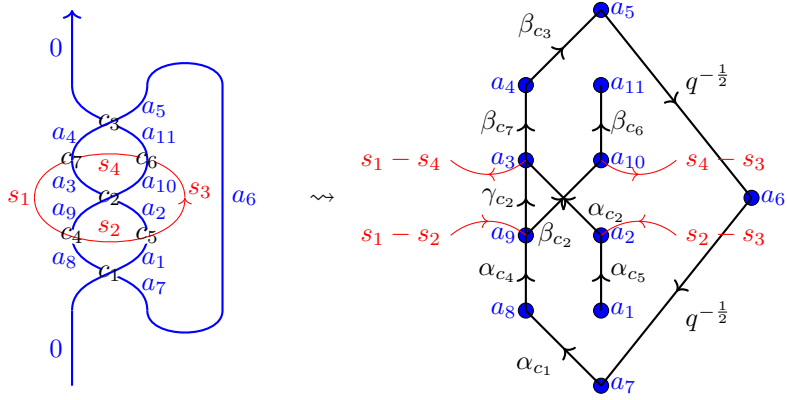
\begin{figure}
\centering
\[
\vcenter{\hbox{
\begin{tikzpicture}
\draw[red] (-0.5, 1.5) to[out=90, in=90] (1.5, 1.5); 
\draw[white, line width=5] (0, 0) to[out=90, in=-90] (1, 1) to[out=90, in=-90] (0, 2) to[out=90, in=-90] (1, 3);
\draw[white, line width=5] (1, 0) to[out=90, in=-90] (0, 1) to[out=90, in=-90] (1, 2) to[out=90, in=-90] (0, 3);

\draw[blue, thick] (0, -1) -- (0, 0);
\draw[blue, thick] (1, 0) to[out=90, in=-90] (0, 1);
\draw[white, line width=5] (0, 0) to[out=90, in=-90] (1, 1);
\draw[blue, thick] (0, 0) to[out=90, in=-90] (1, 1);
\draw[blue, thick] (1, 1) to[out=90, in=-90] (0, 2);
\draw[white, line width=5] (0, 1) to[out=90, in=-90] (1, 2);
\draw[blue, thick] (0, 1) to[out=90, in=-90] (1, 2);
\draw[blue, thick] (1, 2) to[out=90, in=-90] (0, 3);
\draw[white, line width=5] (0, 2) to[out=90, in=-90] (1, 3);
\draw[blue, thick] (0, 2) to[out=90, in=-90] (1, 3);
\draw[blue, thick, ->] (0, 3) -- (0, 4);
\draw[blue, thick] (1, 3) to[out=90, in=90] (2, 3) -- (2, 0) to[out=-90, in=-90] (1, 0);
\draw[white, line width=5] (-0.5, 1.5) to[out=-90, in=-90] (1.5, 1.5); 
\draw[red, ->] (-0.5, 1.5) to[out=-90, in=-90] (1.5, 1.5); 

\node[blue, left] at (0, -0.5){$0$};
\node[blue, left] at (0, 3.5){$0$};
\node[blue, below right] at (0.8, 0.9){$a_1$};
\node[blue, above right] at (0.8, 1.1){$a_2$};
\node[blue, below left] at (0.2, 1.9){$a_3$};
\node[blue, above left] at (0.2, 2.1){$a_4$};
\node[blue, below right] at (0.8, 2.9){$a_5$};
\node[blue, right] at (2, 1.5){$a_6$};
\node[blue, above right] at (0.8, 0.1){$a_7$};
\node[blue, below left] at (0.2, 0.9){$a_8$};
\node[blue, above left] at (0.2, 1.1){$a_9$};
\node[blue, below right] at (0.8, 1.9){$a_{10}$};
\node[blue, above right] at (0.8, 2.1){$a_{11}$};
\node at (0.5, 0.5){$c_1$};
\node at (0.5, 1.5){$c_2$};
\node at (0.5, 2.5){$c_3$};
\node at (0, 1){$c_4$};
\node at (1, 1){$c_5$};
\node at (1, 2){$c_6$};
\node at (0, 2){$c_7$};
\node[red] at (-0.7, 1.5){$s_1$};
\node[red] at (0.5, 1.1){$s_2$};
\node[red] at (1.7, 1.6){$s_3$};
\node[red] at (0.5, 1.9){$s_4$};
\end{tikzpicture}
}}
\quad\rightsquigarrow
\vcenter{\hbox{
\begin{tikzpicture}[
    mid/.style={
        postaction={
            decorate,
            decoration={
                markings,
                mark=at position 0.5 with {\arrow{>}} 
            }
        }
    },
    midback/.style={
        postaction={
            decorate,
            decoration={
                markings,
                mark=at position 0.5 with {\arrow{<}} 
            }
        }
    },
]
\coordinate (a1) at (1, 0);
\coordinate (a2) at (1, 1);
\coordinate (a3) at (0, 2);
\coordinate (a4) at (0, 3);
\coordinate (a5) at (1, 4);
\coordinate (a6) at (3, 1.5);
\coordinate (a7) at (1, -1);
\coordinate (a8) at (0, 0);
\coordinate (a9) at (0, 1);
\coordinate (a10) at (1, 2);
\coordinate (a11) at (1, 3);
\coordinate (c4) at (-1, 1);
\coordinate (c5) at (2, 1);
\coordinate (c6) at (2, 2);
\coordinate (c7) at (-1, 2);
\draw[fill=blue] (a1) circle (0.1);
\draw[fill=blue] (a2) circle (0.1);
\draw[fill=blue] (a3) circle (0.1);
\draw[fill=blue] (a4) circle (0.1);
\draw[fill=blue] (a5) circle (0.1);
\draw[fill=blue] (a6) circle (0.1);
\draw[fill=blue] (a7) circle (0.1);
\draw[fill=blue] (a8) circle (0.1);
\draw[fill=blue] (a9) circle (0.1);
\draw[fill=blue] (a10) circle (0.1);
\draw[fill=blue] (a11) circle (0.1);
\node[blue, right] at (a1){$a_1$};
\node[blue, right] at (a2){$a_2$};
\node[blue, left] at (a3){$a_3$};
\node[blue, left] at (a4){$a_4$};
\node[blue, right] at (a5){$a_5$};
\node[blue, right] at (a6){$a_6$};
\node[blue, right] at (a7){$a_7$};
\node[blue, left] at (a8){$a_8$};
\node[blue, left] at (a9){$a_9$};
\node[blue, right] at (a10){$a_{10}$};
\node[blue, right] at (a11){$a_{11}$};
\draw[mid, thick] (a1) -- (a2);
\draw[mid, thick] (a2) -- (a3);
\draw[mid, thick] (a3) -- (a4);
\draw[mid, thick] (a4) -- (a5);
\draw[mid, thick] (a5) -- (a6);
\draw[mid, thick] (a6) -- (a7);
\draw[mid, thick] (a7) -- (a8);
\draw[mid, thick] (a8) -- (a9);
\draw[mid, thick] (a9) -- (a10);
\draw[mid, thick] (a9) -- (a3);
\draw[mid, thick] (a10) -- (a11);
\draw[mid, red] (c4) to[out=40, in=140] (a9);
\draw[mid, red] (c5) to[out=140, in=40] (a2);
\draw[mid, red] (a10) to[out=-40, in=-140] (c6);
\draw[mid, red] (a3) to[out=-140, in=-40] (c7);
\node[red, left] at (c4){$s_1-s_2$};
\node[red, right] at (c5){$s_2-s_3$};
\node[red, right] at (c6){$s_4-s_3$};
\node[red, left] at (c7){$s_1-s_4$};
\node[right] at ($0.5*(a1)+0.5*(a2)$){$\alpha_{c_5}$};
\node[right] at ($0.7*(a2)+0.3*(a3)$){$\alpha_{c_2}$};
\node[left] at ($0.5*(a3)+0.5*(a4)$){$\beta_{c_7}$};
\node[above left] at ($0.5*(a4)+0.5*(a5)$){$\beta_{c_3}$};
\node[above right] at ($0.5*(a5)+0.5*(a6)$){$q^{-\frac{1}{2}}$};
\node[below right] at ($0.5*(a6)+0.5*(a7)$){$q^{-\frac{1}{2}}$};
\node[below left] at ($0.5*(a7)+0.5*(a8)$){$\alpha_{c_1}$};
\node[left] at ($0.5*(a8)+0.5*(a9)$){$\alpha_{c_4}$};
\node[below] at ($0.6*(a9)+0.4*(a10)+(0, -0.1)$){$\beta_{c_2}$};
\node[left] at ($0.5*(a9)+0.5*(a3)$){$\gamma_{c_2}$};
\node[right] at ($0.5*(a10)+0.5*(a11)$){$\beta_{c_6}$};
\end{tikzpicture}
}}
\]
\caption{Random walk (Markov chain) example; $a_1, \cdots, a_{11}$ label the internal arcs of $K$ minus the local maxima and minima, $c_1, \cdots, c_7$ label the crossings, and $s_1, \cdots, s_4$ denote the spins on arcs of $L$.}
\label{fig:Markov-chain}
\end{figure}
Let $\mathcal{B}$ be the transition matrix of this Markov chain; explicitly, 
\[
\mathcal{B} = 
\begin{pmatrix}
0 & \alpha_{c_5} & 0 & 0 & 0 & 0 & 0 & 0 & 0 & 0 & 0 \\
0 & 0 & \alpha_{c_2} & 0 & 0 & 0 & 0 & 0 & 0 & 0 & 0 \\
0 & 0 & 0 & \beta_{c_{7}} & 0 & 0 & 0 & 0 & 0 & 0 & 0 \\
0 & 0 & 0 & 0 & \beta_{c_3} & 0 & 0 & 0 & 0 & 0 & 0 \\
0 & 0 & 0 & 0 & 0 & q^{-\frac{1}{2}} & 0 & 0 & 0 & 0 & 0 \\
0 & 0 & 0 & 0 & 0 & 0 & q^{-\frac{1}{2}} & 0 & 0 & 0 & 0 \\
0 & 0 & 0 & 0 & 0 & 0 & 0 & \alpha_{c_1} & 0 & 0 & 0 \\
0 & 0 & 0 & 0 & 0 & 0 & 0 & 0 & \alpha_{c_4} & 0 & 0 \\
0 & 0 & \gamma_{c_2} & 0 & 0 & 0 & 0 & 0 & 0 & \beta_{c_2} & 0 \\
0 & 0 & 0 & 0 & 0 & 0 & 0 & 0 & 0 & 0 & \beta_{c_6} \\
0 & 0 & 0 & 0 & 0 & 0 & 0 & 0 & 0 & 0 & 0
\end{pmatrix}.
\]
The parameters eventually get specialized in the following way:
\[
\alpha_{c_1} = \alpha_{c_2} = \beta_{c_3} = x^{-\frac12},\quad \gamma_{c_2} = 1-x^{-1},\quad \alpha_{c_4} = \alpha_{c_5} = \beta_{c_6} = \beta_{c_7} = q^{-\frac{m}{2}}.
\]
The free state sum is then the weighted count of primitive multi-cycle-paths with the prescribed sources and sinks. 
Suppose the state on $L$ is $(s_1,s_2,s_3,s_4) = (2,1,0,1)$ for example. 
Then, there are 2 ways to connect the sources to sinks; either $(a_9 \rightarrow a_{10},\; a_{2} \rightarrow a_3)$ or $(a_9 \rightarrow a_{3},\; a_{2} \rightarrow a_{10})$. 
As a result, the total weighted count of multi-cycle-paths in this example is
\[
\frac{\mathrm{adj}(I-\mathcal{B})_{a_{9},a_{10}}\mathrm{adj}(I-\mathcal{B})_{a_{2},a_{3}} + \mathrm{adj}(I-\mathcal{B})_{a_{9},a_{3}}\mathrm{adj}(I-\mathcal{B})_{a_{2},a_{10}}}{\det(I-\mathcal{B})^{3}}.
\]
\end{eg}

\section{From skeins to $q$-series}\label{sec:skein-q-series}
\subsection{Brief review of inverted state sum}
Here we briefly recall how the BPS $q$-series are defined for complements of ``nice'' knots, using the inverted state sum \cite{Park_inverted, Park-thesis}. 
An \emph{inversion datum} on a diagram of $K$ is an assignment of signs ($+$ or $-$) on each arc, in such a way that the $-$-sign (and hence also the $+$-sign) appears an even number of times at each crossing. 
An inversion datum determines an \emph{inverted state sum}: 
it is defined in the same way as the infinite state sum using $R$-matrices for Verma modules that we saw in the previous section, except we sum over negative (instead of non-negative) integer spins on the arcs labeled $-$ by the inversion datum, using the fact that the $R$-matrices naturally extend to all integer spin indices.\footnote{There is also an overall sign $(-1)^s$ in the inverted state sum that is determined by the inversion datum; see \cite{Park_inverted}.} 
We call an inversion datum \emph{nice} if the inverted state sum converges as a Laurent series in $x^{-\frac12}$ with coefficients in Laurent polynomials in $q^{\frac14}$, i.e., if only finitely many states contribute to each $x$-degree. 
Likewise, we call a knot diagram \emph{nice} if it admits a nice inversion datum, and call the knot itself \emph{nice} if it admits a nice knot diagram. 
Many knots are known to be nice; for instance, any homogeneous braid closure is nice, and all fibered knots with at most 12 crossings are known to be nice \cite{Park_inverted, OSSS}. 
\begin{thm}[{\cite{Park_inverted}}]
For any nice inversion datum on a diagram of $K$, the $\hbar$-expansion of the associated inverted state sum agrees with the $x^{-1}$-expansion of the MMR expansion. 
In particular, the inverted state sum is an invariant of nice knots (i.e., independent of the choice of the nice diagram). 
\end{thm}
For any nice knot $K$, we call the resulting invariant the \emph{BPS $q$-series} of $S^3 \setminus K$ and denote it by $\widehat{Z}_{S^3 \setminus K}(x,q)$.\footnote{The BPS $q$-series for knot complements are also often denoted by $F_K(x,q)$.}

\subsection{Proof of Theorem \ref{mainthm:skein-q-series}}\label{subsec:proof-of-Thm-B}
By assumption, there is a nice diagram of $K$ as a long knot. 
Consider a $(1,1)$-tangle diagram of $K \sqcup L$ where the diagram of the long knot $K$ is nice, while the diagram of $L$ can be arbitrary; see, e.g., Figure \ref{fig:(1,1)-tangle}. 
We can consider the inverted state sum on this diagram, with the associated inversion datum on $K$. 
The inverted state sum still converges in $\mathbb{Z}[q^{\pm \frac14}]((x^{-\frac12}))$;\footnote{
Here, we are using power series in $x^{-1}$, following the convention of \cite{Park_inverted}. 
One can turn this into a power series in $x$ using Weyl symmetry $\widehat{Z}_{S^3 \setminus K}(x,q) = -\widehat{Z}_{S^3 \setminus K}(x^{-1},q)$. 
} 
the insertion of $L$ does not affect the convergence of the inverted state sum, as there are only finitely many possible spins on $L$. 

To conclude the proof, it suffices to show that the $\hbar$-expansion of this series agrees with that of the MMR expansion, 
since the MMR expansion is isotopy-invariant and factors through the skein relations on $L$. 
The proof of that is completely analogous to the one given in the proof of \cite[Thm. 1]{Park_inverted}, which was for the case when $L$ is the empty link: 
\begin{enumerate}
\item The differential operators $\mathsf{D}_d^{\mathrm{total}}$ remain the same, before and after inversion; see Corollary \ref{cor:diff-op-inversion}. 
Therefore it suffices to show that the free state sum (as in Lemma \ref{lem:free-state-sum}) matches, before and after inversion. 
\item Recall, from the proof of Lemma \ref{lem:free-state-sum}, that the free state sum is a finite sum of rational functions of the form
\[
\frac{\prod_{1\leq i\leq \delta}\mathrm{adj}(I-\mathcal{B})_{a_i,a_i'}}{\det(I-\mathcal{B})^{1+\delta}},
\]
where $\mathcal{B}$ is the transition matrix of the random walk, and $a_i, a_i'$ are some arcs. 
A key observation is that, just like how $\det(I-\mathcal{B})$ can be interpreted in terms of the random walk of fermions, i.e., as the signed sum of weights of simple multi-cycles
\[
\det(I-\mathcal{B}) = \sum_{\mathcal{C}} (-1)^{|\mathcal{C}|} \mathrm{wt}(\mathcal{C}),
\]
there is an analogous interpretation for the cofactors $\mathrm{adj}(I-\mathcal{B})_{a,a'}$. 
\begin{lem}
\[
\mathrm{adj}(I-\mathcal{B})_{a,a'} = \sum_{(\mathcal{C},P)}(-1)^{|\mathcal{C}|} \mathrm{wt}(\mathcal{C},P),
\]
where the sum is over all simple $(a \rightarrow a')$-multi-cycle-paths $(\mathcal{C},P)$.\footnote{That is, $P$ is a simple path from $a$ to $a'$ and does not share any arcs with the simple multicycle $\mathcal{C}$.} 
\end{lem}
\begin{proof}
From the definition of a cofactor, we have
\[
\mathrm{adj}(I-\mathcal{B})_{a,a'} =
\sum_{\substack{\sigma \in S_{\mathcal{A}_K} \\ \sigma(a')=a}} \mathrm{sgn(\sigma)} \prod_{b \neq a'} (I-\mathcal{B})_{b,\sigma(b)}.
\]
The cycle decomposition of $\sigma$ contains a unique cycle 
\[
a' \mapsto a \mapsto \sigma(a) \mapsto \sigma^2(a) \mapsto \cdots \mapsto a'
\]
containing $a$ and $a'$, which we view as a simple path $P$ from $a$ to $a'$. 
The set of all the remaining non-trivial cycles forms a simple multi-cycle $\mathcal{C}$ disjoint from $P$. 
Rewriting the sum in terms of simple $(a\rightarrow a')$-multi-cycle-paths, we conclude
\begin{align*}
&\sum_{\substack{\sigma \in S_{\mathcal{A}_K} \\ \sigma(a')=a}} \mathrm{sgn(\sigma)} \prod_{b \neq a'} (I-\mathcal{B})_{b,\sigma(b)} \\
&=
\sum_{(\mathcal{C},P)} 
\qty(
\prod_{C \in \mathcal{C}}(-1)^{l(C)-1} \cdot (-1)^{l(P)}) \cdot 
\qty(
\prod_{C \in \mathcal{C}}(-1)^{l(C)} \cdot (-1)^{l(P)} \cdot
\mathrm{wt}(\mathcal{C},P)
) \\
&= \sum_{(\mathcal{C},P)} (-1)^{|\mathcal{C}|} \mathrm{wt}(\mathcal{C},P).
\end{align*}
\end{proof}
It was shown in \cite{Park_inverted} (see, especially \cite[Fig. 8]{Park_inverted}) that there is a bijection between the simple multi-cycles before and after inversion; the bijection simply swaps the set of used arcs and unused arcs within the set of inverted arcs. 
The same bijection can be extended to a bijection between simple multi-cycle-paths before and after inversion. 
Moreover, just as in the proof of \cite[Thm. 1]{Park_inverted}, it can be easily checked that the bijection preserves the weights up to an overall monomial, which cancels the overall monomial prefactor (see, e.g. the numerator of \cite[eqn. 8]{Park_inverted}). 
It follows that the free state sum matches before and after inversion. 
\end{enumerate}

\qed

\subsection{Boundary action}
It is worth noting how the boundary action of the skein algebra $\SkAlg^{\mathrm{SL}_2}_q(T^2)$ on $\Sk^{\mathrm{SL}_2}_q(S^3 \setminus K)$ affects the corresponding BPS $q$-series. 
Let $\hat{x}^{\frac12}$ and $\hat{y}$ be operators acting $\mathbb{Z}[q^{\pm \frac14}]$-linearly on $\mathbb{Z}[q^{\pm \frac14}]((x^{\frac12}))$ by 
\begin{align*}
\hat{x}^{\frac12} : x^u &\mapsto x^{u+\frac12}, \\
\hat{y} : x^u &\mapsto q^u x^u.
\end{align*}
By a local calculation, one can see that the meridian (colored by $V_2$) acts by $\hat{x}^{\frac12}+\hat{x}^{-\frac12}$, while the longitude (colored by $V_2$) acts by $\hat{y}+\hat{y}^{-1}$, since
\[
V_2 \otimes V_\infty(x) \cong V_\infty(q x) \oplus V_\infty(q^{-1}x).
\]
It follows that the boundary action of $\SkAlg^{\mathrm{SL}_2}_q(T^2)$ on $\Sk^{\mathrm{SL}_2}_q(S^3 \setminus K)$, under $\widehat{Z}$, factors through the embedding of Frohman--Gelca \cite{FrohmanGelca}
\[
\SkAlg^{\mathrm{SL}_2}_q(T^2) \hookrightarrow \mathbb{Z}[q^{\pm \frac14}]\langle \hat{x}^{\pm \frac12}, \hat{y}^{\pm 1} \rangle
\]
of the skein algebra of $T^2$ into the quantum torus.

\subsection{Examples}\label{subsec:two-variable-series-module-examples}
\begin{eg}[Unknot]
Let $\widehat{Z}_{S^3 \setminus \mathbf{0}_1, W_n}(x, q)$ denote the two-variable series associated to the complement of the unknot, with an insertion of the Wilson line defect\footnote{As remarked in an earlier footnote, a Wilson line defect is just an embedded link each of whose component is colored by some object of $\mathrm{Rep}_q(G)$, and we are using $G = \mathrm{SL}_2$ here.} $W_n$ along the meridian, colored by $V_n$. 
Let us use the symbol $\chi_m := \frac{x^{\frac{m}{2}}-x^{-\frac{m}{2}}}{x^{\frac12}-x^{-\frac12}}$ to simplify notation. 
Then, 
\[
\frac{\widehat{Z}_{S^3 \setminus \mathbf{0}_1, W_n}(x, q)}{x^{\frac12}-x^{-\frac12}} = \chi_n. 
\]
In general, for any knot $K$, if $\widehat{Z}_{S^3 \setminus K, W_n}(x, q)$ denotes the two-variable series associated to the complement of $K$, with an insertion of the Wilson line colored by $V_n$ along the meridian, then we have
\[
\widehat{Z}_{S^3 \setminus K, W_n}(x, q) = \chi_n \cdot \widehat{Z}_{S^3 \setminus K}(x,q). 
\]
Note that 
\[
\chi_2 \cdot \chi_n = \chi_{n-1} + \chi_{n+1}
\]
for any $n\geq 2$, which is nothing but the character formula for the fusion of two Wilson lines along the meridian. 

Let $R_{\chi}$ denote the unital commutative $\mathbb{Z}[q^{\pm \frac14}]$-algebra spanned by $\{\chi_n\}_{n\geq 1}$, i.e., the ring of characters of $\mathfrak{sl}_2$. 
Also, let the \emph{two-variable series module} $V_{x,q}(S^3 \setminus K)$ be the image of the skein module $\Sk_q^{\mathrm{SL}_2}(S^3 \setminus K)$ under $\widehat{Z}$, normalized by $x^{\frac12}-x^{-\frac12}$:
\[
V_{x,q}(S^3 \setminus K) := \frac{1}{x^{\frac12}-x^{-\frac12}} \mathrm{Im}\qty(\widehat{Z} : \Sk_q^{\mathrm{SL}_2}(S^3 \setminus K) \rightarrow \mathbb{Z}[q^{\pm \frac14}]((x^{\frac12})) ). 
\]
Then, the above discussion implies that, $V_{x,q}(S^3 \setminus K)$ is naturally a module over $R_{\chi}$. 
Clearly, the two-variable series module for the unknot complement is the rank-$1$ free $R_{\chi}$-module spanned by $1$: 
\[
V_{x,q}(S^3 \setminus \mathbf{0}_1) = R_{\chi}\{1\}. 
\] 
\end{eg}

\begin{eg}[Trefoil knot]\label{eg:trefoil-two-variable-series-module}
\begin{figure}
\centering
\[
\vcenter{\hbox{
\begin{tikzpicture}
\draw[red] (-0.5, 1.5) to[out=90, in=90] (1.5, 1.5); 
\draw[white, line width=5] (0, 0) to[out=90, in=-90] (1, 1) to[out=90, in=-90] (0, 2) to[out=90, in=-90] (1, 3);
\draw[white, line width=5] (1, 0) to[out=90, in=-90] (0, 1) to[out=90, in=-90] (1, 2) to[out=90, in=-90] (0, 3);

\draw[blue, thick] (0, -1) -- (0, 0);
\draw[blue, thick] (1, 0) to[out=90, in=-90] (0, 1);
\draw[white, line width=5] (0, 0) to[out=90, in=-90] (1, 1);
\draw[blue, thick] (0, 0) to[out=90, in=-90] (1, 1);
\draw[blue, thick] (1, 1) to[out=90, in=-90] (0, 2);
\draw[white, line width=5] (0, 1) to[out=90, in=-90] (1, 2);
\draw[blue, thick] (0, 1) to[out=90, in=-90] (1, 2);
\draw[blue, thick] (1, 2) to[out=90, in=-90] (0, 3);
\draw[white, line width=5] (0, 2) to[out=90, in=-90] (1, 3);
\draw[blue, thick] (0, 2) to[out=90, in=-90] (1, 3);
\draw[blue, thick, ->] (0, 3) -- (0, 4);
\draw[blue, thick] (1, 3) to[out=90, in=90] (2, 3) -- (2, 0) to[out=-90, in=-90] (1, 0);
\draw[white, line width=5] (-0.5, 1.5) to[out=-90, in=-90] (1.5, 1.5); 
\draw[red, ->] (-0.5, 1.5) to[out=-90, in=-90] (1.5, 1.5); 

\node[red, left] at (-0.5, 1.5){$W_n$};
\node[blue, left] at (0, -0.5){$\mathbf{3}_1^r$};
\end{tikzpicture}
}}
\]
\caption{The trefoil knot with Wilson line defect $W_n$}
\label{fig:trefoil}
\end{figure}
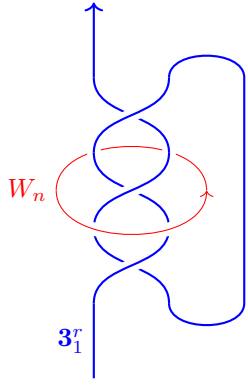
Let $\widehat{Z}_{S^3 \setminus \mathbf{3}_1^r, W_n}(x, q)$ denote the two-variable series associated to the complement of the right-handed trefoil knot, with an insertion of the Wilson line defect $W_n$ colored by $V_n$ that goes around the braid $\sigma_1^3$; see Figure \ref{fig:trefoil}. 
Then, we have
\begin{align*}
&\widehat{Z}_{S^3 \setminus \mathbf{3}_1^r , W_1}(x, q) = \widehat{Z}_{S^3 \setminus \mathbf{3}_1^r}(x, q)\\ 
&\;= -q x^{\frac12} + q^2 x^{\frac52} + q^3 x^{\frac72} -q^6 x^{\frac{11}{2}} -q^8 x^{\frac{13}{2}} +q^{13} x^{\frac{17}{2}} +q^{16} x^{\frac{19}{2}} + O(x^{\frac{21}{2}}),\\
&\widehat{Z}_{S^3 \setminus \mathbf{3}_1^r, W_2}(x, q)\\ 
&\;= -q^{\frac{1}{2}} x^{-\frac12} +q^{\frac{3}{2}} x^{\frac52} +q^{\frac{5}{2}} x^{\frac72} -q^{\frac{11}{2}} x^{\frac{11}{2}} -q^{\frac{15}{2}} x^{\frac{13}{2}} +q^{\frac{25}{2}} x^{\frac{17}{2}} +q^{\frac{31}{2}} x^{\frac{19}{2}} +O(x^{\frac{21}{2}})\\
&\;= q^{\frac{1}{2}} (x^{\frac12}-x^{-\frac12}) + q^{-\frac{1}{2}} \widehat{Z}_{S^3 \setminus \mathbf{3}_1^r}(x, q),\\
&\widehat{Z}_{S^3 \setminus \mathbf{3}_1^r, W_3}(x, q) = x^{\frac32}-x^{-\frac32},\\
&\widehat{Z}_{S^3 \setminus \mathbf{3}_1^r, W_4}(x, q)\\ 
&\;= -q^{-\frac{1}{2}} x^{-\frac52} + q^{-\frac{3}{2}} x^{\frac12} -q^{\frac{1}{2}} x^{\frac72} +q^{\frac{7}{2}} x^{\frac{11}{2}} +q^{\frac{11}{2}} x^{\frac{13}{2}} -q^{\frac{21}{2}} x^{\frac{17}{2}} -q^{\frac{27}{2}} x^{\frac{19}{2}} +O(x^{\frac{21}{2}})\\
&\;= q^{-\frac{1}{2}}(x^{\frac52}-x^{-\frac52}) -q^{-\frac{5}{2}}\widehat{Z}_{S^3 \setminus \mathbf{3}_1^r}(x, q),\\
&\widehat{Z}_{S^3 \setminus \mathbf{3}_1^r, W_5}(x, q)\\ 
&\;= -q^{-1} x^{-\frac72} +q^{-3} x^{-\frac12} -q^{-2} x^{\frac52} +q^{2} x^{\frac{11}{2}} +q^{4} x^{\frac{13}{2}} -q^{9} x^{\frac{17}{2}} -q^{12} x^{\frac{19}{2}} +O(x^{\frac{21}{2}})\\
&\;= -q^{-3}(x^{\frac12}-x^{-\frac12}) +q^{-1}(x^{\frac72}-x^{-\frac72}) -q^{-4} \widehat{Z}_{S^3 \setminus \mathbf{3}_1^r}(x, q),
\end{align*}
and so on. 
That is, 
\begin{align*}
\frac{\widehat{Z}_{S^3 \setminus \mathbf{3}_1^r, W_1}(x, q)}{x^{\frac12}-x^{-\frac12}} &= \frac{\widehat{Z}_{S^3 \setminus \mathbf{3}_1^r}(x, q)}{x^{\frac12}-x^{-\frac12}},\\
q^{\frac{1}{2}} \frac{\widehat{Z}_{S^3 \setminus \mathbf{3}_1^r, W_2}(x, q)}{x^{\frac12}-x^{-\frac12}} &= q\chi_1 + \frac{\widehat{Z}_{S^3 \setminus \mathbf{3}_1^r}(x, q)}{x^{\frac12}-x^{-\frac12}},\\
q \frac{\widehat{Z}_{S^3 \setminus \mathbf{3}_1^r, W_3}(x, q)}{x^{\frac12}-x^{-\frac12}} &= q \chi_3,\\
-q^{\frac{5}{2}} \frac{\widehat{Z}_{S^3 \setminus \mathbf{3}_1^r, W_4}(x, q)}{x^{\frac12}-x^{-\frac12}} &= -q^2 \chi_5 + \frac{\widehat{Z}_{S^3 \setminus \mathbf{3}_1^r}(x, q)}{x^{\frac12}-x^{-\frac12}},\\
-q^{4} \frac{\widehat{Z}_{S^3 \setminus \mathbf{3}_1^r, W_5}(x, q)}{x^{\frac12}-x^{-\frac12}} &= q\chi_1 -q^3\chi_7 + \frac{\widehat{Z}_{S^3 \setminus \mathbf{3}_1^r}(x, q)}{x^{\frac12}-x^{-\frac12}},\\
-q^{\frac{11}{2}} \frac{\widehat{Z}_{S^3 \setminus \mathbf{3}_1^r, W_6}(x, q)}{x^{\frac12}-x^{-\frac12}} &= q \chi_3 -q^{4}\chi_9,\\
q^{8} \frac{\widehat{Z}_{S^3 \setminus \mathbf{3}_1^r, W_7}(x, q)}{x^{\frac12}-x^{-\frac12}} &= -q^2\chi_5 +q^6\chi_{11} + \frac{\widehat{Z}_{S^3 \setminus \mathbf{3}_1^r}(x, q)}{x^{\frac12}-x^{-\frac12}},\\
q^{\frac{21}{2}} \frac{\widehat{Z}_{S^3 \setminus \mathbf{3}_1^r, W_8}(x, q)}{x^{\frac12}-x^{-\frac12}} &= q\chi_1 -q^3\chi_7 +q^8\chi_{13} + \frac{\widehat{Z}_{S^3 \setminus \mathbf{3}_1^r}(x, q)}{x^{\frac12}-x^{-\frac12}},\\
q^{\frac{26}{2}} \frac{\widehat{Z}_{S^3 \setminus \mathbf{3}_1^r, W_9}(x, q)}{x^{\frac12}-x^{-\frac12}} &= q\chi_3 -q^4\chi_9 +q^{10}\chi_{15}.
\end{align*}
Curiously, $\widehat{Z}_{S^3 \setminus \mathbf{3}_1^r, W_{3n}}(x, q)$ all seem to be Laurent polynomials. 

The above computation seems to suggest that the two-variable series module $V_{x,q}(S^3 \setminus \mathbf{3}_1^r)$, as an $R_{\chi}$-module, is spanned by $1$ and $\frac{\widehat{Z}_{S^3 \setminus \mathbf{3}_1^r}(x, q)}{x^{\frac12}-x^{-\frac12}}$; in fact, free of rank 2, which is consistent with an old result of \cite{BullockLoFaro} on skein modules of complements of twist knots: 
\[
V_{x,q}(S^3 \setminus \mathbf{3}_1^r) = R_{\chi} \bigg\{ 1, \frac{\widehat{Z}_{S^3 \setminus \mathbf{3}_1^r}(x, q)}{x^{\frac12}-x^{-\frac12}} \bigg\}.
\]
Note, $2$ is exactly the number of branches of the A-polynomial of $\mathbf{3}_1^r$. 
\end{eg}

\begin{eg}[Figure-eight knot]\label{eg:figure-eight}
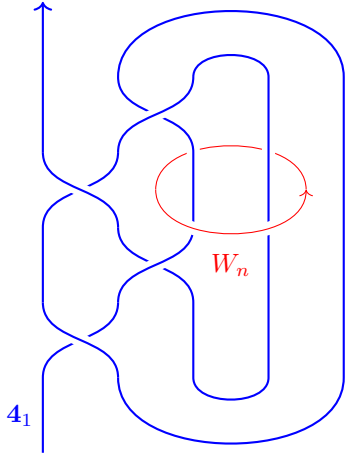
\begin{figure}
\centering
\[
\vcenter{\hbox{
\begin{tikzpicture}
\draw[red] (1.5, 2.5) to[out=90, in=90] (3.5, 2.5); 

\draw[white, line width=5] (0, -1) -- (0, 0) to[out=90, in=-90] (1, 1) to[out=90, in=-90] (2, 2) -- (2, 3) to[out=90, in=-90] (1, 4) to[out=90, in=90] (4, 4) -- (4, 0) to[out=-90, in=-90] (1, 0) to[out=90, in=-90] (0, 1) -- (0, 2) to[out=90, in=-90] (1, 3) to[out=90, in=-90] (2, 4) to[out=90, in=90] (3, 4) -- (3, 0) to[out=-90, in=-90] (2, 0) -- (2, 1) to[out=90, in=-90] (1, 2) to[out=90, in=-90] (0, 3) -- (0, 4) -- (0, 5); 
\draw[blue, thick, ->] (0, -1) -- (0, 0) to[out=90, in=-90] (1, 1) to[out=90, in=-90] (2, 2) -- (2, 3) to[out=90, in=-90] (1, 4) to[out=90, in=90] (4, 4) -- (4, 0) to[out=-90, in=-90] (1, 0) to[out=90, in=-90] (0, 1) -- (0, 2) to[out=90, in=-90] (1, 3) to[out=90, in=-90] (2, 4) to[out=90, in=90] (3, 4) -- (3, 0) to[out=-90, in=-90] (2, 0) -- (2, 1) to[out=90, in=-90] (1, 2) to[out=90, in=-90] (0, 3) -- (0, 4) -- (0, 5); 
\draw[white, line width=5] (1, 0) to[out=90, in=-90] (0, 1);
\draw[blue, thick] (1, 0) to[out=90, in=-90] (0, 1);
\draw[white, line width=5] (1, 2) to[out=90, in=-90] (0, 3);
\draw[blue, thick] (1, 2) to[out=90, in=-90] (0, 3);
\draw[white, line width=5] (1, 1) to[out=90, in=-90] (2, 2);
\draw[blue, thick] (1, 1) to[out=90, in=-90] (2, 2);
\draw[white, line width=5] (1, 3) to[out=90, in=-90] (2, 4);
\draw[blue, thick] (1, 3) to[out=90, in=-90] (2, 4);

\draw[white, line width=5] (1.5, 2.5) to[out=-90, in=-90] (3.5, 2.5); 
\draw[red, ->] (1.5, 2.5) to[out=-90, in=-90] (3.5, 2.5); 

\node[red] at (2.5, 1.5){$W_n$};
\node[blue, left] at (0, -0.5){$\mathbf{4}_1$};
\end{tikzpicture}
}}
\]
\caption{The figure-eight knot with Wilson line defect $W_n$}
\label{fig:figure-eight}
\end{figure}
From the result of \cite{BullockLoFaro} on skein modules of complements of twist knots, $V_{x,q}(S^3 \setminus \mathbf{4}_1)$ must be a free $R_{\chi}$-module of rank 3, which we study below. 
Note, $3$ is precisely the number of branches of the A-polynomial of $\mathbf{4}_1$. 

Let $\widehat{Z}_{S^3 \setminus \mathbf{4}_1, W_n}(x, q)$ denote the two-variable series associated to the complement of the figure-eight knot, with an insertion of the Wilson line on the ``knot $y$'' in \cite{BullockLoFaro} colored by $V_n$; see Figure \ref{fig:figure-eight}.  
Then, 
\begin{align*}
&\widehat{Z}_{S^3 \setminus \mathbf{4}_1, W_1}(x, q)= \widehat{Z}_{S^3 \setminus \mathbf{4}_1}(x, q) \\
&= x^{\frac12}\\
&\; + 2 x^{\frac32}\\
&\; + (q^{-1} + 3 + q) x^{\frac52}\\
&\; + (2q^{-2} +2q^{-1} + 5 + 2q +2q^2) x^{\frac72}\\
&\; + (q^{-4} +3q^{-3} +4q^{-2} +5q^{-1} + 8 + 5q +4q^2 +3q^3 +q^4) x^{\frac92}\\
&\; + O(x^{\frac{11}{2}}),\\
&\widehat{Z}_{S^3 \setminus \mathbf{4}_1, W_2}(x, q) \\
&= 
2q^{\frac{1}{2}} x^{\frac12}\\
&\; + q^{\frac{1}{2}}(3+q) x^{\frac32}\\
&\; + q^{\frac{1}{2}}(q^{-1} + 5 + 2q +2q^2) x^{\frac52}\\
&\; + q^{\frac{1}{2}}(2q^{-2} +3q^{-1} + 8 + 5q +4q^2 +3q^3 + q^4) x^{\frac72}\\
&\; + q^{\frac{1}{2}}(q^{-4} +3q^{-3} +6q^{-2} +7q^{-1} + 14 + 10q +10q^2 +7q^3 +6q^4 +2q^5 +2q^6) x^{\frac92}\\
&\; + O(x^{\frac{11}{2}}),\\
&\widehat{Z}_{S^3 \setminus \mathbf{4}_1, W_3}(x, q) \\
&=
q(2+q) x^{\frac12}\\
&\; + q(2+2q+2q^2) x^{\frac32}\\
&\; + q(3+4q+4q^2+3q^3+q^4) x^{\frac52}\\
&\; + q(q^{-1} +6 +7q +8q^2 +7q^3 +6q^4 +2q^5 +2q^6) x^{\frac72}\\
&\; + q(2q^{-2} +5q^{-1} +11 +14q +15q^2 +15q^3 +14q^4 +11q^5 +7q^6 +4q^7 +3q^8 +q^9) x^{\frac92}\\
&\; + O(x^{\frac{11}{2}}). 
\end{align*}
The two-variable series module $V_{x,q}(S^3 \setminus \mathbf{4}_1)$ must be the free $R_{\chi}$-module spanned by these three two-variable series, divided by $x^{\frac12}-x^{-\frac12}$. 

We observe the following relations among these two-variable series:\footnote{Recall that $\mathbf{4}_1$ is amphichiral, so the mirror of $\mathbf{4}_1$ is itself, while the defect $W$ will map to a different defect. 
The $\widehat{Z}_{S^3 \setminus \mathbf{4}_1, W_2}(x, q^{-1})$ below is really the BPS $q$-series for the insertion of this mirror defect.}
\[
\chi_3 \widehat{Z}_{S^3 \setminus \mathbf{4}_1}(x, q) = 
q^{-\frac{1}{2}} \widehat{Z}_{S^3 \setminus \mathbf{4}_1, W_2}(x, q) +
q^{\frac{1}{2}} \widehat{Z}_{S^3 \setminus \mathbf{4}_1, W_2}(x, q^{-1}),
\]
\begin{align*}
&-q^{-1} \widehat{Z}_{S^3 \setminus \mathbf{4}_1, W_3}(x, q) + q^{-\frac{1}{2}} \chi_3 \widehat{Z}_{S^3 \setminus \mathbf{4}_1, W_2}(x, q) \\
&\quad= 
\widehat{Z}_{S^3 \setminus \mathbf{4}_1}(x, q) + (q^{-\frac{1}{2}} \widehat{Z}_{S^3 \setminus \mathbf{4}_1, W_2}(x, q) + q^{\frac{1}{2}} \widehat{Z}_{S^3 \setminus \mathbf{4}_1, W_2}(x, q^{-1}) - 2).
\end{align*}
Therefore, a basis for $V_{x,q}(S^3 \setminus \mathbf{4}_1)$, as a free $R_{\chi}$-module of rank $3$, can be taken to be
$1$, $\frac{\widehat{Z}_{S^3 \setminus \mathbf{4}_1}(x, q)}{x^{\frac12}-x^{-\frac12}}$, and $\frac{\widehat{Z}_{S^3 \setminus \mathbf{4}_1, W_2}(x, q)}{x^{\frac12}-x^{-\frac12}}$: 
\[
V_{x,q}(S^3 \setminus \mathbf{4}_1) = R_{\chi} 
\bigg\{ 1, 
\frac{\widehat{Z}_{S^3 \setminus \mathbf{4}_1}(x,q)}{x^{\frac12}-x^{-\frac12}}, 
\frac{\widehat{Z}_{S^3 \setminus \mathbf{4}_1, W_2}(x,q)}{x^{\frac12}-x^{-\frac12}}
\bigg\}.
\]

It turns out that the two-variable series above have simple expressions in terms of \emph{inverted Habiro series} \cite{Park_inverted}: 
\begin{align*}
\frac{\widehat{Z}_{S^3 \setminus \mathbf{4}_1}(x,q)}{x^{\frac12}-x^{-\frac12}} &= -\sum_{n\geq 0}\frac{1}{\prod_{0\leq j\leq n}(x+x^{-1} - q^{j} - q^{-j})},\\
\frac{\widehat{Z}_{S^3 \setminus \mathbf{4}_1, W_2}(x,q)}{x^{\frac12}-x^{-\frac12}} &= -\sum_{n\geq 0}\frac{q^{\frac{1}{2}}(1+q^n)}{\prod_{0\leq j\leq n}(x+x^{-1} - q^{j} - q^{-j})},\\
\frac{\widehat{Z}_{S^3 \setminus \mathbf{4}_1, W_3}(x,q)}{x^{\frac12}-x^{-\frac12}} &= -\sum_{n\geq 0}\frac{q^{n+1}(1+q+q^n)}{\prod_{0\leq j\leq n}(x+x^{-1} - q^{j} - q^{-j})}. 
\end{align*}
Hence, another way to describe $V_{x,q}(S^3 \setminus \mathbf{4}_1)$ is that it is spanned by three sequences of inverted Habiro coefficients, given by $\{1\}_{n\geq 0}$, $\{q^n\}_{n\geq 0}$, and $\{q^{2n}\}_{n\geq 0}$. 
\end{eg}

The appearance of $1$ in the two-variable series modules in all three examples discussed above naturally leads to the following question: 
\begin{qn}\label{qn:1-in-Vxq}
Is $1$ always in the two-variable series module $V_{x,q}(S^3 \setminus K)$? 
In other words, is there always an element $W \in \Sk_q^{\mathrm{SL}_2}(S^3 \setminus K)$ of the skein module such that the colored Jones polynomial $\frac{1}{[n]} J_{(K,n),W}(q)$ is always $1$?
\end{qn}

\section{Surgery}\label{sec:surgery}
In this section, we show how to perform surgery to get an analogous map from the skein module of a closed 3-manifold -- which is a finite-dimensional $\mathbb{Q}(q^{\frac14})$-vector space \cite{GunninghamJordanSafronov} -- to some vector space of $q$-series. 

\subsection{Twisted spin$^c$ structures}\label{subsec:twisted-spin-c}
Before getting into the proof, let us first define twisted spin$^c$ structures that appear in the statement of Theorem \ref{mainthm:surgery}. 
Recall that isomorphism classes of principal $SO(3)$-bundles over a $3$-manifold $Y$ are classified by their $w_2 \in H^2(Y;\mathbb{Z}/2)$. 
For any $\alpha \in H_1(Y;\mathbb{Z}/2)$, one way to construct a principal $SO(3)$-bundle $P$ over $Y$ with $w_2(P) = \mathrm{PD}(\alpha)$ is by twisting the $SO(3)$-frame bundle $Fr$ in a tubular neighborhood of a 1-cycle $C_\alpha \subset Y$ representing $\alpha$, in such a way that the clutching map along the meridians of $C_\alpha$ represents the generator of $\pi_1(SO(3)) \cong \mathbb{Z}/2$. 
Let us denote the resulting $SO(3)$-bundle by $Fr(C_\alpha)$. 

\begin{defn}
A \emph{$C_\alpha$-twisted spin$^c$ structure on $Y$} is a lift of $Fr(C_\alpha)$ to a principal $\mathrm{Spin}^c(3)$-bundle.  
\end{defn}
Clearly, when $C_\alpha = 0$, this is nothing but an ordinary spin$^c$ structure. 
We denote the set of $C_\alpha$-twisted spin$^c$ structures on $Y$ by $\mathrm{Spin}^c(Y,C_\alpha)$. 
Just like the usual spin$^c$ structures, $\mathrm{Spin}^c(Y,C_\alpha)$ is an $H^2(Y;\mathbb{Z}) \cong H_1(Y;\mathbb{Z})$-torsor. 
If $C'_\alpha \subset Y$ is another $1$-cycle representing the same class $\alpha \in H_1(Y;\mathbb{Z}/2)$, then any $2$-chain $D\subset Y$ with $\partial D = C_\alpha - C'_\alpha$ induces an isomorphism of $H_1(Y;\mathbb{Z})$-torsors
\[
\Phi_D : \mathrm{Spin}^c(Y,C_\alpha) \cong \mathrm{Spin}^c(Y,C'_\alpha), 
\]
which is compatible with composition of $2$-chains. 
If $D_1$ and $D_2$ are both $2$-chains from $C_\alpha$ to $C'_\alpha$, then $[D_1-D_2] \in H_2(Y;\mathbb{Z}/2)$, and
\[
\Phi_{D_1} - \Phi_{D_2} = \beta([D_1-D_2]),
\]
where $\beta : H_2(Y;\mathbb{Z}/2) \rightarrow H_1(Y;\mathbb{Z})$ is the Bockstein. 
We will sometimes simply write $\mathrm{Spin}^c(Y,\alpha)$ to mean $\mathrm{Spin}^c(Y,C_\alpha)$ for some choice of $C_\alpha$, and call its elements \emph{$\alpha$-twisted spin$^c$ structures}. 
From the discussion above, without choosing $C_\alpha$, $\mathrm{Spin}^c(Y,\alpha)$ is canonically defined only up to gauge transformation (i.e., translation by elements of order $2$ in $H_1(Y;\mathbb{Z})$). 

See Appendix \ref{sec:G-spin-c} for the analogous story for an arbitrary semisimple gauge group $G$. 

\begin{rmk}
There is a useful characterization of $\mathrm{Spin}^c(Y,\alpha)$ when $Y$ is presented as a surgery on a framed link $L = L_1 \sqcup \cdots \sqcup L_s \subset S^3$ with linking matrix $M$. 
Recall from \cite[Sec. 4.2]{GukovManolescu} that
\[
\mathrm{Spin}^c(Y) 
\cong 
\frac{m + 2\mathbb{Z}^s}{2M \mathbb{Z}^s} 
\cong 
\frac{\delta + 2\mathbb{Z}^s}{2M\mathbb{Z}^s},
\]
where $m$ is the framing vector ($m_i = M_{ii}$), and
$\delta$ is the degree vector ($\delta_i = \sum_{j\neq i} M_{ij}$). 
The first isomorphism can be seen by considering the $4$-manifold $W$ with $\partial W = Y$ given by the surgery trace: 
The restriction map $\mathrm{Spin}^c(W) \rightarrow \mathrm{Spin}^c(Y)$ is surjective, 
while the map $\mathrm{Spin}^c(W) \overset{c_1}{\rightarrow} H^2(W;\mathbb{Z}) \cong \mathbb{Z}^s$ given by the first Chern class is injective, with the image consisting exactly of vectors whose mod-$2$ reduction match with $w_2(TW)$; i.e., the characteristic vectors $m+2\mathbb{Z}^s \subset \mathbb{Z}^s \cong H^2(W;\mathbb{Z})$. 
The second isomorphism is given by adding $Mu = m + \delta$, where $u = (1,\cdots,1)$. 

We can describe $\mathrm{Spin}^c(Y,\alpha)$ in a similar way. 
Given a 1-cycle $C_\alpha \subset S^3 \setminus L \subset Y$ representing $\alpha \in H_1(Y;\mathbb{Z}/2)$, we can choose a relative 2-cycle $D_\eta \subset W$ by capping $C_\alpha$ off using meridional disks, representing a class $\eta \in H_2(W,Y;\mathbb{Z}/2)$ with $\partial\eta = \alpha$. 
Then, if $\Sigma_i \in H_2(W;\mathbb{Z})$ is the capped core of the $i$-th 2-handle, 
\[
\mathrm{PD}(\eta)(\Sigma_i) = \Sigma_i \cdot \eta = \mathrm{lk}(C_\alpha, L_i)\;\textrm{mod } 2.
\]
Let $P_\eta$ be the (isomorphism class of) $SO(4)$-bundle over $W$ which is identified with the usual frame bundle outside a tubular neighborhood of $D_\eta$ but glued across $D_\eta$ in such a way that the clutching map along the meridians of $D_\eta$ represents the generator of $\pi_1(SO(4)) \cong \mathbb{Z}/2$. 
Then, 
\[
w_2(P_\eta) = w_2(TW) + \mathrm{PD}(\eta).
\]
Let $\mathrm{Spin}^c(W,\eta)$ be the set of (isomorphism classes of) lifts of $P_\eta$ to a $\mathrm{Spin}^c(4)$ bundle. 
Then, the restriction map $\mathrm{Spin}^c(W,\eta) \rightarrow \mathrm{Spin}^c(Y,\alpha)$ is surjective, while the first Chern class map
$\mathrm{Spin}^c(W,\eta) \overset{c_1}{\rightarrow} H^2(W;\mathbb{Z}) \cong \mathbb{Z}^s$
is injective, with the image consisting exactly of vectors whose mod-$2$ reduction match with $w_2(TW) + \mathrm{PD}(\eta)$. 
It follows that
\[
\mathrm{Spin}^c(Y, \alpha) 
\cong 
\frac{t(C_\alpha) + m + 2\mathbb{Z}^s}{2M\mathbb{Z}^s}
\cong 
\frac{t(C_\alpha) + \delta + 2\mathbb{Z}^s}{2M\mathbb{Z}^s},
\]
where $t(C_\alpha)$ is the vector with $t(C_\alpha)_i := \mathrm{lk}(C_\alpha, L_i)\;\textrm{mod } 2$. 
\end{rmk}

\subsection{Proof of Theorem \ref{mainthm:surgery}}
We need to show\footnote{That checking the handle slide invariance is sufficient follows from the standard handle-sliding lemma. See, e.g., \cite[Prop. 2.2 (5)]{Przytycki}.} that, if $L' \subset S^3 \setminus K$ is obtained from $L \subset S^3 \setminus K$ by handle-sliding a component of $L$ over $K$ thought of as a $p$-framed knot, then
\[
\mathcal{L}^{(p)} \qty[
(x^{\frac12}-x^{-\frac12})
\widehat{Z}_{S^3 \setminus K, L}(x,q)
]
=
\mathcal{L}^{(p)} \qty[
(x^{\frac12}-x^{-\frac12})
\widehat{Z}_{S^3 \setminus K, L'}(x,q)
].
\]
Here, $\mathcal{L}^{(p)}$ is the $p$-surgery Laplace transform \cite[Eqn.~1]{GukovManolescu}, which is a $\mathbb{Z}[q^{\pm \frac14}]$-linear map defined on each monomial in $x$ by
\begin{equation}\label{eqn:Laplace}
\mathcal{L}^{(p)} : x^u \mapsto q^d q^{-\frac{1}{p}u^2}, 
\end{equation}
where $d \in \mathbb{Q}$ is determined by $p$ and is independent of $u$; see Appendix \ref{sec:gluing}. 
It suffices to consider the case when the defect $L$ is colored by $V_2$.\footnote{This is because $V_N$-colored defect on $L$ is equivalent to a linear combination of $V_2$-colored defects on cablings of $L$.} 

A choice of an $\alpha$-twisted spin$^c$ structure $\mathfrak{s} \in \mathrm{Spin}^c(S^3_p(K),\alpha) \cong (\alpha + \mathbb{Z})/p\mathbb{Z}$ (with $\alpha = 0$ if $p$ is odd, and $\alpha \in \frac12 \mathbb{Z}/\mathbb{Z}$ if $p$ is even) simply corresponds to restricting to monomials $x^u$ corresponding to a fixed congruence class $\mathfrak{s} = u \;(\textrm{mod }p)$, and hence we suppress $\mathfrak{s}$ from the notation in this discussion. 

Firstly, note that we have the following identity:
\begin{lem}\label{lem:p-surgery-identity}
For any $u \in \frac{1}{2}\mathbb{Z}$ and $\epsilon \in \{\pm\}$, 
\begin{align}\label{eq:Laplace-identity}
&\mathcal{L}^{(p)}\qty[
(x^{\frac12} - x^{-\frac12})\cdot
(q^{\frac{\epsilon}{2}}\hat{x}^{\frac12} - q^{-\frac{\epsilon}{2}}\hat{x}^{-\frac12})\hat{y}^{\epsilon}\, x^u
]
\\
&=
\mathcal{L}^{(p)}\qty[
(x^{\frac12} - x^{-\frac12})\cdot 
(q^{-\frac{\epsilon}{2}}\hat{x}^{\frac12} - q^{\frac{\epsilon}{2}} \hat{x}^{-\frac12})
(q^{\frac14} \hat{x}^{-\frac{\epsilon}{2}})^{p}
x^u
]. \nonumber
\end{align}
\end{lem}
\begin{proof}
Straightforward calculation. 
\end{proof}

It follows that, using the fusion rules reviewed in Appendix \ref{sec:trivalent-vertices}, we have 
\begin{align*}
\vcenter{\hbox{
\begin{tikzpicture}
\draw[thick, blue, ->] (0, 0) -- (0, 2);
\draw[red] (-1, 0.5) to[out=0, in=-90] (-0.5, 1) to[out=90, in=0] (-1, 1.5);
\node[blue, below] at (0, 0){$V_\infty(x)$};
\node[blue, above] at (0, 2){$V_\infty(x)$};
\end{tikzpicture}
}}
&\overset{\eqref{eq:fusion}}{=}
\sum_{\epsilon \in \{\pm\}}
-(q^{\frac{\epsilon}{2}}x^{\frac{1}{2}} - q^{-\frac{\epsilon}{2}}x^{-\frac{1}{2}})
\vcenter{\hbox{
\begin{tikzpicture}
\draw[thick, blue, ->] (0, 0) -- (0, 2);
\draw[red] (-1, 0.5) -- (0, 0.5);
\draw[red] (-1, 1.5) -- (0, 1.5);
\draw[thick, cyan] (0, 0.5) -- (0, 1.5);
\node[blue, below] at (0, 0){$V_\infty(x)$};
\node[blue, above] at (0, 2){$V_\infty(x)$};
\node[cyan, right] at (0, 1){$V_\infty(q^{\epsilon}x)$};
\end{tikzpicture}
}}
\\
&= 
\sum_{\epsilon \in \{\pm\}}
-(q^{\frac{\epsilon}{2}}x^{\frac{1}{2}} - q^{-\frac{\epsilon}{2}}x^{-\frac{1}{2}})
\hat{y}^{\epsilon}
\vcenter{\hbox{
\begin{tikzpicture}
\draw[thick, blue, ->] (0, 0) -- (0, 2);
\draw[red] (-1, 0.5) -- (0, 0.5);
\draw[red] (-1, 1.5) -- (0, 1.5);
\draw[thick, teal] (0, 0) -- (0, 0.5);
\draw[thick, teal, ->] (0, 1.5) -- (0, 2);
\node[teal, below] at (0, 0){$V_\infty(q^{-\epsilon}x)$};
\node[teal, above] at (0, 2){$V_\infty(q^{-\epsilon}x)$};
\node[blue, right] at (0, 1){$V_\infty(x)$};
\end{tikzpicture}
}}
\\
&\overset{\eqref{eq:Laplace-identity}}{\cong_p}
\sum_{\epsilon \in \{\pm\}}
-(q^{-\frac{\epsilon}{2}}x^{\frac{1}{2}} - q^{\frac{\epsilon}{2}}x^{-\frac{1}{2}})
(q^{\frac{1}{4}}x^{-\frac{\epsilon}{2}})^p
\vcenter{\hbox{
\begin{tikzpicture}
\draw[thick, blue, ->] (0, 0) -- (0, 2);
\draw[red] (-1, 0.5) -- (0, 0.5);
\draw[red] (-1, 1.5) -- (0, 1.5);
\draw[thick, teal] (0, 0) -- (0, 0.5);
\draw[thick, teal, ->] (0, 1.5) -- (0, 2);
\node[teal, below] at (0, 0){$V_\infty(q^{-\epsilon}x)$};
\node[teal, above] at (0, 2){$V_\infty(q^{-\epsilon}x)$};
\node[blue, right] at (0, 1){$V_\infty(x)$};
\end{tikzpicture}
}}
\\
&\overset{\eqref{eq:eigen-value-eq}}{=} 
\sum_{\epsilon \in \{\pm\}}
-(q^{-\frac{\epsilon}{2}}x^{\frac{1}{2}} - q^{\frac{\epsilon}{2}}x^{-\frac{1}{2}})
\vcenter{\hbox{
\begin{tikzpicture}
\draw[red] (0.1, 0.9) to[out=90, in=-90] (-0.1, 1.05);
\draw[red] (0.1, 1.2) to[out=90, in=-90] (-0.1, 1.35);
\draw[white, line width=3] (0, 0) -- (0, 2);
\draw[thick, blue, ->] (0, 0) -- (0, 2);
\draw[red] (-1, 0.5) -- (0, 0.5);
\draw[red] (-1, 1.5) to[out=0, in=180] (-0.1, 0.75);
\draw[white, line width=3] (-0.1, 0.75) to[out=0, in=-90] (0.1, 0.9);
\draw[red] (-0.1, 0.75) to[out=0, in=-90] (0.1, 0.9);
\draw[white, line width=3] (-0.1, 1.05) to[out=90, in=-90] (0.1, 1.2);
\draw[red] (-0.1, 1.05) to[out=90, in=-90] (0.1, 1.2);
\draw[red] (-0.1, 1.35) to[out=90, in=180] (0, 1.5);
\draw[thick, teal] (0, 0) -- (0, 0.5);
\draw[thick, teal, ->] (0, 1.5) -- (0, 2);
\node[teal, below] at (0, 0){$V_\infty(q^{-\epsilon}x)$};
\node[teal, above] at (0, 2){$V_\infty(q^{-\epsilon}x)$};
\node[blue, right] at (0.3, 1){$V_\infty(x)$};
\node[red, right] at (0.1, 1.5){$p\text{ full twists}$};
\end{tikzpicture}
}}
\\
&\overset{\eqref{eq:fusion}}{=}
\vcenter{\hbox{
\begin{tikzpicture}
\draw[red] (0.1, 0.9) to[out=90, in=-90] (-0.1, 1.05);
\draw[red] (0.1, 1.2) to[out=90, in=-90] (-0.1, 1.35);
\draw[white, line width=3] (0, 0) -- (0, 2);
\draw[thick, blue, ->] (0, 0) -- (0, 2);
\draw[red] (-1, 0.5) -- (-0.2, 0.5) to[out=0, in=90] (-0.1, 0.4) -- (-0.1, 0);
\draw[red] (-1, 1.5) to[out=0, in=180] (-0.1, 0.75);
\draw[white, line width=3] (-0.1, 0.75) to[out=0, in=-90] (0.1, 0.9);
\draw[red] (-0.1, 0.75) to[out=0, in=-90] (0.1, 0.9);
\draw[white, line width=3] (-0.1, 1.05) to[out=90, in=-90] (0.1, 1.2);
\draw[red] (-0.1, 1.05) to[out=90, in=-90] (0.1, 1.2);
\draw[red] (-0.1, 1.35) -- (-0.1, 2);
\node[blue, right] at (0.3, 1){$V_\infty(x)$};
\node[red, right] at (0.1, 1.5){$p\text{ full twists}$};
\end{tikzpicture}
}},
\end{align*}
where $\cong_p$ means the two sides are equal once we apply the $p$-surgery formula. 
In the last equality, we are first sliding the top trivalent junction along $K$ -- using \eqref{eq:sliding}, see also Lemma~\ref{lem:inverted-sliding} -- until it is just below the bottom trivalent junction and then using \eqref{eq:fusion} to detach $L'$ from $K$. 
Therefore, our map is invariant under handle slide. 
\qed

\begin{rmk}
It should be possible to extend this argument to rational surgery by analyzing the fusion rules for torus braids. 
Moreover, one can likely adapt this proof to show well-definedness -- i.e., invariance under handle slides -- of $\widehat{Z}$ of closed 3-manifolds presented by surgery on a framed link, assuming convergence of the $q$-series. 
We plan to investigate this in future work. 
\end{rmk}

\begin{rmk}
In case of integer homology spheres, our map $\widehat{Z}$ can be thought of as the $q$-series analog of a recently constructed map \cite{GaroufalidisLe} from the skein module to the Habiro ring (i.e., functions defined near roots of unity). 
We expect that our map provides an analytic continuation of their map in an appropriate sense. 
\end{rmk}

\subsection{$R_{\chi}$-module structure vs surgery}\label{subsec:R_chi-action-surgery}
We have previously seen in Section \ref{subsec:two-variable-series-module-examples} that the two-variable series modules $V_{x,q}(S^3 \setminus K)$ are naturally modules over $R_{\chi}$, with the module structure given by an insertion of a Wilson line along the meridian. 
Here, we briefly discuss how the $R_{\chi}$-module structure interacts with the Laplace transform, focusing on the $-1$-surgery, as it will be used in the examples below. 

Recall from \cite[Prop.~4.1.5]{Park-thesis} that the $-1$-surgery Laplace transform acts on each inverted Habiro factors by
\[
\qty(\frac{1}{\prod_{1\leq j\leq n}(x+x^{-1}-q^{j}-q^{-j})})
\bigg\vert_{x^u \mapsto q^{u^2}} 
= \frac{q^{n^2}}{(q^{n+1};q)_n},
\]
where $(\cdots)\vert_{x^u \mapsto q^{u^2}}$ in the left-hand side means to first expand $(\cdots)$ into a power series in $x$ and then to replace each monomial $x^u$ by $q^{u^2}$. 
It turns out, there are similar identities even when the inverted Habiro factor is multiplied by an element of $R_\chi$: 
\begin{lem}\label{lem:R_chi-action-surgery}
There are polynomials $P_m^{\chi}(t) \in \mathbb{Z}[q^{\pm \frac14}][t]$ of degree $m-1$ such that
\[
\qty(\frac{\chi_{m}}{\prod_{1\leq j\leq n}(x+x^{-1}-q^{j}-q^{-j})})
\bigg\vert_{x^u \mapsto q^{u^2}}  = \frac{q^{n^2}}{(q^{n+1};q)_n} P_m^{\chi}(q^{-n}).
\]
In particular, 
\[
\qty(\frac{\chi_{m}}{\prod_{1\leq j\leq n}(x+x^{-1}-q^{j}-q^{-j})})
\bigg\vert_{x^u \mapsto q^{u^2}} 
\in 
\mathbb{Z}[q^{\pm \frac14}]
\left\{
\frac{q^{n^2-\lambda n}}{(q^{n+1};q)_n} \;\bigg\vert\; 0\leq \lambda \leq m-1
\right\},
\]
with the coefficients in $\mathbb{Z}[q^{\pm \frac14}]$ independent from $n$. 
\end{lem}
\begin{proof}
For $m=1$, this is \cite[Prop.~4.1.5]{Park-thesis}, and for $m=2$, this is \cite[Lem.~1]{CCKPS_3d-modularity-revisited}. 
For general $m$, it suffices to observe that there is a recurrence relation for $P_m^{\chi}(t)$. 
That is, if we set $H_n := \frac{1}{\prod_{1\leq j\leq n}(x+x^{-1}-q^{j}-q^{-j})}$, we have
\[
(\chi_{m+2} + \chi_{m-2})H_n = (x+x^{-1}) \chi_m H_n = (q^n + q^{-n})\chi_m H_{n} + \chi_m H_{n-1},
\]
and thus 
\[
P_{m+2}^{\chi}(t) + P_{m-2}^{\chi}(t) = 
(t+t^{-1})P_m^{\chi}(t) + 
(q t^2-1)(1+t^{-1}) P_m^{\chi}(qt).
\]
Note, the right-hand side of the recurrence is a polynomial in $t$, since the terms of $O(t^{-1})$ cancel out. 
\end{proof}
Explicitly, for the first few values of $m$, 
\begin{align*}
P_1^{\chi}(t) &= 1, \\
P_2^{\chi}(t) &= q^{\frac{1^2}{4}}(1+t), \\
P_3^{\chi}(t) &= q^{\frac{2^2}{4}}((1+q^{-1})t +t^2), \\
P_4^{\chi}(t) &= q^{\frac{3^2}{4}}(-q^{-1} +q^{-2}t +(1+q^{-1}+q^{-2})t^2 +t^3), \\
P_5^{\chi}(t) &= q^{\frac{4^2}{4}}
(-q^{-2} 
-(q^{-1}+q^{-2})t
+(q^{-2}+q^{-3}+q^{-4})t^2 \\
&\qquad\qquad\qquad\qquad
+(1+q^{-1}+q^{-2}+q^{-3})t^3
+t^4).
\end{align*}
Given a sequence of inverted Habiro coefficients $\{a_{-1-n}(q)\}_{n\geq 0}$, the new sequence formed by $\{q^{\lambda n} a_{-1-n}(q)\}_{n\geq 0}$ for some $\lambda \in \mathbb{Z}$ is sometimes called a \emph{descendant} of the original sequence. 
Hence, the observation above may be rephrased as: the $R_\chi$-action generates the descendant $q$-series.

\subsection{Examples}
For $Y = S^3_p(K)$ and a choice of $\alpha \in H_1(Y;\mathbb{Z}/2)$ and an $\alpha$-twisted spin$^c$ structure $\mathfrak{s} \in \mathrm{Spin}^c(Y, \alpha)$, let the \emph{$q$-series module} $V_q(Y)_{\mathfrak{s}}$ be the image of the skein module $\Sk_q^{\mathrm{SL}_2}(Y)_\alpha$ under $\widehat{Z}_\mathfrak{s}$. 
After tensoring with $\mathbb{Q}(q^{\frac14})$, this is always a finite-dimensional $\mathbb{Q}(q^{\frac14})$-vector space. 
In the examples below, we will consider integer homology spheres and omit $\mathfrak{s}$ from the notation since there is a unique choice.

\begin{eg}[$-1$-surgery on $\mathbf{3}_1^l$]
Let $Y = S^3_{-1}(\mathbf{3}_1^l) = \Sigma(2,3,5)$, the Poincar\'e homology sphere. 
It has $3$ $\mathrm{SL}_2(\mathbb{C})$-flat connections (all of which are $\mathrm{SU}(2)$-flat connections), and in fact
\[
\dim_{\mathbb{Q}(q^{\frac14})} \Sk_q^{\mathrm{SL}_2}(\Sigma(2,3,5)) = 3, 
\]
see, e.g., \cite{DKS_Seifert}. 
The left-handed trefoil is the mirror of the right-handed one, so its two-variable series module can be obtained from that of the right-handed one (Example \ref{eg:trefoil-two-variable-series-module}) by replacing $q$ by $q^{-1}$: 
\[
V_{x,q}(S^3 \setminus \mathbf{3}_1^l) = R_{\chi} \bigg\{ 1, \frac{\widehat{Z}_{S^3 \setminus \mathbf{3}_1^l}(x, q)}{x^{\frac12}-x^{-\frac12}} \bigg\}. 
\]
Multiplying them by $(x^{\frac12}-x^{-\frac12})^2$ and applying the $-1$-surgery Laplace transform, together with Lemma \ref{lem:R_chi-action-surgery}, we have
\begin{align*}
V_q(\Sigma(2,3,5)) &= \mathbb{Z}[q^{\pm\frac14}] \left\{ 2(1-q),\; \chi_0^{\lambda}(q^{-1}) := \sum_{n\geq 0} \frac{(-1)^{n}q^{\frac{n(3n-1)}{2} + \lambda n}}{(q^{n+1};q)_{n}} \;\bigg\vert\; \lambda \leq 0\right\},
\end{align*}
where $\chi_0^{\lambda}(q^{-1})$ are the descendants of the mirror ($q \leftrightarrow q^{-1}$) of the order $5$ mock theta function $\chi_0(q)$. 
Let $\hat{x}$ (resp., $\hat{y}$) be an operator acting on the descendants by multiplication by $q^{\lambda}$ (resp., by shifting $\lambda$ to $\lambda+1$), so that $\hat{y}\hat{x} = q \hat{x} \hat{y}$. 
Then, these descendants satisfy the following degree $3$ inhomogeneous $q$-difference equation
\begin{equation*}
-q\qty(
(1+\hat{y})(1-q^{-1} \hat{y}^2) + q \hat{x} \hat{y}^3
)\,
\chi_0^{\lambda}(q^{-1}) = 2(1-q), 
\end{equation*}
and the corresponding degree $4$ homogeneous $q$-difference equation
\begin{equation*}
\qty(
1 -  (1+q^{-1}) \hat{y}^2 + q \hat{x} \hat{y}^3 + (q^{-1}-q^2 \hat{x}) \hat{y}^4
)\,
\chi_0^{\lambda}(q^{-1}) = 0
\end{equation*}
obtained by acting by $(1-\hat{y})$ on both sides. 
Note that, when $\lambda = -3$, the above equation reduces to the following degree $3$ linear relation:
\[
\chi_0^{-3}(q^{-1}) - (1+q^{-1}) \chi_0^{-1}(q^{-1}) + q^{-2} \chi_0^{0}(q^{-1}) = 0.
\]
It follows that
\begin{align*}
V_q(\Sigma(2,3,5)) &= \mathbb{Z}[q^{\pm\frac14}] \left\{ \chi_0^{\lambda}(q^{-1}) \;\bigg\vert\; \lambda \leq 0 \right\} \\
&= \mathbb{Z}[q^{\pm\frac14}] \left\{ \chi_0^{\lambda_0}(q^{-1}), \chi_0^{\lambda_0 -1}(q^{-1}), \chi_0^{\lambda_0 -2}(q^{-1})\right\}
\end{align*}
for any $\lambda_0 \leq 0$, 
consistent with the fact that 
\[
\dim_{\mathbb{Q}(q^{\frac14})} V_q(\Sigma(2,3,5)) \leq \dim_{\mathbb{Q}(q^{\frac14})} \Sk_q^{\mathrm{SL}_2}(\Sigma(2,3,5)) =3.
\]
\end{eg}

\begin{eg}[$-1$-surgery on $\mathbf{3}_1^r$]\label{eg:Sigma(2,3,7)}
Let $Y = S^3_{-1}(\mathbf{3}_1^r) = \Sigma(2,3,7)$. 
It has $4$ $\mathrm{SL}_2(\mathbb{C})$-flat connections (with 3 of them being $\mathrm{SU}(2)$-flat connections and the remaining one being complex), and in fact 
\[
\dim_{\mathbb{Q}(q^{\frac14})} \Sk_q^{\mathrm{SL}_2}(\Sigma(2,3,7)) = 4, 
\]
see, e.g., \cite{DKS_Seifert}. 
Applying the $-1$-surgery Laplace transform to $V_{x,q}(S^3 \setminus \mathbf{3}_1^r)$ (Example \ref{eg:trefoil-two-variable-series-module}), we have
\[
V_q(\Sigma(2,3,7)) = \mathbb{Z}[q^{\pm \frac14}] \left\{ 
2(1-q),\; 
\mathcal{F}_0^{\lambda}(q^{-1}) := 
\sum_{n\geq 0} \frac{(-1)^n q^{\frac{n(n+1)}{2} + \lambda n}}{(q^{n+1};q)_n} 
\;\bigg\vert\; \lambda \leq 0
\right\},
\]
where $\mathcal{F}_0^{\lambda}(q^{-1})$ are the descendants of the mirror ($q \leftrightarrow q^{-1}$) of the order $7$ mock theta function $\mathcal{F}_0(q)$. 
They satisfy the following degree $3$ inhomogeneous $q$-difference equation
\[
-q((1+\hat{y})(1-q^{-1}\hat{y}^2) + q \hat{x} \hat{y})
\mathcal{F}_0^{\lambda}(q^{-1}) = 2(1-q),
\]
and it follows that
\begin{align*}
V_q(\Sigma(2,3,7)) &= \mathbb{Z}[q^{\pm \frac14}]
\left\{ 
\mathcal{F}_0^{\lambda}(q^{-1})
\;\bigg\vert\; \lambda \in \mathbb{Z}
\right\} \\
&= 
\mathbb{Z}[q^{\pm\frac14}] \left\{ \mathcal{F}_0^{\lambda_0}(q^{-1}), \mathcal{F}_0^{\lambda_0 +1}(q^{-1}), \mathcal{F}_0^{\lambda_0 +2}(q^{-1}), \mathcal{F}_0^{\lambda_0 +3}(q^{-1})\right\}
\end{align*}
for any $\lambda_0 \in \mathbb{Z}$, consistent with the fact that 
\[
\dim_{\mathbb{Q}(q^{\frac14})} V_q(\Sigma(2,3,7)) \leq \dim_{\mathbb{Q}(q^{\frac14})} \Sk_q^{\mathrm{SL}_2}(\Sigma(2,3,7)) =4.
\]
\end{eg}

\begin{eg}[$-1$-surgery on $\mathbf{4}_1$]
Let $Y = S^3_{-1}(\mathbf{4}_1) = -\Sigma(2,3,7)$, which is the orientation reversal of $\Sigma(2,3,7)$ studied in Example \ref{eg:Sigma(2,3,7)}. 
Applying the $-1$-surgery Laplace transform to $V_{x,q}(S^3 \setminus \mathbf{4}_1)$ from Example \ref{eg:figure-eight}, we have
\[
V_q(-\Sigma(2,3,7)) = \mathbb{Z}[q^{\pm \frac14}]
\left\{ 
2(1-q),\;
\mathcal{F}_0^\lambda(q) := \sum_{n\geq 0} \frac{q^{n^2 - \lambda n}}{(q^{n+1};q)_n} \;\bigg\vert\; \lambda \geq -2 
\right\}.
\]
Thanks to the degree $3$ inhomogeneous $q$-difference equation 
\begin{equation}\label{eq:F-inhom-qdiff}
\qty(
(1+\hat{y})(1-q \hat{y}^2)
- q^{-1} \hat{x}^{-1}\hat{y}
)
\mathcal{F}_0^{\lambda}(q)
=
2(1-q), 
\end{equation}
it follows that 
\begin{align*}
V_q(-\Sigma(2,3,7)) &= \mathbb{Z}[q^{\pm \frac14}]
\left\{ 
\mathcal{F}_0^{\lambda}(q)
\;\bigg\vert\; \lambda \in \mathbb{Z}
\right\} \\
&= 
\mathbb{Z}[q^{\pm\frac14}] \left\{ \mathcal{F}_0^{\lambda_0}(q), \mathcal{F}_0^{\lambda_0 +1}(q), \mathcal{F}_0^{\lambda_0 +2}(q), \mathcal{F}_0^{\lambda_0 +3}(q)\right\}
\end{align*}
for any $\lambda_0 \in \mathbb{Z}$. 
As expected, this is the mirror ($q \leftrightarrow q^{-1}$) of $V_q(\Sigma(2,3,7))$. 

\end{eg}

Analogously to Question \ref{qn:1-in-Vxq}, we ask the following:
\begin{qn}
Is $1$ always in $V_q(Y)$ (after tensoring with $\mathbb{Q}(q^{\frac14})$)? 
\end{qn}

\section{Modularity and duality}\label{sec:modularity}
In this section, we will show that the $q$-series module $V_q(\pm\Sigma(2,3,7))$ discussed in the previous section extends to a matrix-valued holomorphic quantum modular form. 
The first step is the factorization of an appropriate state integral \cite{GK, GGMW, Wheeler-thesis}, which we review below. 

\subsection{From $q$-series to state integrals and back}
Recall that the \emph{non-compact quantum dilogarithm} $\Phi_{\mathsf{b}}(x)$ is defined as follows. 
\begin{defn}
\[
\Phi_{\mathsf{b}}(x) := 
\exp\qty(\int_{-\infty}^{+\infty} \frac{e^{-2i xw}}{4 \sinh(w \mathsf{b}) \sinh(w \mathsf{b}^{-1})}\frac{dw}{w})
= \frac{(e^{2\pi (x+c_{\mathsf{b}})\mathsf{b}}; \mathbf{e}(\mathsf{b}^2))_\infty}{(e^{2\pi (x-c_{\mathsf{b}})\mathsf{b}^{-1}}; \mathbf{e}(-\mathsf{b}^{-2}))_\infty},
\]
where $c_\mathsf{b} := \frac{i}{2}(\mathsf{b} + \mathsf{b}^{-1})$, and the first expression is valid for $|\Im(x)| < |\Im(c_{\mathsf{b}})|$, 
while the second expression is valid for $\Im(\mathsf{b}^2) > 0$. 
\end{defn}
We also set
\[
\tau := \mathsf{b}^2,\quad q:= e^{2\pi i \tau},\quad \tilde{q} := e^{2\pi i(-\frac{1}{\tau})}.
\]

We would like to find state integrals\footnote{Integrals involving products and ratios of $\Phi_{\mathsf{b}}$ and a Gaussian term are often called ``state integrals''.} realizing the holomorphic quantum modularity of the $q$-series module
\[
V_q(-\Sigma(2,3,7)) = \mathbb{C}(q^{\frac14}) 
\left\{ \mathcal{F}_0^{-\lambda}(q) = \sum_{n\geq 0} \frac{q^{n^2 + \lambda n}}{(q^{n+1};q)_n} \;\bigg\vert\; \lambda \in \mathbb{Z} \right\}. 
\]
The idea is to start from the $q$-hypergeometric series, and replace $q^{k n^2}$ by $e^{2\pi i (-kx^2)}$ ($k=1$ in our case), and $\frac{1}{(q;q)_{ln}}$ by $\Phi_{\mathsf{b}}(l(x-c_{\mathsf{b}}) + c_{\mathsf{b}})$ ($l = 1, 2$ in our case). 
This gives the following ``homogeneous'' state integral:\footnote{
We thank Campbell Wheeler for correspondences and for suggesting these state integrals. 
} 
\[
S_{\lambda,\mu}(\tau) := 
\int_{\mathbb{R}+i\epsilon} \frac{\Phi_{\mathsf{b}}(2x - c_{\mathsf{b}})}{\Phi_{\mathsf{b}}(x)} 
e^{2\pi i (-x^2 - i (\lambda \mathsf{b} + \mu \mathsf{b}^{-1})x)} dx,
\]
where we have also inserted an extra factor $e^{2\pi i(-i(\lambda \mathsf{b} + \mu \mathsf{b}^{-1})x)}$ dependent on a pair of integers $\lambda, \mu \in \mathbb{Z}$. 
This state integral can be computed by pushing the contour to $i\infty$, collecting the residues at the poles of $\Phi_{\mathsf{b}}(2x - c_{\mathsf{b}})$ that are not canceled by the poles of $\Phi_{\mathsf{b}}(x)$.\footnote{This is possible because the contour integral along a big semicircle vanishes in the large radius limit; the absolute value of the integrand is $O(e^{-k|x|^2})$, where $k$ is
$(2^2-1^2)-2(1) = 1>0$ for $\arg x < \frac{\pi}{2}$
and $0 + 2(1) = 2 > 0$ for $\arg x > \frac{\pi}{2}$.} 
The poles are at 
\[
x_{\frac{m}{2},\frac{n}{2}} := i (\frac{m}{2}+\frac{1}{2})\mathsf{b} + i(\frac{n}{2} + \frac{1}{2})\mathsf{b}^{-1},
\]
with $m, n\geq 0$, not both even, and they are all simple. 
A straightforward calculation shows
\begin{align*}
S_{\lambda,\mu}(\tau)&= \sum_{\substack{m,n \geq 0 \\ \text{not both }m,n\text{ even}}} 
\mathrm{Res}_{x = x_{\frac{m}{2}, \frac{n}{2}}} 
\qty(
\frac{\Phi_{\mathsf{b}}(2x - c_{\mathsf{b}})}{\Phi_{\mathsf{b}}(x)} 
e^{2\pi i (-x^2 - i (\lambda \mathsf{b} + \mu \mathsf{b}^{-1})x)}
) \\
&= 
\sum_{\substack{m,n \geq 0 \\ \text{not both }m,n\text{ even}}} \mathrm{Res}_{x = 0} 
\left(
\frac{\Phi_{\mathsf{b}}(2x + 2x_{\frac{m}{2}, \frac{n}{2}} - c_{\mathsf{b}})}{\Phi_{\mathsf{b}}(x + x_{\frac{m}{2}, \frac{n}{2}})} \right.
\\
&\qquad\qquad\qquad\qquad\times \left.
e^{-2\pi i (x+ x_{\frac{m}{2}, \frac{n}{2}})^2 + 2\pi \mathsf{b}\lambda (x+ x_{\frac{m}{2}, \frac{n}{2}}) + 2\pi \mathsf{b}^{-1} \mu (x+ x_{\frac{m}{2}, \frac{n}{2}})}
\right)
\\
&= 
\sum_{\substack{m,n \geq 0 \\ \text{not both }m,n\text{ even}}} \mathrm{Res}_{x = 0} 
\bigg(
\frac{1}{1 - e^{4\pi \mathsf{b}^{-1} x}} \\
&\qquad\times
\frac{
(q^{m+1} e^{4\pi \mathsf{b} x}; q)_\infty
}{
((-1)^{n} q^{\frac{m}{2}+1} e^{2\pi \mathsf{b}x}; q)_\infty
}
\frac{
((-1)^{m} \tilde{q}^{-\frac{n}{2}} e^{2\pi \mathsf{b}^{-1}x}; \tilde{q})_\infty
}{
(\tilde{q}^{-1} e^{4\pi \mathsf{b}^{-1} x}; \tilde{q}^{-1})_{n}
(\tilde{q} e^{4\pi \mathsf{b}^{-1} x}; \tilde{q})_\infty
}
\\
&\qquad \times 
e^{-2\pi i x^2 + 2\pi ((m+1 + \lambda)\mathsf{b} + (n+1 + \mu)\mathsf{b}^{-1}) x}
) 
\bigg)\\ 
&\qquad \times 
(-1)^{(n+1)\lambda + (m+1)\mu + (m+1)(n+1)} q^{(\frac{m+1}{2})^2 + \frac{m+1}{2}\lambda} \tilde{q}^{-(\frac{n+1}{2})^2-\frac{n+1}{2}\mu}\\
&= 
-\frac{\mathsf{b}}{4\pi}
\sum_{\substack{m,n \geq 0 \\ \text{not both }m,n\text{ even}}} 
(-1)^{(n+1)\lambda}
q^{(\frac{m+1}{2})^2 + \frac{m+1}{2}\lambda} 
\frac{
(q^{m+1}; q)_\infty 
}{
((-1)^{n} q^{\frac{m}{2}+1}; q)_\infty
} \\
&\qquad\qquad\qquad\qquad \times
(-1)^{(m+1)\mu}
\tilde{q}^{-(\frac{n+1}{2})^2-\frac{n+1}{2}\mu}
\frac{
((-1)^{m} \tilde{q}^{-\frac{n}{2}}; \tilde{q})_\infty
}{
(\tilde{q}^{-1} ; \tilde{q}^{-1})_{n}
(\tilde{q} ; \tilde{q})_\infty
}.
\end{align*}
Separating the sum into parts with $(m, n)\text{ mod }2 = (0, 1), \;(1, 0), \;\text{and } (1, 1)$ and simplifying the infinite $q$-Pochhammer symbols using modularity $\eta(-\tau^{-1}) = \sqrt{-i \tau}\, \eta(\tau)$ of the $\eta$-function $\eta(\tau) = q^{\frac{1}{24}} (q;q)_\infty$ and similar properties of the $\eta$-quotients
\[
\frac{
(q;q)_\infty
}{
(\tilde{q}; \tilde{q})_\infty
}
= 
\frac{q^{-\frac{1}{24}}\eta(\tau)}{\tilde{q}^{-\frac{1}{24}}\eta(-\frac{1}{\tau})}
= 
e^{\frac{\pi i}{4}}
q^{-\frac{1}{24}} \tilde{q}^{\frac{1}{24}} \tau^{-\frac12},
\]
\[
\frac{(\tilde{q}^{\frac12}; \tilde{q})_\infty}{(-q;q)_\infty} 
= \frac{\tilde{q}^{\frac{1}{48}} \eta(-\frac{1}{2\tau})/\eta(-\frac{1}{\tau})}{q^{-\frac{1}{24}} \eta(2\tau) / \eta(\tau)}
= \sqrt{2}q^{\frac{1}{24}} \tilde{q}^{\frac{1}{48}},
\]
\[
\frac{(-\tilde{q}^{\frac12};\tilde{q})_\infty}{(-q^{\frac12};q)_\infty}
= \frac{\tilde{q}^{\frac{1}{48}} \frac{\eta(-\frac{1}{\tau})^2}{\eta(-\frac{1}{2\tau}) \eta(-\frac{2}{\tau})}}{q^{\frac{1}{48}} \frac{\eta(\tau)^2}{\eta(\frac{\tau}{2}) \eta(2\tau)}}
= q^{-\frac{1}{48}} \tilde{q}^{\frac{1}{48}},
\]
we conclude that
\begin{align*}
S_{\lambda,\mu}(\tau) &= 
-\frac{\tau^{\frac12}}{4\pi}
\bigg(
(-1)^{\frac{1}{4}}\sqrt{2}\tau^{-\frac12}\; \tilde{q}^{\frac{1}{16}} B_{\mu}(\tilde{q}^{-1}) \; A_{\lambda}(q)\\
&\qquad+ (-1)^{\frac{1}{4}}\sqrt{2}\tau^{-\frac12}\; A_{\mu}(\tilde{q}^{-1}) \; q^{-\frac{1}{16}}B_{\lambda}(q)\\
&\qquad+ \tau^{-\frac12}\; \tilde{q}^{\frac{1}{16}} C_{\mu}(\tilde{q}^{-1}) \; q^{-\frac{1}{16}}C_{\lambda}(q)
\bigg),
\end{align*}
where
\begin{align*}
A_{\lambda}(q) &:= 
\sum_{m \geq 0} 
q^{(m + \frac{1}{2})^2 + (m + \frac{1}{2})\lambda} 
\frac{
(-q;q)_{m}
}{(q;q)_{2m}
},
\\
B_{\lambda}(q) &:= 
\sum_{m \geq 0}
(-1)^{\lambda}
q^{(m+1)^2 + (m+1)\lambda} 
\frac{
(q^{\frac12}; q)_{m+1}
}{
(q;q)_{2m+1}
},
\\
C_{\lambda}(q) &:= 
\sum_{m \geq 0} 
q^{(m+1)^2 + (m+1)\lambda} 
\frac{
(-q^{\frac12}; q)_{m+1}
}{
(q;q)_{2m+1}
}.
\end{align*}
Note, we haven't gotten the $q$-series we wanted yet, since they should come from the poles of $\Phi_{\mathsf{b}}(2x-c_{\mathsf{b}})$ at $x_{\frac{m}{2},\frac{n}{2}}$ when both $m$ and $n$ are even, except that those poles are canceled by those of $\Phi_{\mathsf{b}}(x)$. 
Hence, to get the desired $q$-series, we need to add extra poles by inserting the factor $\frac{1}{1 - e^{2\pi \mathsf{b}^{-1}(x-c_{\mathsf{b}})}}$; this gives the ``inhomogeneous'' state integral:\footnote{In the notation, $m$ for ``modified.''}
\[
mS_{\lambda,\mu}(\tau) := 
\int_{\mathbb{R}+i\epsilon} 
\frac{1}{1 - e^{2\pi \mathsf{b}^{-1}(x-c_{\mathsf{b}})}} 
\frac{\Phi_{\mathsf{b}}(2x - c_{\mathsf{b}})}{\Phi_{\mathsf{b}}(x)} 
e^{2\pi i (-x^2 - i (\lambda \mathsf{b} + \mu \mathsf{b}^{-1})x)} dx.
\]
That is, we are adding new simple poles at $x_{m,0}$, $m \in \mathbb{Z}$. 
Following the same procedure as before, we have 
\begin{align*}
mS_{\lambda,\mu}(\tau) &= 
-\frac{\tau^{\frac12}}{4\pi}
\bigg(
(-1)^{\frac{1}{4}}\sqrt{2}\tau^{-\frac12}\; \tilde{q}^{\frac{1}{16}} mB_{\mu}(\tilde{q}^{-1}) \; A_{\lambda}(q)\\
&\qquad+ (-1)^{\frac{1}{4}}\sqrt{2}\tau^{-\frac12}\; mA_{\mu}(\tilde{q}^{-1}) \; q^{-\frac{1}{16}}B_{\lambda}(q)\\
&\qquad+ \tau^{-\frac12}\; \tilde{q}^{\frac{1}{16}} mC_{\mu}(\tilde{q}^{-1}) \; q^{-\frac{1}{16}}C_{\lambda}(q) \\
&\qquad+ (-1)^{\mu+1}\tilde{q}^{-\frac{1}{4} - \frac{\mu}{2}}\;
Z_{\lambda}(q)
\bigg),
\end{align*}
where
\begin{align*}
Z_\lambda(q) &:= \sum_{m\geq 0} (-1)^{\lambda}q^{(m+\frac12)^2 + (m+\frac12)\lambda}\frac{(q;q)_m}{(q;q)_{2m}}, \\
mA_{\lambda}(q) &:= 
\sum_{m \geq 0} 
q^{(m + \frac{1}{2})^2 + (m + \frac{1}{2})\lambda} 
\frac{1}{1+q^m}
\frac{
(-q;q)_{m}
}{(q;q)_{2m}
},
\\
mB_{\lambda}(q) &:= 
\sum_{m \geq 0}
(-1)^{\lambda}
q^{(m+1)^2 + (m+1)\lambda} 
\frac{1}{1-q^{\frac{1}{2}+m}}
\frac{
(q^{\frac12}; q)_{m+1}
}{
(q;q)_{2m+1}
},
\\
mC_{\lambda}(q) &:= 
\sum_{m \geq 0} 
q^{(m+1)^2 + (m+1)\lambda} 
\frac{1}{1+q^{\frac{1}{2}+m}}
\frac{
(-q^{\frac12}; q)_{m+1}
}{
(q;q)_{2m+1}
}.
\end{align*}
Note, $Z_{\lambda}(q)$ is the desired $q$-series $\mathcal{F}_0^{-(\lambda+1)}(q)$, up to sign and an overall $q$-power: 
\begin{equation}\label{eq:Z-to-F}
Z_{\lambda}(q) = q^{\frac14}(-q^{\frac12})^{\lambda} \mathcal{F}_0^{-(\lambda+1)}(q).
\end{equation}
Our discussion so far can be summarized in the following proposition:
\begin{prop}[Factorization of state integrals]\label{prop:factorization-state-integral}
Define
\[
Q_{\lambda}(q)
:= 
\qty(
\begin{array}{c|ccc}
1 & 0 & 0 & 0 \\
\hline
0 & 1 & 0 & 0 \\
0 & 0 & q^{-\frac{1}{16}} & 0 \\
0 & 0 & 0 & q^{-\frac{1}{16}}
\end{array}
)
\qty(
\begin{array}{c|ccc}
1 & Z_{\lambda}(q) & Z_{\lambda+1}(q) & Z_{\lambda+2}(q)\\
\hline
0 & A_{\lambda}(q) & A_{\lambda+1}(q) & A_{\lambda+2}(q) \\
0 & B_{\lambda}(q) & B_{\lambda+1}(q) & B_{\lambda+2}(q) \\
0 & C_{\lambda}(q) & C_{\lambda+1}(q) & C_{\lambda+2}(q) 
\end{array}
),
\]
and
\[
Q^\vee_{\lambda}(q) := 
\qty(
\begin{array}{c|ccc}
1 & 0 & 0 & 0 \\
\hline
0 & 1 & 0 & 0 \\
0 & 0 & q^{-\frac{1}{16}} & 0 \\
0 & 0 & 0 & q^{-\frac{1}{16}}
\end{array}
)
\qty(
\begin{array}{c|ccc}
-q^{\frac14}(-q^{\frac12})^{\lambda} & 0 & 0 & 0\\
\hline
mA_{\lambda}(q) & A_{\lambda}(q) & A_{\lambda+1}(q) & A_{\lambda+2}(q) \\
mB_{\lambda}(q) & B_{\lambda}(q) & B_{\lambda+1}(q) & B_{\lambda+2}(q) \\
mC_{\lambda}(q) & C_{\lambda}(q) & C_{\lambda+1}(q) & C_{\lambda+2}(q) 
\end{array}
).
\]
Then, we have
\begin{align*}
&Q^\vee_{\mu}(\tilde{q}^{-1})^T
\cdot
\qty(
\begin{array}{c|ccc}
1 & 0 & 0 & 0 \\
\hline
0 & 0 & (-1)^{\frac14}\sqrt{2}\tau^{-\frac{1}{2}}  & 0 \\
0 & (-1)^{\frac14}\sqrt{2}\tau^{-\frac{1}{2}}  & 0 & 0 \\
0 & 0 & 0 & (-1)^{\frac14}\tau^{-\frac{1}{2}} 
\end{array}
)
\cdot
Q_{\lambda}(q)
\\
&= 
-4\pi \tau^{-\frac{1}{2}}
\qty(
\begin{array}{c|ccc}
-\tilde{q}^{-\frac14}(-\tilde{q}^{-\frac12})^{\mu} & mS_{\lambda,\mu}(\tau) & mS_{\lambda+1,\mu}(\tau) & mS_{\lambda+2,\mu}(\tau) \\
\hline
0 & S_{\lambda,\mu}(\tau) & S_{\lambda+1,\mu}(\tau) & S_{\lambda+2,\mu}(\tau) \\
0 & S_{\lambda,\mu+1}(\tau) & S_{\lambda+1,\mu+1}(\tau) & S_{\lambda+2,\mu+1}(\tau) \\
0 & S_{\lambda,\mu+2}(\tau) & S_{\lambda+1,\mu+2}(\tau) & S_{\lambda+2,\mu+2}(\tau)
\end{array}
)
.
\end{align*}
In particular, the above product extends holomorphically from $\mathbb{C}\setminus \mathbb{R}$ to the cut plane $\mathbb{C}_S = \mathbb{C} \setminus \mathbb{R}_{\leq 0}$. 
\end{prop}

\begin{rmk}
While the matrices $Q_{\lambda}(q)$ and $Q^\vee_{\lambda}(q)$ depend on $\lambda \in \mathbb{Z}$, their column spaces are independent of the exact value of $\lambda$, since they satisfy the following recurrence relations:
\begin{align*}
Q_{\lambda+1}(q) 
&= 
Q_{\lambda}(q) \cdot
\qty(
\begin{array}{c|ccc}
1 & 0 & 0 & 2(1-q)q^{\frac14}(-q^{\frac12})^{\lambda+3} \\
\hline
0 & 0 & 0 & -q^{\frac52} \\
0 & 1 & 0 & q^2 \\
0 & 0 & 1 & q^{\frac12} (1 + q^{\lambda+3})
\end{array}
), \\
Q^\vee_{\lambda+1}(q)
&=
Q^\vee_{\lambda}(q)
\cdot
\qty(
\begin{array}{c|ccc}
-q^{\frac12} & 0 & 0 & 0 \\
\hline
q^{\frac12} & 0 & 0 & -q^{\frac{5}{2}} \\
0 & 1 & 0 & q^2 \\
0 & 0 & 1 & q^{\frac12}(1+q^{\lambda+3})
\end{array}
)
.
\end{align*}
Note, the column span of the top row of $Q_\lambda(q)$ is exactly the $q$-series module $V_q(Y)$ we are interested in. 
Also, the last column of the $q$-difference equation for $Q_\lambda(q)$ above is exactly the inhomogeneous $q$-difference equation \eqref{eq:F-inhom-qdiff} for $\mathcal{F}_0^{\lambda}(q)$, once we replace all $Z$'s by $\mathcal{F}$'s using \eqref{eq:Z-to-F}. 
\end{rmk}

\subsection{Holomorphic quantum modularity}

\begin{prop}[Semilinear pairing]\label{prop:semilinear-pairing} 
We have the following semilinear pairing\footnote{Such a pairing is often called a ``quadratic relation'' in the literature.}
\begin{align*}
&Q^\vee_{-2-\lambda}(q^{-1})^T
\cdot
\qty(
\begin{array}{c|ccc}
-1 & 0 & 0 & 0 \\
\hline
0 & 2 & 0 & 0 \\
0 & 0 & 1 & 0 \\
0 & 0 & 0 & 1 
\end{array}
)
\cdot
Q_{-2+\lambda}(q)
\\
&=
\qty(
\begin{array}{c|ccc}
-q^{\frac14}(-q^{\frac12})^{\lambda+1} & 2(q^{\lambda} + q^{2\lambda}) & 0 & 2q^{1+\lambda} \\
\hline
0 & 2(q^{\lambda} + q^{2\lambda}) & 2q^{\frac12 + \lambda} & 2q^{1+\lambda} \\
0 & 2q^{-\frac12+\lambda} & 2q^{\lambda} & 2q^{\frac12+\lambda} \\
0 & 2q^{-1+\lambda} & 2q^{-\frac12+\lambda} & 2(q^{\lambda}+q^{2\lambda})
\end{array}
)
=: P_{\lambda}(q)
.
\end{align*}
\end{prop}

\begin{thm}
Fix any $\lambda \in \mathbb{Z}$. 
Then, $Q(q) = Q_{\lambda}(q)$ is a matrix-valued holomorphic quantum modular form, with the automorphy factor (i.e., a matrix-valued $\mathrm{SL}_2(\mathbb{Z})$-cocycle) $\gamma \mapsto j_{\gamma}(\tau)$ generated by
\begin{align*}
j_S(\tau)
&:=
\qty(
\begin{array}{c|ccc}
1 & 0 & 0 & 0 \\
\hline
0 & \tau^{-\frac{1}{2}} & 0 & 0 \\
0 & 0 & \tau^{-\frac{1}{2}} & 0 \\
0 & 0 & 0 & \tau^{-\frac{1}{2}}
\end{array}
)
\qty(
\begin{array}{c|ccc}
-1 & 0 & 0 & 0 \\
\hline
0 & 0 & \frac{(-1)^{\frac14}}{\sqrt{2}} & 0 \\
0 & (-1)^{\frac14}\sqrt{2} & 0 & 0 \\
0 & 0 & 0 & (-1)^{\frac14}
\end{array}
), \\
j_T(\tau)
&:=
\qty(
\begin{array}{c|ccc}
-1 & 0 & 0 & 0 \\
\hline
0 & -1 & 0 & 0 \\
0 & 0 & 0 & (-1)^{\frac38} \\
0 & 0 & (-1)^{\frac38} & 0
\end{array}
).
\end{align*}
In other words, the cocycle
\[
\Omega_{\gamma}(\tau) := Q(q \vert_{\gamma})^{-1}\cdot j_{\gamma}(\tau) \cdot Q(q),
\]
extends holomorphically from $\mathbb{C}\setminus \mathbb{R}$ to the cut plane $\mathbb{C}_{\gamma}$. 
\end{thm}
\begin{proof}
It suffices to check this for the generators $S, T \in \mathrm{SL}_2(\mathbb{Z})$. 
For $T$, this is obvious because $\Omega_T(\tau) = (-1)^{\lambda +\frac12}\mathrm{Id}$. 
For $S$, the semilinear pairing (Proposition \ref{prop:semilinear-pairing}) and the factorization of state integrals (Proposition \ref{prop:factorization-state-integral}) imply that the cocycle
\begin{align*}
\Omega_S(\tau) &= Q_{\lambda}(\tilde{q})^{-1}\cdot j_S(\tau) \cdot Q_{\lambda}(q) \\
&= 
P_{\lambda+2}(\tilde{q})^{-1}\cdot
Q^{\vee}_{-4-\lambda}(\tilde{q}^{-1})^T \cdot
\qty(
\begin{array}{c|ccc}
-1 & 0 & 0 & 0 \\
\hline
0 & 2 & 0 & 0 \\
0 & 0 & 1 & 0 \\
0 & 0 & 0 & 1 
\end{array}
)
\cdot j_S(\tau) \cdot Q_{\lambda}(q) \\
&= 
P_{\lambda+2}(\tilde{q})^{-1}\cdot
Q^{\vee}_{-4-\lambda}(\tilde{q}^{-1})^T \cdot \\
&\qquad \times
\qty(
\begin{array}{c|ccc}
1 & 0 & 0 & 0 \\
\hline
0 & 0 & (-1)^{\frac14}\sqrt{2}\tau^{-\frac{1}{2}}  & 0 \\
0 & (-1)^{\frac14}\sqrt{2}\tau^{-\frac{1}{2}}  & 0 & 0 \\
0 & 0 & 0 & (-1)^{\frac14}\tau^{-\frac{1}{2}} 
\end{array}
) \cdot Q_{\lambda}(q)
\end{align*}
extends holomorphically to the cut plane. 
\end{proof}

\subsection{Speculations on Langlands duality}\label{subsec:Langlands}

We conclude with some speculative remarks on interpretation of the matrix of $q$-series we have observed above, and on Langlands duality for skein modules. 

Langlands duality for skein modules \cite{Jordan} predicts a correspondence between skein modules associated to a reductive group $G$ and those for its Langlands dual group ${}^{L}G$. 
At the most basic level, it conjectures that, for any closed 3-manifold $Y$, its $G$- and ${}^{L}G$-skein modules have the same dimension. 
At a deeper level, one expects a $\mathbb{C}$-linear isomorphism 
\[
\Sk^{G}_{q}(Y) \cong \Sk^{{}^{L}G}_{{}^{L}q}(Y)
\]
between the two skein modules that varies holomorphically in $\tau \in \mathbb{C} \setminus \mathbb{R}$, where $q = e^{2\pi i \tau}$ and ${}^{L}q = e^{-2\pi i \frac{1}{n_G \tau}}$, with $n_G \in \{1, 2, 3\}$ the lacing number of $G$. 
As discussed in \cite{Jordan}, the skein modules are expected to describe the state spaces of Kapustin--Witten 4d TQFT \cite{KapustinWitten}, and, in that context, the Langlands duality for skein modules arises from the $S$-duality on the Kapustin--Witten theory applied to state spaces. 

\begin{rmk}
The columns of $Q(q)$ and $Q^\vee(q)$ are labeled by basis vectors of $\Sk^{\mathrm{SL}_2}_q(Y)$, while their rows are presumably labeled by $\mathrm{SL}_2(\mathbb{C})$-flat connections\footnote{Or, more generally, some equivariant version of sheaf-theoretic $\mathrm{SL}_2(\mathbb{C})$-Floer homology \cite{AbouzaidManolescu}. } on $Y$, with the top row corresponding to the trivial flat connection. 
Thus, it is tempting to interpret the matrices $Q(q)$ and $Q^\vee(q)$ as the partition functions (up to suitable normalization) of the Kapustin--Witten theory on $Y \times [0,\infty)$ and on $Y \times (-\infty, 0]$, with the boundary condition at $0$ given by the Nahm pole boundary specified by an element of the skein module, and the boundary condition at $\pm \infty$ given by the choice of flat connection. 
Considering the Kapustin--Witten partition function on $Y \times [0,1]$, with the two boundary conditions specified by elements of the skein module, it should give a semilinear\footnote{That is, $\langle f(q) L_1, g(q) L_2\rangle = f(q)g(q^{-1}) \langle L_1, L_2\rangle$ for any $f(q), g(q) \in \mathbb{Z}[q^{\pm \frac14}]$ and $L_1, L_2 \in \Sk_q^{\mathrm{SL}_2}(Y)$.} pairing $\langle, \rangle$ on the skein module. 
This might be the origin of the semilinear pairing in Proposition \ref{prop:semilinear-pairing}, though it is not clear why it should take values in Laurent polynomials, rather than power series.  
\end{rmk}

\begin{qn}
For an arbitrary $3$-manifold $Y$ (say, an integer homology sphere), is there a naturally defined semilinear pairing
\[
\langle, \rangle : \mathrm{Sk}^{\mathrm{SL}_2}_q(Y) \times \mathrm{Sk}^{\mathrm{SL}_2}_q(Y) \rightarrow \mathbb{Z}[q^{\pm \frac14}]?
\]
\end{qn}

\begin{rmk}
In a similar vein, the rows and columns of the automorphy factor $j_{\gamma}(\tau)$ are presumably labeled by flat connections, while the rows and columns of the cocycle $\Omega_{\gamma}(\tau)$ are labeled by basis vectors of $\Sk(Y)$. 
Recall from \cite[Conj. 4.2]{Jordan} that, when $G$ is simply-laced and simply connected (so that $Z(G)^\vee \cong Z(G)$), the Langlands duality for skein modules can be phrased as a transcendental isomorphism
\[
\Sk^{G}_q(Y)_{a,b} \overset{S}{\cong} \Sk^G_{\tilde{q}}(Y)_{b,a},
\]
where $\Sk^{G}_q(Y)_{a,b}$ denotes the $a$-graded component of the $b$-twisted $G$-skein module, with $a \in H_1(Y;Z(G)^\vee)$ and $b \in H_1(Y;Z(G))$. 
Even when $Y$ is an integer homology sphere, so that we can ignore the grading and twisting, the Langlands duality predicts a non-trivial transcendental isomorphism
\[
\Sk^{G}_q(Y) \overset{S}{\cong} \Sk^G_{\tilde{q}}(Y). 
\]
The cocycle $\Omega_S(\tau)$ provides a natural candidate for such a map. 
That is, we conjecture that it fits into the following commutative square:
\[
\begin{tikzcd}
\Sk_q^{\mathrm{SL}_2}(Y) \arrow[r, "S"] \arrow[d, "\widehat{Z}"]& \Sk_{\tilde{q}}^{\mathrm{SL}_2}(Y) \arrow[d, "\widehat{Z}"] \\
V_q(Y) \arrow[r, "\Omega_S(\tau)"] & V_{\tilde{q}}(Y)
\end{tikzcd}.
\]
For $3$-manifolds $Y$ with non-trivial $H_1(Y;\mathbb{Z}/2)$, we expect the relation to involve an intricate interplay between grading, twisting, and twisted spin$^c$ structures. 
We plan to investigate this in future work, with the hope that it will shed light on Langlands duality for skein modules beyond the numerical comparison of dimensions. 
\end{rmk}

\appendix
\section{Review of perturbative expansion of $R$-matrices}\label{sec:perturbative-expansion-review}

\subsection{Rozansky's lemmas and their extension}
\begin{lem}[{\cite[Lem. A.1]{Rozansky}}]\label{lem:qbin}
For any $n\geq 0$, 
\begin{equation}\label{eq:qbin-pert-expansion}
\qbin{n}{k}_q = 
\binom{n}{k}\sum_{m\geq 0} Q_{m}(n,k) \hbar^m
\end{equation}
for some polynomial $Q_{m}(n,k) \in \mathbb{Q}[n,k]$ of $n$-degree $\leq m$ and total degree $\leq 2m$. 
\end{lem}
The first few of these polynomials are:
\begin{align*}
Q_0(n,k) &= 1,\\
Q_1(n,k) &= \frac{1}{2} k (n-k),\\
Q_2(n,k) &= \frac{1}{24} k (n-k) \left(3k (n-k) +n +1\right), \\
Q_3(n,k) &= \frac{1}{48} k^2 (k+1) (n-k)^2 (n-k+1).
\end{align*}
Note, $Q_m(n,k) = Q_m(n,n-k)$, which follows from the corresponding symmetry of (quantum) binomial coefficients.

\begin{lem}[{\cite[Lem. A.2, A.3]{Rozansky}}]\label{lem:Pmn}
Let
\[
P(m,n) := \prod_{m+1 \leq l \leq m+n}(1-\alpha(q^l - 1)). 
\]
Then, for any $n\geq 0$, 
\begin{align}\label{eq:P-pert-expansion}
P(m,n) 
&= \sum_{j\geq 0}(-\alpha)^j \hbar^j \prod_{0\leq k\leq j-1}(n-k) \cdot \sum_{d\geq 0} \hbar^d T_{j,d}(m,n),
\end{align}
where $T_{j,d}(m,n)$ is some polynomial in $m$ and $n$ of degree $\leq j+d$. 
\end{lem}
The first few of these polynomials are:
\begin{align*}
T_{0,d}(m,n) &= \delta_{d,0},\\
T_{1,0}(m,n) &= \frac{1}{2}(2m + n + 1), \\
T_{1,1}(m,n) &= \frac{1}{12} \left(6 m^2+6 m (n+1)+2 n^2+3 n+1\right), \\
T_{2,0}(m,n) &= \frac{1}{24} \left(12 m^2+12 m (n+1)+3 n^2+5 n+2\right), \\
T_{1,2}(m,n) &= \frac{1}{24} \left(4 m^3+6 m^2 (n+1)+m \left(4 n^2+6 n+2\right)+n (n+1)^2\right), \\
T_{2,1}(m,n) &= \frac{1}{12} \left(6 m^3+9 m^2 (n+1)+m \left(5 n^2+8 n+3\right)+n (n+1)^2\right), \\
T_{3,0}(m,n) &= \frac{1}{48} \left(8 m^3+12 m^2 (n+1)+2 m \left(3 n^2+5 n+2\right)+n (n+1)^2\right). \\
\end{align*}

When it comes to the inverted state sums discussed in Section \ref{sec:skein-q-series}, we will actually need these lemmas for \emph{all} $n\in \mathbb{Z}$, not just $n \geq 0$. 
This follows from the fact that a polynomial in $n$ vanishing for all $n \geq 0$ is identically zero; we include the argument for completeness: 
\begin{lem}
Lemmas \ref{lem:qbin} and \ref{lem:Pmn} hold for all integers $n \in \mathbb{Z}$, with the same polynomials $Q_m(n,k)$ and $T_{j,d}(m,n)$. 
\end{lem}
\begin{proof}
In both cases, we will show that the perturbative series defined by the right-hand sides of \eqref{eq:qbin-pert-expansion} and \eqref{eq:P-pert-expansion} satisfies the relevant recurrence relations for all $n$, hence both \eqref{eq:qbin-pert-expansion} and \eqref{eq:P-pert-expansion} hold for all $n \in \mathbb{Z}$. 

Lemma \ref{lem:qbin} implies that, for any $n\geq 0$, 
\begin{align*}
&\qbin{n+1}{k}_q - \frac{1-q^{n+1}}{1-q^{n-k+1}}\qbin{n}{k}_q \\
&= \frac{1}{n-k+1}\binom{n}{k} \sum_{m\geq 0} \hbar^m \bigg(
(n+1) Q_m(n+1,k) - Q_m(n,k)\frac{\sum_{a\geq 0} \frac{(n+1)^{a+1}}{(a+1)!} \hbar^a}{\sum_{b\geq 0} \frac{(n-k+1)^{b}}{(b+1)!} \hbar^b} \bigg)\\
&= \frac{1}{n-k+1}\binom{n}{k} \sum_{m'\geq 0} \hbar^{m'} Q'_{m'}(n,k),
\end{align*}
where $Q'_{m'}(n,k)$ are the polynomials obtained by expanding the second line of the equation into power series in $\hbar$. 
Since each $Q'_{m'}(n,k)$ vanishes for all $n\geq 0$, it must vanish identically. 
Thus, \eqref{eq:qbin-pert-expansion} holds for all $n \in \mathbb{Z}$. 

Lemma \ref{lem:Pmn} implies that, for any $n\geq 0$, 
\begin{align*}
&P(m, n+1)- P(m,n)(1-\alpha(q^{m+n+1}-1)) \\
&= 
\sum_{j\geq 0}(-\alpha)^j \hbar^j \prod_{0\leq k\leq j-1}(n+1-k) \cdot \sum_{d\geq 0} \hbar^d T_{j,d}(m,n+1)\\
&\quad- \sum_{j\geq 0}(-\alpha)^j \hbar^j \prod_{0\leq k\leq j-1}(n-k) \cdot \sum_{d\geq 0} \hbar^d T_{j,d}(m,n)
\qty(
1 - \alpha \sum_{l\geq 0} \frac{(m+n+1)^l}{l!} \hbar^{l}
) \\
&= \sum_{j\geq 0}(-\alpha)^j \hbar^j \prod_{0\leq k\leq j-2}(n-k) \cdot \sum_{d\geq 0} \hbar^d
\bigg(
(n+1)T_{j,d}(m,n+1) \\
&\quad - (n-j+1)T_{j,d}(m,n) - \sum_{0\leq d' \leq d} \frac{(m+n+1)^{d-d'+1}}{(d-d'+1)!} T_{j-1,d'}(m,n)
\bigg),
\end{align*}
so the polynomial in the big parenthesis in the last line vanishes for $n\geq 0$, hence it vanishes identically. 
It follows that \eqref{eq:P-pert-expansion} holds for all $n\in \mathbb{Z}$. 
\end{proof}

\begin{cor}\label{cor:diff-op-inversion}
The differential operators $\mathsf{D}_d$ and $\mathsf{D}'_d$ constructed below remain the same, even when some spins on $K$ are inverted. 
\end{cor}

\subsection{Proof of Lemmas \ref{lem:diff-op} and \ref{lem:diff-op-relative}}\label{subsec:diff-op-lemmas}
We will only consider $R$ (i.e., positive crossing), as the analogous statements for $R^{-1}$ follow immediately from those of $R$ from the identity
\[
R^{-1}(x_1, x_2, e^{\hbar})_{i,j}^{i', j'} = R(x_2^{-1},x_1^{-1},e^{-\hbar})_{j,i}^{j',i'}. 
\]
\begin{proof}[Proof of Lemma \ref{lem:diff-op}]
The proof is essentially identical to that of \cite[Cor. 3.1]{Rozansky}; the only difference is that 
we are using different colors, $V_\infty(x_1)$ and $V_\infty(x_2)$, for the two strands.\footnote{What is nice about this, as opposed to the case when the two strands are colored the same, is that the differential operator is determined essentially uniquely, because $\alpha$ and $\beta$ are specialized to different values.}
That is, we expand the $R$-matrix using Lemmas \ref{lem:qbin} and \ref{lem:Pmn}, and replace every $j$, $j'$, and $i-j'$ by $\alpha \partial_\alpha$, $\beta \partial_\beta$, and $\gamma \partial_\gamma$, respectively, as follows:
\begin{align*}
&x_1^{\frac{1}{4}}x_2^{\frac{1}{4}} R(x_1, x_2, q)_{i,j}^{i', j'} \\
&= 
\delta_{i+j, i'+j'} x_1^{-\frac{i-j'}{4} - \frac{j}{2}} x_2^{\frac{i-j'}{4} - \frac{j'}{2}} q^{(j+\frac{1}{2})(j'+\frac{1}{2})} \qbin{i}{j'}_q \prod_{j+1 \leq l \leq i'}(1 - q^l x_2^{-1}) \\
&= 
\delta_{i+j, i'+j'} x_1^{-\frac{i-j'}{4} - \frac{j}{2}} x_2^{\frac{i-j'}{4} - \frac{j'}{2}} 
\sum_{a\geq 0} \frac{(j+\frac{1}{2})^a(j'+\frac{1}{2})^a}{a!}\hbar^a \cdot
\binom{i}{j'} \sum_{b\geq 0} Q_b(i, i-j') \hbar^b \\
&\quad \times
(1 - x_2^{-1})^{i-j'} \sum_{n\geq 0} \qty(-\frac{x_2^{-1}}{1-x_2^{-1}})^{n} \hbar^n \prod_{0\leq k\leq n-1}(i-j' - k) \sum_{m\geq 0} \hbar^m T_{n,m}(j, i-j') \\
&= \sum_{d \geq 0} \hbar^d
\left(\sum_{\substack{a,b,n,m \geq 0 \\ a+b+n+m = d}}
\qty(- \alpha^{\frac{1}{2}} \beta^{\frac{3}{2}} \gamma^{-1})^n
\frac{(\alpha \partial_{\alpha} + \frac{1}{2})^a (\beta \partial_{\beta} + \frac{1}{2})^a}{a!} 
Q_b(\gamma \partial_{\gamma} + \beta \partial_{\beta}, \gamma \partial_{\gamma}) \right.\\
&\quad \times \left.
\prod_{0\leq k\leq n-1}(\gamma \partial_{\gamma} - k) \;T_{n,m}(\alpha \partial_{\alpha}, \gamma \partial_{\gamma}) \right)
\mathsf{R}(\alpha, \beta, \gamma)_{i,j}^{i',j'} \bigg\vert_{\substack{\alpha = x_1^{-\frac{1}{2}} \\ \beta = x_2^{-\frac{1}{2}} \\ \gamma = x_1^{-\frac{1}{4}} x_2^{\frac{1}{4}} (1-x_2^{-1})}}.
\end{align*}
In other words, 
\[
x_1^{\frac{1}{4}}x_2^{\frac{1}{4}} R(x_1, x_2, q)_{i,j}^{i', j'} = 
\sum_{d \geq 0} \hbar^d \mathsf{D}_d \mathsf{R}(\alpha, \beta, \gamma)_{i,j}^{i',j'} \bigg\vert_{\substack{\alpha = x_1^{-\frac{1}{2}} \\ \beta = x_2^{-\frac{1}{2}} \\ \gamma = x_1^{-\frac{1}{4}} x_2^{\frac{1}{4}} (1-x_2^{-1})}},
\]
where $\mathsf{D}_d \in \mathbb{Q}[\alpha^{\frac{1}{2}}, \beta^{\frac{1}{2}}, \gamma^{\pm 1}, \partial_{\alpha}, \partial_{\beta}, \partial_{\gamma}]$ is a sequence of differential operators defined as above. 
Moreover, the order of $\mathsf{D}_d$ (i.e., the $\partial$-degree of $\mathsf{D}_d$) is bounded above by
\[
\deg_{\partial} \mathsf{D}_d \leq 2a + \deg Q_b + n + \deg T_{n,m} \leq 2a + 2b + n + (n+m) \leq 2d,
\]
as claimed. 
\end{proof}

\begin{proof}[Proof of Lemma \ref{lem:diff-op-relative}]
The proof is almost identical; the point is that, once we fix $\underline{i}', \underline{j} \in \{0, 1, \cdots, m-1\}$, we can find differential operators purely in terms of $\beta$ and $\partial_\beta$: 
\begin{align*}
&x_1^{\frac{1}{4}}x_2^{\frac{1}{4}} R(x, q^m, q)_{i, \underline{j}}^{\underline{i}', j'} \\
&= 
\delta_{i+\underline{j}, \underline{i}'+j'} 
x^{-\frac{i-j'}{4} - \frac{\underline{j}}{2}} 
(q^m)^{\frac{i-j'}{4} - \frac{j'}{2}} 
q^{(\underline{j}+\frac{1}{2})(j'+\frac{1}{2})} 
\qbin{i}{j'}_q 
\prod_{\underline{j}+1 \leq l \leq \underline{i}'}(1 - q^l q^{-m}) \\
&= 
\delta_{i+\underline{j}, \underline{i}'+j'} 
x^{-\frac{i-j'}{4} - \frac{\underline{j}}{2}} 
(q^m)^{\frac{i-j'}{4} - \frac{j'}{2}} 
\sum_{a\geq 0} \frac{(\underline{j}+\frac{1}{2})^a(j'+\frac{1}{2})^a}{a!}\hbar^a\\
&\quad \times
\binom{i}{j'} \sum_{b\geq 0} Q_b(j'+(\underline{i}'-\underline{j}), \underline{i}'-\underline{j}) \hbar^b \cdot
\prod_{\underline{j}+1 \leq l \leq \underline{i}'}\frac{1 - q^l q^{-m}}{1-q^{-m}} \cdot (1-q^{-m})^{\underline{i}'-\underline{j}}\\
&= 
\sum_{d\geq 0} \hbar^d 
\left(
\sum_{\substack{a, b \geq 0 \\ a+b=d}}
\frac{(\underline{j}+\frac{1}{2})^a(\beta \partial_{\beta}+\frac{1}{2})^a}{a!}
Q_b(\beta \partial_{\beta} + \underline{i}'-\underline{j}, \underline{i}'-\underline{j})
\right.
\\
&\quad \times \left.
\prod_{\underline{j}+1 \leq l \leq \underline{i}'}\frac{1 - q^l q^{-m}}{1-q^{-m}}
\right)
\mathsf{R}(\alpha,\beta,\gamma)_{i, \underline{j}}^{\underline{i}',j'} \bigg\vert_{\substack{\alpha = x^{-\frac12} \\ \beta = q^{-\frac{m}{2}} \\ \gamma = x^{-\frac14} q^{\frac{m}{4}} (1-q^{-m})}}
\end{align*}
Further expanding $\prod_{\underline{j}+1 \leq l \leq \underline{i}'}\frac{1 - q^l q^{-m}}{1-q^{-m}}$ into a power series in $\hbar$, we obtain the desired differential operators $\mathsf{D}_d' \in \mathbb{Q}[\beta,\partial_\beta]$ of order $\leq a + b = d$. 
\end{proof}

\section{Review of gluing formulas and overall $q$-power}\label{sec:gluing}
Let $K\subset Y_1$ and $L\subset Y_2$ be knots in integer homology spheres. 
Let $Y_{K,L}$ denote the $3$-manifold obtained by gluing the boundaries of $Y_1\setminus K$ and $Y_2\setminus L$ via
\begin{align*}
\begin{pmatrix}
    \mu_2 \\ \lambda_2
\end{pmatrix} 
&= 
\begin{pmatrix}
    -s & -t \\ p & r
\end{pmatrix}
\begin{pmatrix}
    \mu_1 \\ \lambda_1
\end{pmatrix}, \\
\begin{pmatrix}
    \mu_1 \\ \lambda_1
\end{pmatrix}
&= 
\begin{pmatrix}
-r & -t \\ p & s    
\end{pmatrix}
\begin{pmatrix}
    \mu_2 \\ \lambda_2
\end{pmatrix},
\end{align*}
where $sr-tp = 1$, and $\mu_1, \lambda_1$ are the meridian and the longitude (in the Seifert framing) of $K$, and similarly for $\mu_2, \lambda_2$ and $L$. 
Then, $Y_{K,L}$ has $H_1(Y_{K,L};\mathbb{Z}) \cong \mathbb{Z}/p\mathbb{Z}$, and there are $p$ different choices of spin$^c$-structures. 

In this setup, the conjectural gluing formula for the BPS $q$-series is given by:
\begin{align*}
&\widehat{Z}_{Y_{K,L}, b}(q) \\
&= 
q^{d} \oint \frac{dx}{2\pi i x}\frac{dz}{2\pi i z}
\widehat{Z}_{Y_1\setminus K}(x,q) \widehat{Z}_{Y_2 \setminus L}(z,q)
\sum_{\substack{m,n \in \frac{1}{2} + \mathbb{Z} \\ rm + n \in b + p\mathbb{Z}}}
x^{-m}z^{-n}
q^{-\frac{1}{p}
\left(\begin{smallmatrix}
m & n
\end{smallmatrix}\right)
\left(\begin{smallmatrix}
r & 1 \\
1 & s
\end{smallmatrix}\right)
\left(\begin{smallmatrix}
m \\ 
n
\end{smallmatrix}\right)
} \\
&= 
q^{d}
\sum_{\substack{m,n \in \frac{1}{2} + \mathbb{Z} \\ rm + n \in b + p\mathbb{Z}}}
f_K^{m}(q)f_L^{n}(q)
q^{-\frac{1}{p}
\left(\begin{smallmatrix}
m & n
\end{smallmatrix}\right)
\left(\begin{smallmatrix}
r & 1 \\
1 & s
\end{smallmatrix}\right)
\left(\begin{smallmatrix}
m \\ 
n
\end{smallmatrix}\right)},
\end{align*}
possibly up to an overall sign, 
where $f_K^{m}(q)$ and $f_L^{n}(q)$ are defined by
\[
\widehat{Z}_{Y_1 \setminus K}(x,q) = \sum_{m \in \frac12+\mathbb{Z}} f^m_K(q) x^m,\quad
\widehat{Z}_{Y_2 \setminus L}(z,q) = \sum_{n \in \frac12+\mathbb{Z}} f^n_L(q) z^n,
\]
$b \in (\frac{1+r}{2} + \mathbb{Z})/p\mathbb{Z} \cong \mathrm{Spin}^c(Y_{K,L})$, and
\begin{align*}
d &= \frac{t}{4r} +3\mathrm{sign}(r)\qty(s(p,r)+\frac{\mathrm{sign}(p)}{4}) - \frac{p}{4r},
\end{align*}
where $s(p,r)$ denotes the Dedekind sum.\footnote{Our $d$ is a slightly corrected version of that from the conjectural $\frac{p}{r}$-surgery formula of \cite{GukovManolescu}.
Note, our $d$ is invariant under 
$\left(\begin{smallmatrix}
-s & -t \\
p & r
\end{smallmatrix}\right) \mapsto 
\left(\begin{smallmatrix}
s & t \\
-p & -r
\end{smallmatrix}\right)
$, as it should be, while the one given in \cite{GukovManolescu} isn't. 
} 
See \cite{GuicardiJagadale} for an interpretation of this formula in the TQFT language.
In this formula, $m$ and $n$ denote the choice of spin$^c$-structures (or equivalently, Euler structures) on $Y_1\setminus K$ and $Y_2\setminus L$, and the congruence condition $rm + n \in b + p\mathbb{Z}$ ensures that the two spin$^c$ structures glue together to the fixed spin$^c$ structure of $Y_{K,L}$. 
Note, the congruence condition can be equivalently written as
\[
\begin{pmatrix}
r & 1 \\
1 & s
\end{pmatrix}
\begin{pmatrix}
m \\
n
\end{pmatrix}
\equiv 
\begin{pmatrix}
b \\ 
b'
\end{pmatrix}
\mod p\mathbb{Z},
\]
where 
$b' = s b \in (\frac{1+s}{2} + \mathbb{Z})/p\mathbb{Z}$ and
$b = r b' \in (\frac{1+r}{2} + \mathbb{Z})/p\mathbb{Z}$; 
$b$ and $b'$ are just two different ways of describing the same spin$^c$-structure. 
From this, it is clear that all the powers of $q$ lie in the same coset of $\mathbb{Z}$. 

It is worth noting that this formula is really symmetric under exchanging $K$ and $L$. 
The extra overall factor $q^{d}$ remains the same under the symmetry $r \leftrightarrow s$ as well; this is a consequence of the reciprocity formula for Dedekind sums. 

Finally, when $L$ is the unknot in $S^3$ so that $\widehat{Z}_{S^3 \setminus L}(z,q) = z^{\frac12}-z^{-\frac12}$, this formula reduces to the conjectural $\frac{p}{r}$-surgery formula of Gukov--Manolescu \cite{GukovManolescu}.

\section{Generalized spin$^c$ structures and their twists}\label{sec:G-spin-c}
In Section \ref{subsec:twisted-spin-c}, we defined twisted spin$^c$ structures. 
There should be an analogous story for general semisimple gauge group $G$, where we have a $H^2(Y;Q)$-torsor of ``generalized spin$^c$ structures'' \cite{Park_higher-rank}, which can be twisted by $\alpha \in H_1(Y;Z(G)^\vee) \cong H_1(Y; X/Q)$, where $X$ is the character lattice ($Q \subseteq X \subseteq P$). 
Here, we propose an idea for what kind of geometric object they should be. 

Let $\varphi : SU(2) \rightarrow {}^{L}G_{\mathrm{sc}}$ be the principal embedding\footnote{
Note, in the context of Kapustin--Witten theory, such principal embedding is used to define the Nahm pole boundary condition. 
}, which is unique up to conjugation. 
The image of $-1$ is
\[
\varphi(-1) = [\rho] \in Z({}^{L}G_{\mathrm{sc}}) \cong (P^\vee/Q^\vee)^\vee \cong P/Q, 
\]
where $\rho \in P$ is the Weyl vector. 
The principal embedding descends to $\bar{\varphi} : SO(3) \rightarrow {}^{L}G_{\mathrm{ad}}$, so we can use it to construct the principal ${}^{L}G_{\mathrm{ad}}$-bundle
\[
P_{\mathrm{tan}} := Fr \times_{\bar{\varphi}} {}^{L}G_{\mathrm{ad}}
\]
associated to the frame bundle $Fr$ on $Y$. 

Let $T_Q := \mathfrak{h}^*/Q$ be the compact real torus with period lattice $Q$, so that $\pi_1(T_Q) \cong Q$ (i.e., the maximal torus of ${}^{L}G_{\mathrm{sc}}$). 
Recall that
\[
\pi_1({}^{L}G_{\mathrm{ad}}) \cong (P^\vee/Q^\vee)^\vee \cong P/Q \subset T_Q. 
\]
Define 
\[
{}^{L}\widehat{G}_{\mathrm{ad}} := ({}^{L}G_{\mathrm{sc}} \times T_Q)/\pi_1({}^{L}G_{\mathrm{ad}}),
\]
where the quotient is by the diagonal action. 
Pulling back the central extension
\[
1 \rightarrow T_Q \rightarrow {}^{L}\widehat{G}_{\mathrm{ad}} \rightarrow {}^{L}G_{\mathrm{ad}} \rightarrow 1 
\]
along $\bar{\varphi}$, we get a central extension
\[
1 \rightarrow T_Q \rightarrow \bar{\varphi}^*({}^{L}\widehat{G}_{\mathrm{ad}}) \rightarrow SO(3) \rightarrow 1,
\]
which can be described explicitly by
\[
\mathrm{Spin}^{G}(3) := 
\bar{\varphi}^*({}^{L}\widehat{G}_{\mathrm{ad}}) = (\mathrm{Spin}(3) \times T_Q)/\langle (-1, [\rho])\rangle. 
\]
Since the action of the Weyl group $W$ preserves $P/Q$ (and in particular $[\rho]$), $W$ acts as an automorphism on both ${}^{L}\widehat{G}_{\mathrm{ad}}$ and $\mathrm{Spin}^{G}(3)$. 
When $G=SU(2)$, $\mathrm{Spin}^{G}(3)$ is nothing but $\mathrm{Spin}^c(3)$, with the Weyl group $\mathbb{Z}/2$ acting by the standard spin$^c$ conjugation. 

We can then define a \emph{spin$^{G}$ structure} on $Y$ to be a lift of the principal ${}^{L}G_{\mathrm{ad}}$-bundle $P_{\mathrm{tan}}$ to a ${}^{L}\widehat{G}_{\mathrm{ad}}$-bundle:
\[
\begin{tikzcd}
& & B {}^{L}\widehat{G}_{\mathrm{ad}} \arrow[d] \\
Y \arrow[urr, dashed, "\widetilde{P}_{\mathrm{tan}}"] 
\arrow[r, "Fr", swap]
\arrow[rr, "P_{\mathrm{tan}}", bend right=30, swap]
& BSO(3) \arrow[r, "B\bar{\varphi}", swap] & B{}^{L}G_{\mathrm{ad}}
\end{tikzcd}. 
\]
Writing $\mathrm{Spin}^G(Y)$ for the set of spin$^{G}$ structures on $Y$, it is an $H^2(Y;Q)$-torsor equipped with an action of $W$. 

Recall that principal ${}^{L}G_{\mathrm{ad}}$-bundles on $Y$ are classified (up to isomorphism) by $H^2(Y;\pi_1({}^{L}G_{\mathrm{ad}})) \cong H^2(Y;P/Q)$. 
Given a $1$-cycle $C_\alpha$ representing $\alpha \in H_1(Y;Z(G)^\vee) \cong H_1(Y; X/Q)$, we can use it to twist $P_{\mathrm{tan}}$; as before, we modify $P_{\mathrm{tan}}$ in a tubular neighborhood of $C_\alpha$ using the clutching map which, along the meridian of a 1-cell of $C_\alpha$ colored by $x \in X/Q$, evaluates to $x \in X/Q \subseteq P/Q \cong \pi_1({}^{L}G_{\mathrm{ad}})$. 
Let us call the resulting ${}^{L}G_{\mathrm{ad}}$-bundle $P_{\mathrm{tan}}(C_\alpha)$; its isomorphism class is represented by $\mathrm{PD}(\alpha) \in H^2(Y;P/Q)$. 
To this end, we can define a \emph{$C_\alpha$-twisted spin$^{G}$ structure} to be a lift of $P_{\mathrm{tan}}(C_\alpha)$ to a ${}^{L}\widehat{G}_{\mathrm{ad}}$-bundle. 
As before, if $C_\alpha$ and $C'_\alpha$ are two $1$-cycles representing the same class $\alpha \in H_1(Y;X/Q)$, any $2$-chain $D$ with coefficients in $X/Q$ such that $\partial D = C_\alpha - C'_\alpha$ induces an isomorphism
\[
\Phi_D : \mathrm{Spin}^G(Y, C_\alpha) \cong \mathrm{Spin}^G(Y, C'_\alpha),
\]
which is compatible with composition of the 2-chains. 
If $D_1$ and $D_2$ are both $2$-chains from $C_\alpha$ to $C'_\alpha$, then 
\[
\Phi_{D_1} - \Phi_{D_2} = \beta ([D_1 - D_2]),
\]
where $\beta : H_2(Y;X/Q) \rightarrow H_1(Y;Q)$ is the Bockstein.

\section{Trivalent vertices for Vermas}\label{sec:trivalent-vertices}
\begin{figure}
\centering
\[
\vcenter{\hbox{
\begin{tikzpicture}[
    mid/.style={
        postaction={
            decorate,
            decoration={
                markings,
                mark=at position 0.5 with {\arrow{>}} 
            }
        }
    },
    midback/.style={
        postaction={
            decorate,
            decoration={
                markings,
                mark=at position 0.5 with {\arrow{<}} 
            }
        }
    },
]
\draw[blue, thick, mid] (0, -1) -- (0, 0);
\draw[blue, thick, mid] (0, 0) -- ({sqrt(3)/2}, 1/2);
\draw[blue, thick, mid] (0, 0) -- (-{sqrt(3)/2}, 1/2);
\node[blue, below] at (0, -1){$a$};
\node[blue, left] at ({-sqrt(3)/2}, 1/2){$b$};
\node[blue, right] at ({sqrt(3)/2}, 1/2){$c$};
\node[blue, right] at (0, -1/2){$V_{\infty}((xy)_k)$};
\node[blue, above] at ({-sqrt(3)/2}, 1/2){$V_{\infty}(x)$};
\node[blue, above] at ({sqrt(3)/2}, 1/2){$V_{\infty}(y)$};
\end{tikzpicture}
}}
\qquad
\vcenter{\hbox{
\begin{tikzpicture}[
    mid/.style={
        postaction={
            decorate,
            decoration={
                markings,
                mark=at position 0.5 with {\arrow{>}} 
            }
        }
    },
    midback/.style={
        postaction={
            decorate,
            decoration={
                markings,
                mark=at position 0.5 with {\arrow{<}} 
            }
        }
    },
]
\draw[blue, thick, mid] (0, 0) -- (0, 1);
\draw[blue, thick, midback] (0, 0) -- ({sqrt(3)/2}, -1/2);
\draw[blue, thick, midback] (0, 0) -- (-{sqrt(3)/2}, -1/2);
\node[blue, above] at (0, 1){$c$};
\node[blue, left] at ({-sqrt(3)/2}, -1/2){$a$};
\node[blue, right] at ({sqrt(3)/2}, -1/2){$b$};
\node[blue, right] at (0, 1/2){$V_{\infty}((xy)_k)$};
\node[blue, below] at ({-sqrt(3)/2}, -1/2){$V_{\infty}(x)$};
\node[blue, below] at ({sqrt(3)/2}, -1/2){$V_{\infty}(y)$};
\end{tikzpicture}
}}
\]
\caption{$Y(\protect\substack{x,y\\(xy)_k})_{a}^{b,c}$ and $\dualY (\protect\substack{(xy)_k\\x,y})_{a,b}^{c}$}
\label{fig:trivalent_junctions}
\end{figure}
Here we summarize the trivalent vertices (i.e. fusion rules, a.k.a. $q$-$3j$ symbols) for Verma modules that will be used in Section \ref{sec:surgery}. 
For any integer $k$, let
\[
(xy)_k := q^{-1-2k}xy.
\]
Then, we have the following direct sum decomposition of tensor product of highest weight Verma modules:
\[
V_{\infty}(x) \otimes V_{\infty}(y) \cong \bigoplus_{k\geq 0} V_\infty((xy)_k).
\]
The inclusion of each direct summand $V_\infty((xy)_k)$ and the projection to it define maps
\[
\iota_k : V_{\infty}((xy)_k) \rightarrow V_{\infty}(x)\otimes V_{\infty}(y)
\quad\text{and}\quad
p_k : V_{\infty}(x)\otimes V_{\infty}(y) \rightarrow V_{\infty}((xy)_k)
\]
satisfying $\sum_{k\geq 0} \iota_k \circ p_k = \mathrm{Id}$. 
Such maps were studied in \cite[Sec. 3-4]{Solovyev}. 
We would like to draw these maps diagrammatically as trivalent junctions (see Figure \ref{fig:trivalent_junctions}), and for that purpose, it turns out to be more natural to use slightly rescaled maps
\begin{align*}
&Y(\substack{x,y\\(xy)_k}) : V_{\infty}((xy)_k) \rightarrow V_{\infty}(x)\otimes V_{\infty}(y) \text{ and} \\
&\dualY(\substack{(xy)_k\\ x,y}) : V_{\infty}(x)\otimes V_{\infty}(y) \rightarrow V_{\infty}((xy)_k),
\end{align*}
so that they satisfy
\[
\dualY(\substack{(xy)_k\\ x,y})_{a,b}^{c} := q^{\frac{b}{2}+\frac{1}{4}}y^{-\frac{1}{4}}Y(\substack{(xy)_k,y^{-1}\\x})_{a}^{c,-1-b} =  -q^{-\frac{a}{2}-\frac{1}{4}}x^{\frac{1}{4}}Y(\substack{x^{-1},(xy)_k\\y})_{b}^{-1-a,c}, 
\]
which can be interpreted graphically as in Figure \ref{fig:turning-Y-junction}. 
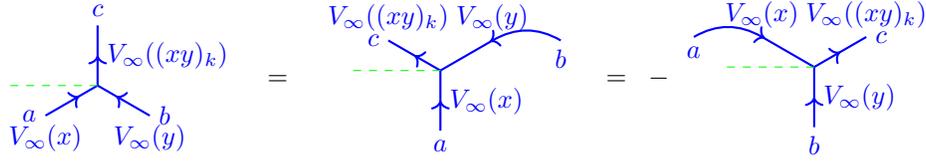
\begin{figure}
\centering
\[
\vcenter{\hbox{
\begin{tikzpicture}[scale=0.8, 
    mid/.style={
        postaction={
            decorate,
            decoration={
                markings,
                mark=at position 0.5 with {\arrow{>}} 
            }
        }
    },
    midback/.style={
        postaction={
            decorate,
            decoration={
                markings,
                mark=at position 0.5 with {\arrow{<}} 
            }
        }
    },
]
\draw[blue, thick, mid] (0, 0) -- (0, 1);
\draw[blue, thick, midback] (0, 0) -- ({sqrt(3)/2}, -1/2);
\draw[blue, thick, midback] (0, 0) -- (-{sqrt(3)/2}, -1/2);
\node[blue, above] at (0, 1){$c$};
\node[blue, left] at ({-sqrt(3)/2}, -1/2){$a$};
\node[blue, right] at ({sqrt(3)/2}, -1/2){$b$};
\node[blue, right] at (0, 1/2){$V_{\infty}((xy)_k)$};
\node[blue, below] at ({-sqrt(3)/2}, -1/2){$V_{\infty}(x)$};
\node[blue, below] at ({sqrt(3)/2}, -1/2){$V_{\infty}(y)$}; 
\draw[green, dashed] (0, 0) -- (-3/2, 0);
\end{tikzpicture}
}}
\;\;=\;\;
\vcenter{\hbox{
\begin{tikzpicture}[scale=0.8, 
    mid/.style={
        postaction={
            decorate,
            decoration={
                markings,
                mark=at position 0.5 with {\arrow{>}} 
            }
        }
    },
    midback/.style={
        postaction={
            decorate,
            decoration={
                markings,
                mark=at position 0.5 with {\arrow{<}} 
            }
        }
    },
]
\draw[blue, thick, mid] (0, -1) -- (0, 0);
\draw[blue, thick, midback] (0, 0) -- ({sqrt(3)/2}, 1/2) to[out=30, in=150] ({2}, 1/2);
\draw[blue, thick, mid] (0, 0) -- (-{sqrt(3)/2}, 1/2);
\node[blue, below] at (0, -1){$a$};
\node[blue, left] at ({-sqrt(3)/2}, 1/2){$c$};
\node[blue, below] at ({2}, 1/2){$b$};
\node[blue, right] at (0, -1/2){$V_{\infty}(x)$};
\node[blue, above] at ({-sqrt(3)/2}, 1/2){$V_{\infty}((xy)_k)$};
\node[blue, above] at ({sqrt(3)/2}, 1/2){$V_{\infty}(y)$};
\draw[green, dashed] (0, 0) -- (-3/2, 0);
\end{tikzpicture}
}}
\;\;=\;\;
-
\vcenter{\hbox{
\begin{tikzpicture}[scale=0.8, 
    mid/.style={
        postaction={
            decorate,
            decoration={
                markings,
                mark=at position 0.5 with {\arrow{>}} 
            }
        }
    },
    midback/.style={
        postaction={
            decorate,
            decoration={
                markings,
                mark=at position 0.5 with {\arrow{<}} 
            }
        }
    },
]
\draw[blue, thick, mid] (0, -1) -- (0, 0);
\draw[blue, thick, mid] (0, 0) -- ({sqrt(3)/2}, 1/2);
\draw[blue, thick, midback] (0, 0) -- (-{sqrt(3)/2}, 1/2) to[out=150, in=30] ({-2}, 1/2);
\node[blue, below] at (0, -1){$b$};
\node[blue, below] at ({-2}, 1/2){$a$};
\node[blue, right] at ({sqrt(3)/2}, 1/2){$c$};
\node[blue, right] at (0, -1/2){$V_{\infty}(y)$};
\node[blue, above] at ({-sqrt(3)/2}, 1/2){$V_{\infty}(x)$};
\node[blue, above] at ({sqrt(3)/2}, 1/2){$V_{\infty}((xy)_k)$};
\draw[green, dashed] (0, 0) -- (-3/2, 0);
\end{tikzpicture}
}}
\]
\caption{$\dualY$-junction by turning the $Y$-junction. An extra $-$-sign is introduced when a strand of a junction is turned across the negative $x$-axis (the green dashed line).}
\label{fig:turning-Y-junction}
\end{figure}
Without the rescaling, there would be some extra factors; see \cite[Eqns.~3.10-3.11]{KirillovReshetikhin} for the finite-dimensional analog. 
Compared to the Racah--Fock formula in \cite[Thm. 4.1]{Solovyev}, 
\[
Y(\substack{x,y\\(xy)_k})_{a}^{b,c} 
= 
N_k\,
q^{-\frac{(\lambda_1 - b)(\lambda_2 - c)}{2}}
\begin{bmatrix}
\lambda_1 & \lambda_2 & \lambda \\
\lambda_1 - b & \lambda_2 - c & \lambda - a
\end{bmatrix}^{\psi}_{q_{\mathrm{Sol}}},
\]
where\footnote{See \cite[p.3]{Solovyev} for the definition of $\Delta(\lambda_1, \lambda_2, \lambda)$.}
\[
N_k := 
(-1)^{\frac{k}{2}} q^{\frac{\lambda_1 \lambda_2}{2} +\frac{k\lambda}{2} + \lambda + 2k + 1} \frac{[k]!}{\Delta(\lambda_1, \lambda_2, \lambda) ((xy)_k;q)_{k+1}},
\]
and
\[
x = q^{2\lambda_1 + 1},\quad y = q^{2\lambda_2 + 1},\quad (xy)_k = q^{2\lambda + 1},\quad \lambda = \lambda_1 + \lambda_2 - k,\quad q_{\mathrm{Sol}} = q^{\frac12}. 
\]
Explicitly, the $Y$-junction weight is given by
\begin{align*}
Y(\substack{x,y\\(xy)_k})_{a}^{b,c} &= 
\delta_{b+c,a+k} \sum_{\substack{m_b, m_c \in \mathbb{Z},\\ m_b+m_c = k}} (-1)^{m_c} q^{\star_q}
x^{\star_x} y^{\star_y}
\qbin{a}{c-m_c}_q\qbin{k}{m_c}_q \\
&\quad\times
\frac{(q^{-b}x; q)_{b-m_b}(q^{-c}y; q)_{c-m_c}}{(q^{-a} (xy)_{k}; q)_{b+c+1}},
\end{align*}
where
\begin{align*}
\star_q &= \qty(-\frac{(a+k)(2k+1)}{2} + \frac{c(2m_b+1) + m_b(m_b+1)}{4} - \frac{b(2m_c+1) + m_c(m_c+1)}{4})\\
&\quad - \frac{(b-m_b)(c-m_c)}{2} - \frac{m_b m_c}{2} \\
&\quad -\frac{(-b-m_b-1)(b-m_b)}{4}-\frac{(-c-m_c-1)(c-m_c)}{4} + \frac{(-a+k)(b+c+1)}{4}\\
&= 
\frac{(c+m_b)(c+m_b+1)-(k+c+1)(b+c)-bc}{2}-\frac{a}{4}-\frac{k(2k+1)}{4}, \\
\star_x &= \qty(-\frac{a-b+2m_b}{4}) - \frac{b-m_b}{2} + \frac{b+c+1}{2}\\
&= \frac{-a+b+2c}{4}+\frac{1}{2}, \\
\star_y &= \qty(\frac{a-c+2m_c}{4}) - \frac{c-m_c}{2} + \frac{b+c+1}{2} \\
&= \frac{a+2b-c}{4}+m_c+\frac{1}{2}.
\end{align*}
The sum above is a finite sum, since the range of summation can be restricted to $0\leq m_b \leq b$ when $b\geq 0$ and to $0\leq m_c \leq c$ when $c\geq 0$. 

Below, we summarize the main properties of these trivalent vertices. 
\begin{enumerate}
\item Rescaled version of $\sum_{k\geq 0} \iota_k \circ p_k = \mathrm{Id}$:
\begin{equation}\label{eq:fusion}
\sum_{0\leq k\leq w}-((xy)_k^{\frac12}-(xy)_k^{-\frac12})\dualY(\protect\substack{(xy)_k\\x,y})_{a,w-a}^{w-k} Y(\protect\substack{x,y\\(xy)_k})_{w-k}^{a',w-a'} = \delta_{a,a'}.
\end{equation}
For its graphical interpretation, see Figure \ref{fig:fusion}. 
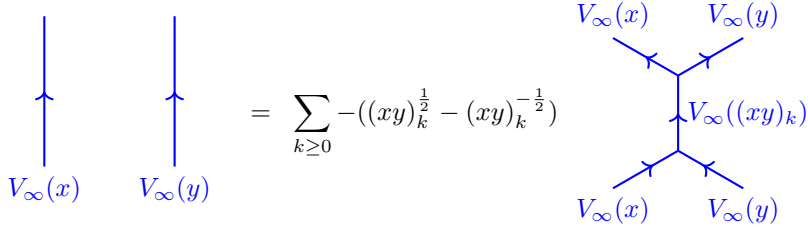
\begin{figure}
\centering
\[
\vcenter{\hbox{
\begin{tikzpicture}[
    mid/.style={
        postaction={
            decorate,
            decoration={
                markings,
                mark=at position 0.5 with {\arrow{>}} 
            }
        }
    },
    midback/.style={
        postaction={
            decorate,
            decoration={
                markings,
                mark=at position 0.5 with {\arrow{<}} 
            }
        }
    },
]
\draw[blue, thick, mid] ({-sqrt(3)/2}, -1/2) -- ({-sqrt(3)/2}, 3/2);
\draw[blue, thick, mid] ({sqrt(3)/2}, -1/2) -- ({sqrt(3)/2}, 3/2);
\node[blue, below] at ({-sqrt(3)/2}, -1/2){$V_{\infty}(x)$};
\node[blue, below] at ({sqrt(3)/2}, -1/2){$V_{\infty}(y)$}; 
\end{tikzpicture}
}}
\;\;=\;\;
\sum_{k\geq 0}
-((xy)_k^{\frac12}-(xy)_k^{-\frac12})
\vcenter{\hbox{
\begin{tikzpicture}[
    mid/.style={
        postaction={
            decorate,
            decoration={
                markings,
                mark=at position 0.5 with {\arrow{>}} 
            }
        }
    },
    midback/.style={
        postaction={
            decorate,
            decoration={
                markings,
                mark=at position 0.5 with {\arrow{<}} 
            }
        }
    },
]
\draw[blue, thick, mid] (0, 0) -- (0, 1);
\draw[blue, thick, midback] (0, 0) -- ({sqrt(3)/2}, -1/2);
\draw[blue, thick, midback] (0, 0) -- (-{sqrt(3)/2}, -1/2);
\draw[blue, thick, mid] (0, 1) -- ({sqrt(3)/2}, 3/2);
\draw[blue, thick, mid] (0, 1) -- (-{sqrt(3)/2}, 3/2);
\node[blue, right] at (0, 1/2){$V_{\infty}((xy)_k)$};
\node[blue, below] at ({-sqrt(3)/2}, -1/2){$V_{\infty}(x)$};
\node[blue, below] at ({sqrt(3)/2}, -1/2){$V_{\infty}(y)$}; 
\node[blue, above] at ({-sqrt(3)/2}, 3/2){$V_{\infty}(x)$};
\node[blue, above] at ({sqrt(3)/2}, 3/2){$V_{\infty}(y)$}; 
\end{tikzpicture}
}}
\]
\caption{Fusion of two strands}
\label{fig:fusion}
\end{figure}
\item 3-fold rotation symmetry: 
\begin{align*}
Y(\substack{x,y\\(xy)_k})_{-1-a}^{-1-b,c} &= -q^{\frac{b-a}{2}}x^{\frac14} (xy)_k^{-\frac14}Y(\substack{y,(xy)_k^{-1}\\x^{-1}})_{b}^{c,a},\label{eq:Y3fold}\\
Y(\substack{x,y\\(xy)_k})_{-1-a}^{b,-1-c} &= -q^{\frac{a-c}{2}} y^{-\frac14} (xy)_k^{\frac14}Y(\substack{(xy)_k^{-1},x\\y^{-1}})_{c}^{a,b}. 
\end{align*}
See Figure \ref{fig:3-fold-rotation-symmetry} for the graphical interpretation of the first identity; it's similar for the second identity. 
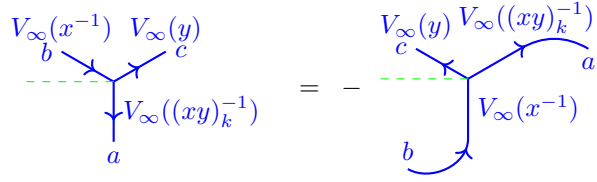
\begin{figure}
\centering
\[
\vcenter{\hbox{
\begin{tikzpicture}[scale=0.8, 
    mid/.style={
        postaction={
            decorate,
            decoration={
                markings,
                mark=at position 0.5 with {\arrow{>}} 
            }
        }
    },
    midback/.style={
        postaction={
            decorate,
            decoration={
                markings,
                mark=at position 0.5 with {\arrow{<}} 
            }
        }
    },
]
\draw[blue, thick, midback] (0, -1) -- (0, 0);
\draw[blue, thick, mid] (0, 0) -- ({sqrt(3)/2}, 1/2);
\draw[blue, thick, midback] (0, 0) -- (-{sqrt(3)/2}, 1/2);
\node[blue, below] at (0, -1){$a$};
\node[blue, left] at ({-sqrt(3)/2}, 1/2){$b$};
\node[blue, right] at ({sqrt(3)/2}, 1/2){$c$};
\node[blue, right] at (0, -1/2){$V_{\infty}((xy)_k^{-1})$};
\node[blue, above] at ({-sqrt(3)/2}, 1/2){$V_{\infty}(x^{-1})$};
\node[blue, above] at ({sqrt(3)/2}, 1/2){$V_{\infty}(y)$};
\draw[green, dashed] (0, 0) -- (-3/2, 0);
\end{tikzpicture}
}}
\;\;=\;\;
-
\vcenter{\hbox{
\begin{tikzpicture}[scale=0.8, 
    mid/.style={
        postaction={
            decorate,
            decoration={
                markings,
                mark=at position 0.5 with {\arrow{>}} 
            }
        }
    },
    midback/.style={
        postaction={
            decorate,
            decoration={
                markings,
                mark=at position 0.5 with {\arrow{<}} 
            }
        }
    },
]
\draw[blue, thick, mid] (-1, -3/2) to[out=-30, in=-90] (0, -1) -- (0, 0);
\draw[blue, thick, mid] (0, 0) -- ({sqrt(3)/2}, 1/2) to[out=30, in=150] ({2}, 1/2);
\draw[blue, thick, mid] (0, 0) -- (-{sqrt(3)/2}, 1/2);
\node[blue, above] at (-1, -3/2){$b$};
\node[blue, left] at ({-sqrt(3)/2}, 1/2){$c$};
\node[blue, below] at ({2}, 1/2){$a$};
\node[blue, right] at (0, -1/2){$V_{\infty}(x^{-1})$};
\node[blue, above] at ({-sqrt(3)/2}, 1/2){$V_{\infty}(y)$};
\node[blue, above] at ({sqrt(3)/2+0.1}, 1/2+0.1){$V_{\infty}((xy)_k^{-1})$};
\draw[green, dashed] (0, 0) -- (-3/2, 0);
\end{tikzpicture}
}}
\]
\caption{$3$-fold rotation symmetry of $Y$}
\label{fig:3-fold-rotation-symmetry}
\end{figure}
\item The trivalent junctions solve the eigenvalue equation for the $R$-matrices: 
\begin{equation}\label{eq:eigen-value-eq}
\sum_{b,c\geq 0}R(x,y)_{b,c}^{b',c'}Y(\substack{x,y\\(xy)_k})_{a}^{b,c} = (-1)^k q^{\frac{k(k+1)}{2}+\frac14} (xy)^{-\frac{1+2k}{4}} \;Y(\substack{y,x\\(xy)_k})_{a}^{b',c'}. 
\end{equation}
See Figure \ref{fig:R-matrix-eigen-equation}. 
\begin{figure}
\centering
\[
\vcenter{\hbox{
\begin{tikzpicture}[scale=0.8, 
    mid/.style={
        postaction={
            decorate,
            decoration={
                markings,
                mark=at position 0.5 with {\arrow{>}} 
            }
        }
    },
    midback/.style={
        postaction={
            decorate,
            decoration={
                markings,
                mark=at position 0.5 with {\arrow{<}} 
            }
        }
    },
]
\draw[blue, thick, mid] (0, 0) to[out=30, in=-90] ({3/4}, 1/2) to[out=90, in=-90] ({-3/4}, 3/2);
\draw[white, line width=5] (0, 0) to[out=150, in=-90] (-{3/4}, 1/2) to[out=90, in=-90] ({3/4}, 3/2);
\draw[blue, thick, mid] (0, 0) to[out=150, in=-90] (-{3/4}, 1/2) to[out=90, in=-90] ({3/4}, 3/2);
\draw[blue, thick, mid] (0, -1) -- (0, 0);
\node[blue, right] at (0, -1/2){$V_{\infty}((xy)_k)$};
\node[blue, left] at ({-3/4}, 1/2){$V_{\infty}(x)$};
\node[blue, right] at ({3/4}, 1/2){$V_{\infty}(y)$};
\end{tikzpicture}
}}
\;\;=\;\;
(-1)^k q^{\frac{k(k+1)}{2}+\frac14} (xy)^{-\frac{1+2k}{4}}
\vcenter{\hbox{
\begin{tikzpicture}[scale=0.8, 
    mid/.style={
        postaction={
            decorate,
            decoration={
                markings,
                mark=at position 0.5 with {\arrow{>}} 
            }
        }
    },
    midback/.style={
        postaction={
            decorate,
            decoration={
                markings,
                mark=at position 0.5 with {\arrow{<}} 
            }
        }
    },
]
\draw[blue, thick, mid] (0, -1) -- (0, 0);
\draw[blue, thick, mid] (0, 0) -- ({sqrt(3)/2}, 1/2);
\draw[blue, thick, mid] (0, 0) -- (-{sqrt(3)/2}, 1/2);
\node[blue, right] at (0, -1/2){$V_{\infty}((xy)_k)$};
\node[blue, above] at ({-sqrt(3)/2}, 1/2){$V_{\infty}(y)$};
\node[blue, above] at ({sqrt(3)/2}, 1/2){$V_{\infty}(x)$};
\end{tikzpicture}
}}
\]
\caption{Eigenvalue equation for the $R$-matrices}
\label{fig:R-matrix-eigen-equation}
\end{figure}
\item Sliding a trivalent junction over (or under) another strand:
\begin{align}\label{eq:sliding}
&\sum_{m\geq 0} \qty(x^{-\frac{1+2k}{4}} R(q^{-1-2k}yz,x)_{a,b}^{c,m})
Y(\substack{y,z\\(yz)_k})_{m}^{d,e} \\
&= 
\sum_{m,n,l\geq 0}
Y(\substack{y,z\\(yz)_k})_{a}^{m,n}
R(z,x)_{n,b}^{l,e} R(y,x)_{m,l}^{c,d} 
. \nonumber
\end{align}
See Figure \ref{fig:sliding-a-junction}.\footnote{
Recall from Table \ref{tab:R-matrices} that, for the crossings among Verma modules $V_\infty(x_1)$ and $V_\infty(x_2)$ where $x_1, x_2$ are generic (so that $x_2 \not\in q^{\mathbb{Z}} x_1$), we omit the exponential prefactor $e^{\frac{\log x_1 \log x_2}{4\log q}}$ from the $R$-matrix. 
More generally, we do the same for crossings among $V_\infty(q^{n_1} x_1)$ and $V_\infty(q^{n_2} x_2)$, so the $R$-matrix associated to their crossing is
\[
q^{\frac{n_1 n_2}{4}} x_1^{\frac{n_2}{4}} x_2^{\frac{n_1}{4}} R(q^{n_1} x_1, q^{n_2} x_2,q), 
\]
and hence the extra factor $x^{-\frac{1+2k}{4}}$ in \eqref{eq:sliding}. 
For crossings among $V_\infty(q^{n_1} x)$ and $V_\infty(q^{n_2} x)$ with $x$ generic, we omit the exponential prefactor $e^{\frac{(\log x)^2 - \log q}{4\log q}}$, so the $R$-matrix associated to their crossing is
\[
q^{\frac{n_1 n_2+1}{4}} x^{\frac{n_1+n_2}{4}} R(q^{n_1} x, q^{n_2} x,q).
\]
} 
\begin{figure}
\centering
\[
\vcenter{\hbox{
\begin{tikzpicture}[scale=0.8, 
    mid/.style={
        postaction={
            decorate,
            decoration={
                markings,
                mark=at position 0.5 with {\arrow{>}} 
            }
        }
    },
    midback/.style={
        postaction={
            decorate,
            decoration={
                markings,
                mark=at position 0.5 with {\arrow{<}} 
            }
        }
    },
]
\draw[blue, thick, mid] (3/2, -1) to[out=90, in=-90] (-9/4, 1/2); 
\draw[blue, thick] (-9/4, 1/2) -- (-9/4, 1);
\draw[white, line width=5] (0, -1) -- (0, 0);
\draw[blue, thick, mid] (0, -1) -- (0, 0);
\draw[blue, thick, mid] (0, 0) to[out=30, in=-90] ({3/4}, 1);
\draw[blue, thick, mid] (0, 0) to[out=150, in=-90] (-{3/4}, 1);
\node[blue, below] at (0, -1){$V_{\infty}((yz)_k)$};
\node[blue, above] at ({-9/4}, 1){$V_{\infty}(x)$};
\node[blue, above] at ({-3/4}, 1){$V_{\infty}(y)$};
\node[blue, above] at ({3/4}, 1){$V_{\infty}(z)$};
\end{tikzpicture}
}}
\;\;=\;\;
\vcenter{\hbox{
\begin{tikzpicture}[scale=0.8, 
    mid/.style={
        postaction={
            decorate,
            decoration={
                markings,
                mark=at position 0.5 with {\arrow{>}} 
            }
        }
    },
    midback/.style={
        postaction={
            decorate,
            decoration={
                markings,
                mark=at position 0.5 with {\arrow{<}} 
            }
        }
    },
]
\draw[blue, thick, mid] (3/2, -1) -- (3/2, -1/2); 
\draw[blue, thick] (3/2, -1/2) to[out=90, in=-90] (-9/4, 1); 
\draw[blue, thick, mid] (0, -1) -- (0, 0); 
\draw[white, line width=5] (0, 0) to[out=30, in=-90] ({3/4}, 1); 
\draw[white, line width=5] (0, 0) to[out=150, in=-90] (-{3/4}, 1); 
\draw[blue, thick, mid] (0, 0) to[out=30, in=-90] ({3/4}, 1); 
\draw[blue, thick, mid] (0, 0) to[out=150, in=-90] (-{3/4}, 1); 
\node[blue, below] at (0, -1){$V_{\infty}((yz)_k)$};
\node[blue, above] at ({-9/4}, 1){$V_{\infty}(x)$};
\node[blue, above] at ({-3/4}, 1){$V_{\infty}(y)$};
\node[blue, above] at ({3/4}, 1){$V_{\infty}(z)$};
\end{tikzpicture}
}}
\]
\caption{Sliding a trivalent junction over another strand}
\label{fig:sliding-a-junction}
\end{figure}
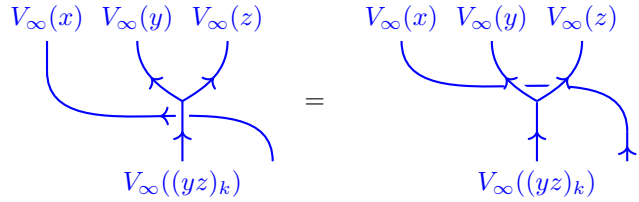
\end{enumerate}

The following ``inverted'' version of sliding relation is used in the proof of Theorem \ref{mainthm:surgery}: 
\begin{lem}\label{lem:inverted-sliding}
One can slide a trivalent junction involving $V_N$ over any inverted crossing.  That is, the sliding relation \eqref{eq:sliding} holds even after inverting the relevant indices. 
For instance, 
\begin{align*}
&\sum_{m\geq 0} 
R(x,y)_{-1-a,b}^{-1-c,m} 
Y(\substack{q^N, q^{-N+1+2k}x\\x})_{m}^{d,e}
\\
&=
\sum_{m,n,l \geq 0} 
Y(\substack{q^N, q^{-N+1+2k}x\\x})_{-1-a}^{m, -1-n} 
\qty(y^{\frac{-N+1+2k}{4}} R(q^{-N+1+2k}x,y)_{-1-n,b}^{-1-l,e})
\\
&\qquad\qquad \times
\qty(y^{\frac{N}{4}} R(q^N,y)_{m,-1-l}^{-1-c,d}),
\end{align*}
which can be graphically interpreted as Figure \ref{fig:sliding-across-inverted-crossing}.  
\begin{figure}
\centering
\[
\vcenter{\hbox{
\begin{tikzpicture}[scale=0.8, 
    mid/.style={
        postaction={
            decorate,
            decoration={
                markings,
                mark=at position 0.5 with {\arrow{>}} 
            }
        }
    },
    midback/.style={
        postaction={
            decorate,
            decoration={
                markings,
                mark=at position 0.5 with {\arrow{<}} 
            }
        }
    },
]
\draw[blue, thick, mid] (1, -1) -- (0, 0);
\draw[blue, thick, midback] (0, 0) -- (-1, 1);
\draw[white, line width=5] (-1, -1) -- (1, 1);
\draw[blue, thick, midback] (-1, -1) -- (0, 0);
\draw[blue, thick, mid] (0, 0) -- (1, 1);
\draw[red] (0.7, 1.3) to[out=-135, in=90] (0.3, 0.3);
\node[blue, below] at (1, -1){$V_\infty(y)$};
\node[blue, below] at (-1, -1){$V_\infty(x^{-1})$};
\node[blue, above] at (-1, 1){$V_\infty(y^{-1})$};
\node[blue, above right] at (1, 1){$V_\infty(q^{-N+1+2k}x)$};
\node[red, above] at (0.7, 1.3){$V_N$};
\end{tikzpicture}
}}
\;\;=\;\;
\vcenter{\hbox{
\begin{tikzpicture}[scale=0.8, 
    mid/.style={
        postaction={
            decorate,
            decoration={
                markings,
                mark=at position 0.5 with {\arrow{>}} 
            }
        }
    },
    midback/.style={
        postaction={
            decorate,
            decoration={
                markings,
                mark=at position 0.5 with {\arrow{<}} 
            }
        }
    },
]
\draw[blue, thick, mid] (1, -1) -- (0, 0);
\draw[blue, thick, midback] (0, 0) -- (-1, 1);
\draw[white, line width=5] (-1, -1) -- (1, 1);
\draw[white, line width=5] (0.7, 1.3) to[out=-135, in=90] (-0.3, -0.3);
\draw[blue, thick, midback] (-1, -1) -- (0, 0);
\draw[blue, thick, mid] (0, 0) -- (1, 1);
\draw[red] (0.7, 1.3) to[out=-135, in=90] (-0.3, -0.3);
\node[blue, below] at (1, -1){$V_\infty(y)$};
\node[blue, below] at (-1, -1){$V_\infty(x^{-1})$};
\node[blue, above] at (-1, 1){$V_\infty(y^{-1})$};
\node[blue, above right] at (1, 1){$V_\infty(q^{-N+1+2k}x)$};
\node[red, above] at (0.7, 1.3){$V_N$};
\end{tikzpicture}
}}
\]
\caption{Sliding a junction across an inverted crossing}
\label{fig:sliding-across-inverted-crossing}
\end{figure}
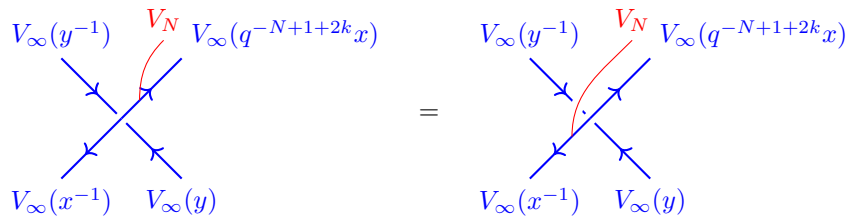
\end{lem}

\begin{proof}
Recall the basic fact that if $f(X_1, \cdots, X_r)$ and $g(X_1, \cdots, X_r)$ are rational functions over $\mathbb{Q}(q)$ such that $f(q^{n_1}, \cdots, q^{n_r}) = g(q^{n_1}, \cdots, q^{n_r})$ for all tuples $(n_1, \cdots, n_r) \in \mathbb{Z}^r$ with $n_1, \cdots, n_r \geq M$ for some big enough $M$, then $f = g$. 
This is because, $f-g$, after clearing the denominator, is a Laurent polynomial with infinitely many zeros, thus vanishing identically. 

The proof of the lemma is a corollary of this fact. 
We will only consider the case shown in Figure \ref{fig:sliding-across-inverted-crossing} as the proof in the other cases is analogous. 
Fix $b, d, e, k \in \mathbb{Z}_{\geq 0}$. 
From the usual sliding relation \eqref{eq:sliding}, we have
\begin{align*}
&
R(x,y)_{a,b}^{a+b+k-d-e,d+e-k} 
Y(\substack{q^N, q^{-N+1+2k}x\\x})_{d+e-k}^{d,e}
\\
&=
\sum_{0\leq m \leq N-1} 
Y(\substack{q^N, q^{-N+1+2k}x\\x})_{a}^{m, a+k-m} 
\qty(y^{\frac{-N+1+2k}{4}} R(q^{-N+1+2k}x,y)_{a+k-m,b}^{a+b+k-m-e,e})
\\
&\qquad\qquad \times
\qty(y^{\frac{N}{4}} R(q^N,y)_{m,a+b+k-m-e}^{a+b+k-e-d,d}),
\end{align*}
for all large enough $a$. 
After dividing both sides by a common factor 
\[
x^{-\frac{a}{4}} y^{\frac{a}{4}} (q^{b+1} y^{-1}; q)_{a+k-d-e},
\]
one can check that both sides are rational functions in $q^a$. 
Therefore, it follows that the equality holds for all $a \in \mathbb{Z}$, as desired. 
\end{proof}

\bibliographystyle{amsplain}
\bibliography{ref}

\end{document}